\documentclass[11pt]{article}
\usepackage[margin=1in]{geometry}
\usepackage[numbers,square,sort]{natbib}

\usepackage{graphicx}
\usepackage{multirow}
\usepackage{amsmath,amssymb,amsfonts, amsthm}
\usepackage{mathrsfs}
\usepackage{xcolor}
\usepackage{textcomp}
\usepackage{tabularx}
\usepackage{booktabs}
\usepackage{algorithm}
\usepackage{algorithmicx}
\usepackage{algpseudocode}
\usepackage{listings}
\usepackage{mathtools}
\usepackage{enumitem}
\usepackage{bm}
\usepackage[labelsep=period]{caption}
\usepackage{subcaption}  
\usepackage{nicefrac}
\usepackage{xfrac}
\usepackage{threeparttable}
\usepackage{acro}
\usepackage{authblk}
\usepackage{hyperref}
\usepackage{xurl}
\usepackage{cleveref}
\usepackage[T1]{fontenc}

\DeclareAcronym{cco}{
  short = CCO ,
  long  = Chance-Constrained Optimization ,
  short-plural = s ,
}
\DeclareAcronym{saa}{
  short = SAA ,
  long  = Sample Average Approximation ,
  short-plural = s ,
}
\DeclareAcronym{sa}{
  short = SA ,
  long  = Scenario Approach ,
  short-plural = s ,
}
\DeclareAcronym{ldp}{
  short = LDP ,
  long  = Large Deviation Principle ,
  short-plural = s ,
}
\DeclareAcronym{iid}{
  short = i.i.d. ,
  long  = independent identically distributed ,
  short-plural = s ,
}
\DeclareAcronym{is}{
  short = IS ,
  long  = Importance Sampling ,
  short-plural = s ,
}

\newtheorem{theorem}{Theorem}[section]
\newtheorem{proposition}[theorem]{Proposition}
\newtheorem{lemma}[theorem]{Lemma}
\newtheorem{corollary}[theorem]{Corollary}

\theoremstyle{definition}
\newtheorem{definition}[theorem]{Definition}
\newtheorem{example}[theorem]{Example}
\newtheorem{assumption}[theorem]{Assumption}

\theoremstyle{remark}
\newtheorem{remark}[theorem]{Remark}

\newcommand{\vzero}{\bm{0}}
\newcommand{\vone}{\bm{1}}
\newcommand{\va}{\bm{a}}
\newcommand{\vb}{\bm{b}}
\newcommand{\vc}{\bm{c}}
\newcommand{\vd}{\bm{d}}

\newcommand{\vw}{\bm{w}}
\newcommand{\vx}{\bm{x}}
\newcommand{\vy}{\bm{y}}
\newcommand{\vz}{\bm{z}}
\newcommand{\vxi}{\bm{\xi}}
\newcommand{\vmu}{\bm{\mu}}
\newcommand{\vnu}{\bm{\nu}}
\newcommand{\vtheta}{\bm{\theta}}
\newcommand{\supportset}{\Xi} 
\newcommand{\minimize}{\mathop{\text{minimize}}\limits}
\newcommand{\maximize}{\mathop{\text{maximize}}\limits}
\newcommand{\RV}{\mathscr{RV}}
\newcommand{\MRV}{\mathscr{MRV}}

\title{Decision-Scaled Scenario Approach for Rare Chance-Constrained Optimization\thanks{Jaeseok Choi and Anirudh Subramanyam acknowledge support by the U.S. National Science Foundation under Grant DMS-2229408. Constantino Lagoa acknowledges support by the U.S. National Science Foundation under Grant ECCS-2317272.}}

\author[1]{Jaeseok Choi}
\author[2]{Anand Deo}
\author[3]{Constantino Lagoa}
\author[1]{Anirudh Subramanyam}

\affil[1]{Department of Industrial and Manufacturing Engineering, Pennsylvania State University \authorcr
\texttt{\{jxc6747, subramanyam\}@psu.edu}}
\affil[2]{Decision Sciences, Indian Institute of Management \authorcr
\texttt{anand.deo@iimb.ac.in}}
\affil[3]{School of Electrical Engineering and Computer Science, Pennsylvania State University \authorcr
\texttt{cml18@psu.edu}}

\date{}

\begin{document}

\maketitle

\begin{abstract}
Chance-constrained optimization is a suitable modeling framework for safety-critical applications where violating constraints is nearly unacceptable. 
The scenario approach is a popular solution method for these problems, 
due to its straightforward implementation and ability to preserve problem structure.
However, in the rare-event regime where constraint violations must be kept extremely unlikely, the scenario approach becomes computationally infeasible due to the excessively large sample sizes it demands.
We address this limitation with a new yet straightforward decision-scaling method that relies exclusively on original data samples and a single scalar hyperparameter that scales the constraints in a way amenable to standard solvers.
Our method leverages large deviation principles 
under mild nonparametric assumptions satisfied by 
commonly used distribution families in practice.
For a broad class of problems satisfying certain practically verifiable structural assumptions, the method achieves a polynomial reduction in sample size requirements compared to the classical scenario approach, while also guaranteeing asymptotic feasibility in the rare-event regime.
Numerical experiments spanning finance and engineering applications show that our decision-scaling method significantly expands the scope of problems that can be solved both efficiently and reliably.

\vspace{1em}
\noindent\textbf{Keywords:} chance-constrained optimization, scenario approach, rare events, large deviation principle \\
\noindent\textbf{Mathematics Subject Classification:} 90C15, 60F10, 65C05
\end{abstract}

\section{Introduction}\label{intro}
Rare but catastrophic events--such as the COVID-19 pandemic and the 2025 Iberian Peninsula blackout--highlight the critical need for decision frameworks that balance safety with efficiency in engineering and risk management~\citep{ivanov2020viability,bajo2025iberian}.
\ac{cco}, introduced by \citet{charnes1958cost}, provides a structured framework for incorporating safety requirements into decision-making. 

We consider the following general \ac{cco} formulation:
\begin{equation}
\label{eqn:cc_basic}
\tag{$\text{CCP}_{\varepsilon}$}
\begin{aligned}
\minimize_{\vx\in \mathcal{X}} &\quad c(\vx)\\
\text{subject to} &\quad \mathbb{P}_{\vxi}\left(\{\vz: g(\vx,\vz) \leq 0\}\right) \geq 1-\varepsilon.
\end{aligned}
\end{equation}
Here, $\vx \in \mathbb{R}^n$ is the decision vector restricted to a deterministic \textit{closed convex} set $\mathcal X \subseteq \mathbb{R}^n$,
$c:\mathbb{R}^n \to \mathbb{R}$ is a \textit{convex} cost function, and $\varepsilon \in (0,1)$ is a prescribed risk tolerance. 
The uncertainty is captured by a random vector
$\vxi\colon\Omega \to \supportset\subseteq\mathbb{R}^m$ defined on a probability space $(\Omega,\mathcal{F},\mathbb{P})$.
Its probability distribution, denoted by $\mathbb{P}_{\vxi}$, is supported on the set $\supportset$\footnote{The support $\supportset$ is the smallest closed set with $\mathbb{P}_{\vxi}(\supportset)=1$.};
we emphasize that our method does not require knowledge of the full distributional form, only certain properties about its tail behavior which we formalize later.
The constraint function $g\colon\mathcal{X}\times \supportset \to \mathbb{R}$ encodes 
performance requirements, where 
$g(\cdot, \vz)$ is \textit{closed} and \textit{convex} for any fixed $\vz$ and $g(\vx, \cdot)$ is Borel measurable and \emph{continuous} for any fixed $\vx$.
The formulation accommodates joint chance constraints of the form
$\mathbb{P}_{\vxi}\left(\{\vz: g_1(\vx,\vz) \leq 0, \; g_2(\vx,\vz) \leq 0, \ldots, g_K(\vx,\vz) \leq 0\} \right) \geq 1-\varepsilon$
by defining
$g(\vx, \vz) = \max_{i=1,\ldots,K}g_i(\vx,\vz)$.
The objective is to
minimize $c(\vx)$ while ensuring %
that the decision $\vx$ satisfies the performance requirement $g(\vx,\vz) \leq 0$ with probability
at least $1-\varepsilon$; equivalently, the probability of constraint violation must not exceed $\varepsilon$.

This paper specifically focuses on the rare-event regime where $\varepsilon$ is exceptionally small, typically below $10^{-3}$.
Such stringent reliability requirements arise naturally in many safety-critical applications. For example, structural engineering demands failure probabilities below
$10^{-4}$ for critical infrastructures~\citep{duckett2005risk};
power grids target maximum loss-of-load probabilities of roughly $10^{-4}$~\citep{milligan2011methods};
and financial institutions maintain default rates below $10^{-3}$~\citep{zimper2014minimal}.

\acp{cco} have been successfully applied across diverse domains, including power systems~\citep{bienstock2014chance}, emergency response design~\citep{beraldi2004designing}, portfolio optimization~\citep{bonami2009exact},
and humanitarian logistics~\citep{elcci2018chance}.
Despite their broad applicability, solving these problems remains computationally intractable in many practical settings.
The fundamental difficulty stems from the chance constraint itself: evaluating
if a candidate solution $\vx$ satisfies $\mathbb{P}_{\vxi}\left(\{\vz: g(\vx,\vz) \leq 0\}\right) \geq 1-\varepsilon$ requires computing
a multi-dimensional integral over $\mathbb{P}_{\vxi}$, which admits closed-form solutions only for highly specialized combinations of constraint
functions and probability distributions~\citep{lagoa2005probabilistically, ahmed2008solving}.
Moreover, even when $g(\cdot,\vz)$ is convex for each $\vz$, the feasible
region defined by the chance constraint is typically non-convex, precluding the use of modern convex optimization algorithms~\citep{nemirovski2007convex}.
Indeed, the general \ac{cco} problem is known to be
NP-hard~\citep{luedtke2010integer}.

Given these complexities, much of the existing research has focused on sample-based approximation methods with two primary approaches: \ac{saa} and \ac{sa}.
The former approximates the chance constraint in~\eqref{eqn:cc_basic} by its empirical counterpart.
While \ac{saa} solutions converge to the true solution as sample size increases~\citep{luedtke2008sample,pagnoncelli2009sample}, the method faces significant practical limitations: the resulting problem remains non-convex, and the indicator functions encoding constraint violations often necessitate mixed-integer reformulations~\citep{luedtke2010integer}, substantially increasing computational burden. Although \cite{geletu2017inner,pena2020solving} have proposed smoothed \ac{saa} methods using continuous approximations, convexity still cannot be guaranteed.

In contrast, the \ac{sa}, based on earlier results on statistical learning theory (e.g., see~\citet{Vidyasagar2003,tempo2005randomized}) and later adapted for convex optimization problems by \citet{calafiore2005uncertain}, offers a more tractable alternative. Rather than approximating the probability in~\eqref{eqn:cc_basic}, \ac{sa} replaces the chance constraint with a finite set of deterministic constraints,
\begin{equation}\label{eq:sp_constraints}
    g\big(\vx,\vz^{(j)}\big) \leq  0, \quad j=1,2,\ldots,N,
\end{equation}
enforced simultaneously for the $N$ samples $\vz^{(1)},\vz^{(2)},\ldots,\vz^{(N)}$ drawn from $\mathbb{P}_{\vxi}$. 
This formulation preserves convexity: when $g(\cdot,\vz)$ is convex for each $\vz$, the resulting scenario problem remains convex and thus computationally tractable. Furthermore, for a given risk tolerance level $\varepsilon$  
and sampling confidence $\beta$,
explicit bounds on the required sample size have been derived~\citep{calafiore2005uncertain, calafiore2006scenario, campi2008exact, calafiore2010random}. We formally review these bounds later; for now, we highlight that these bounds provide \emph{distribution-free} guarantees on solution feasibility that hold regardless of the underlying probability distribution.

A major drawback of the SA, however, is that the required sample size $N$ scales as $\varepsilon^{-1}$ to ensure chance constraint feasibility.
This requirement becomes computationally prohibitive for safety-critical applications where the allowed risk level $\varepsilon$ is typically smaller than $10^{-3}$.
Indeed, since each sample introduces a constraint in the optimization problem, larger sample sizes directly translate to increased computational complexity.
For instance, to guarantee a maximum risk of
$\varepsilon = 0.1\%$ with $1-\beta = 95\%$ confidence, a standard sample size bound from~\cite{campi2009scenario} requires solving an optimization problem with 7,992 constraints,
even when the problem features only a single decision variable. 
This requirement becomes prohibitive for large systems with stringent risk requirements and limited computational resources, such as real-time model predictive control and other edge applications with constrained hardware.

Motivated by these challenges, this paper develops a sample-efficient \ac{sa} for the rare-event regime
while maintaining theoretical feasibility guarantees. Our central research question is whether we can leverage the tail properties of $\mathbb{P}_{\vxi}$ and structure of $g$ 
to reduce sample complexity without restricting applicability.
To address this question, we propose a novel 
\emph{decision-scaled} \ac{sa} that exploits the interplay between the tail behavior of $\mathbb{P}_{\vxi}$ and the asymptotic structure of $g$.
Our method replaces the classical SA constraints~\eqref{eq:sp_constraints} with the following scaled constraints:
\begin{equation}\label{eq:ssp_constraints}
    g\big(s^{-\gamma}\vx,\vz^{(j)}\big) \leq  0, \quad j=1,2,\ldots,N_s,
\end{equation}
where $s\geq1$ is a tunable hyperparameter and $\gamma \neq 0$ is a constant determined by the asymptotic structure of $g$.
Our main result provides an explicit bound for $N_s$ that scales as $\varepsilon^{-(1/s^{\alpha})}$, where
$\alpha > 0$ is a certain tail index of the distribution $\mathbb{P}_{\vxi}$,
which is known for common distributions or can be reliably estimated from data.
The bound for $N_s$ depends only on the problem dimension $n$, risk level $\varepsilon$, confidence level $\beta$, hyperparameter $s$ and the tail index $\alpha$.
Notably, this bound is orders of magnitude smaller than the sample size required by the classical \ac{sa}.
We show that when the constraint $g$ exhibits a practically verifiable asymptotic structure and the distribution $\mathbb{P}_{\vxi}$ satisfies certain nonparametric tail conditions, solving the scenario problem using the scaled constraints~\eqref{eq:ssp_constraints} guarantees chance-constraint feasibility for all sufficiently small $\varepsilon$.

Our main contributions can be summarized as follows.
\begin{enumerate}
    \item \textbf{Polynomial reduction in sample complexity.}   
    Our decision-scaled \ac{sa} achieves a polynomial reduction in the required sample size compared to classical \ac{sa}, from $O(\varepsilon^{-1})$ to $O(\varepsilon^{-1/s^{\alpha}})$ for any $s > 1$, while providing asymptotic feasibility guarantees for~\eqref{eqn:cc_basic}. 
    This improvement is realized through a simple scaling of the decision vector 
    requiring no specialized algorithmic machinery or distribution-specific tuning. 
    
    \item \textbf{Wide applicability through mild structural conditions.}
    Roughly speaking, our approach requires only that the constraint function $g(\vx,\vz)$ be asymptotically homogeneous (a property satisfied by broad classes of constraints including linear, quadratic, posynomial, and other nonlinear forms) and that $\mathbb{P}_{\vxi}$ have a certain polynomial tail decay in log-scale.
    This latter condition 
    is satisfied by many distributions encountered in practice and can also be verified using standard techniques from extreme value theory.
    These mild assumptions enable us to establish a uniform large deviation principle for chance constraints in the rare-event regime.
  
    \item \textbf{Numerical validation.} 
    We validate our decision-scaled \ac{sa} through numerical experiments on portfolio optimization, structural engineering, and norm optimization benchmarks, 
    which feature linear, nonlinear, and joint chance constraint structures, respectively.
    In these test cases, our method achieves substantial computational reductions compared to classical \ac{sa} while guaranteeing feasibility.
    Open-source implementations are provided.
\end{enumerate}

\subsection{Related Literature}
Various approaches have been recently developed to address the computational challenges of solving \acp{cco}, particularly in the rare-event regime~\citep{subramanyam2023chance}. We categorize
these methods and position our contribution within the existing landscape.

Several works have adapted \ac{is} concepts to improve efficiency in evaluating the chance constraints in~\eqref{eqn:cc_basic}. 
Early work by~\citet{rubinstein1997optimization,rubinstein2002cross} established \ac{is} techniques for rare events in optimization contexts. 
\citet{nemirovski2006scenario} introduced sample scaling via majorizing distributions for \ac{sa}, although finding suitable distributions for a given application remains challenging.
\citet{barrera2016chance} combined \ac{saa} with \ac{is} to uniformly reduce the required sample size across all feasible decisions. However, their approach relies on specific problem structures, such as independent Bernoulli distributed uncertainty, to derive a suitable \ac{is} distribution.
\citet{blanchet2024efficient} proposed conditional \ac{is} for heavy-tailed distributions in \ac{sa}, requiring analytically computable constraint approximations.
Domain-specific \ac{is} methods like \citet{lukashevich2023importance} exploit structure in specific power systems optimization problems.

Recent works leverage \acp{ldp} for rare-event \acp{cco}.
\cite{tong2022optimization} reformulated chance constraints as bilevel problems, eliminating sampling but requiring the uncertainty to follow Gaussian mixture distributions.
\cite{blanchet2024optimization} derived asymptotic relationships between \ac{cco} and \ac{sa} under mild assumptions similar to ours, without addressing practical sample complexity or algorithmic implementation.
Similarly, \citet{deo2025scaling} characterized the scaling behavior of optimal values in the rare-event regime, focusing on theoretical properties rather than computational methods.
Our earlier work~\citep{choi2024reduced} introduced constraint scaling for linear constraints to reduce sample complexity of \ac{sa}. 
This paper extends that idea by scaling decision variables instead of constraints, enabling application to general nonlinear constraints.

Several techniques improve \ac{sa} efficiency without specifically targeting rare events.
\citet{campi2011sampling,romao2022exact}
developed constraint removal methods that improve performance of \ac{sa} by achieving less conservative solutions, rather than reducing sample complexity. \citet{care2014fast,schildbach2014scenario} 
proposed methods to reduce the required sample size of \ac{sa} by exploiting problem structure or domain knowledge. However, these approaches often require iterative algorithms or are tailored to specific applications such as model predictive control.

In contrast to the aforementioned works, our decision-scaling method provides a general recipe for achieving polynomial sample reduction for broad classes of
problems without requiring distribution-specific tuning, analytical approximations, or domain-specific knowledge.

\paragraph{Outline.} The rest of the paper is organized as follows.
Section~\ref{Sec:Preliminaries} provides technical preliminaries on tail modeling of distributions, followed by our main assumptions on the uncertainty distribution $\mathbb{P}_{\vxi}$ and the constraint function $g$. 
Section~\ref{Sec:ProposedMethod}
introduces our proposed decision-scaled \ac{sa} and presents our main theoretical contributions. 
Section~\ref{Sec:VerifyingConditions}
provides systematic and practical verification procedures for our 
assumptions along with illustrative examples.
Section~\ref{Sec:NumExp} demonstrates the effectiveness of our method through experiments on three benchmark problems: portfolio optimization, reliability-based column design, and norm optimization.
All proofs are deferred to Section~\ref{Sec:Proofs} to maintain readability of the main exposition.

\paragraph{Notations.}
Scalars are denoted by plain type ($a$), vectors by boldface ($\va$), and sets by calligraphic script ($\mathcal{A}$). 
We write $\mathbb{R}_{>0}\coloneqq(0,\infty)$ and $\mathbb{R}_{\ge0}\coloneqq[0,\infty)$. 
The zero vector is $\vzero$ and the vector of ones is $\vone$; their dimensions will be clear from context.
For sets $\mathcal{A}$ and $\mathcal{B}$,
their Cartesian product is $\mathcal{A}\times\mathcal{B}$, and their Minkowski sum is $\mathcal{A} + \mathcal{B} \coloneqq \{\va + \vb : \va \in \mathcal{A}, \vb \in \mathcal{B}\}$.
The closure, interior, and complement of $\mathcal{A}$ are denoted by $\overline{\mathcal{A}}$, $\mathcal{A}^\circ$, and $\mathcal{A}^c$, respectively.

Unless stated otherwise, $\|\cdot\|$ denotes the $\ell_2$-norm.
The compact ball of radius $\theta>0$ centered at $\vc$ is $\mathcal{B}_\theta(\vc)\coloneqq \{\vx : \|\vx-\vc\| \leq \theta\}$, with $\mathcal{B}_\theta\coloneqq\mathcal{B}_\theta(\vzero)$.
Given a function $f:\mathbb{R}^n\to\mathbb{R}$, the $a$-superlevel set of $f$ restricted to $\mathcal{B}_\theta$ is $\mathcal{L}^\theta_{\ge a}(f) \coloneqq \{\vx \in \mathcal{B}_\theta : f(\vx) \geq a\}$ and the corresponding strict superlevel set is $\mathcal{L}^\theta_{>a}(f) \coloneqq \{\vx \in \mathcal{B}_\theta : f(\vx) > a\}$. When $\theta=\infty$, we write $\mathcal{L}_{\ge a}(f)$ and $\mathcal{L}_{>a}(f)$, respectively.
The notation $\{\vz_u\}$ denotes a collection of vectors parameterized by $u \in \mathbb{R}_{>0}$.
We write $\lim_{u\to\infty}\vz_u=\vz$ if and only if
$\lim_{k\to\infty}\vz_{u_k}=\vz$ for every sequence $\{u_k\}_{k\in\mathbb{N}}\subset\mathbb{R}_{>0}$ satisfying $\lim_{k\to\infty}u_k=\infty$.

\section{Preliminaries and Assumptions}\label{Sec:Preliminaries}

\subsection{Regular Variation}
\begin{definition}\label{def:RV}
    A function $f:\mathbb{R}_{>0}\to\mathbb{R}_{>0}$ is 
    said to be
    \emph{regularly varying} with index $\kappa\in\mathbb{R}_{>0}$\footnote{Regular variation can permit $\kappa \in \mathbb{R}$ in general.}  
    if 
    \begin{align}\label{eqn:rv}
    \lim_{u\to\infty}\frac{f(uz)}{f(u)}=z^\kappa,
    \quad\forall  z>0.
    \end{align}
\end{definition}
We write $f\in \RV(\kappa)$, or simply $f\in \RV$ when the index is not explicitly needed.
Regular variation characterizes functions with polynomial-like asymptotic behavior. 
Examples of such functions include $z^\kappa$, 
$z^\kappa\log(1+z)$, and $z^\kappa\exp\left(\sqrt{\log(1+z)}\right)$. 

\begin{definition}\label{def:mrv}
    A function $f:\mathcal{Z}\to \mathbb R_{\ge0}$, where $\mathcal{Z}\subseteq\mathbb{R}^m$ is a cone, is said to be \emph{multivariate regularly varying} if there exists $h\in \RV$ and a continuous function $f^*:\mathcal{Z}\to\mathbb{R}_{\ge0}$ such that 
\begin{equation}\label{eqn:mrv}
\lim_{u\to\infty}\frac{f(u\vz_u)}{h(u)} = f^*(\vz)
\end{equation}   
for every convergent sequence $\{\vz_u\}\to\vz$.
\end{definition}
We write $f\in \MRV(\mathcal{Z},h,f^*)$ or simply $f\in \MRV$.
Here, $h$ and $f^*$ encode the asymptotic radial scaling and directional dependence of $f$, respectively.
We refer the reader to~\citet[Chapter 2 and 5]{resnick2007heavy} for more details.

\subsection{Assumptions on the Uncertainty Distribution}
The multivariate extension of regular variation provides a natural framework for characterizing the tail behavior of $\mathbb{P}_{\vxi}$.

\begin{assumption}\label{assum:light-tail}
The probability distribution $\mathbb{P}_{\vxi}$ admits a density function of the form $\exp(-Q(\vz))$ where $Q\in \MRV(\supportset, q,\lambda)$ with $q\in\RV(\alpha)$ for some $\alpha>0$. Additionally, $\lambda(\vz)>0$ for all $\vz\in\supportset\cap\{\vz'\in\mathbb{R}^m:\|\vz'\|=1\}$.
\end{assumption}

\begin{remark}\label{Remark:ClosedCone}
Implicit in Assumption~\ref{assum:light-tail} is that the support set $\supportset$ must be a closed cone.
The closedness follows from the definition of support while the conic property arises from the multivariate regular variation of $Q$.
\end{remark}

Intuitively, Assumption~\ref{assum:light-tail} states that the negative log-density $Q(u\vz)$ behaves like $q(u)\lambda(\vz)$ for large $u$ and unit $\vz$,
where $q$ governs how rapidly $Q$ grows as one moves away from the origin,
while $\lambda$ acts as a directional coefficient modulating this growth.
We emphasize that this assumption is nonparametric in nature.
This is in contrast to existing literature that often relies on specific parametric distribution families~\citep{barrera2016chance,tong2022optimization}. It imposes only a mild structural condition on the distribution tail, requiring that the log-density exhibits polynomial behavior asymptotically. 
Importantly, this assumption does not require the exact functional form of the distribution but only its tail behavior,
eliminating the need for normalization constants that are often intractable to compute.
The rate of the behavior is governed by the tail index $\alpha>0$, where a larger
$\alpha>0$ corresponds to a lighter tail (i.e., faster decay of tail probabilities). This general framework encompasses a wide range of commonly used distributions and their mixtures, including light-tailed families like the Gaussian ($\alpha=2$) and Exponential ($\alpha=1$), as well as sub-exponential families like the Weibull ($\alpha<1$).

This framework is flexible enough to accommodate non-parametric tail correlations, as the dependence structure between marginal density is implicitly encoded in $\lambda$.
This includes various dependences modeled by copulas, such as the Gumbel, Clayton, and other Archimedean types; see \citet[EC.3]{deo2023achieving} for further details on such constructions.
We emphasize that a significant advantage of our proposed decision-scaled \ac{sa} is that it does not require their explicit distributional formulation. 

Our method's implementation relies solely on $\alpha$, which characterizes the tail decay rate of the marginal distributions and is invariant to the underlying dependence structure. This property makes our approach truly practical and nonparametric; from the implementation perspective, distributions with the same tail index, such as Chi and Gaussian ($\alpha=2$), are treated identically.
The tail index $\alpha$ is often a well-known parameter for common distributions.
Table~\ref{table:distributions} provides several examples of distribution families and their corresponding tail indices.
The practicality of our method is further enhanced by the fact that $\alpha$ can be reliably estimated from data without fitting an entire (parameteric) distribution to the data~\citep{de2016approximation,einmahl2021testing,einmahl2025empirical}.

\begin{table}[!ht]
    \centering
    \begin{threeparttable}
    \caption{Examples of distributions satisfying Assumption~\ref{assum:light-tail}}
    \label{table:distributions}
    \begin{tabularx}{\textwidth}{X|c}
        \toprule
        Distribution family\tnote{1} & Tail index $\alpha$\\
        \midrule\midrule
        Chi-squared, Erlang, Exponential, Gamma, Inverse Gaussian, Laplace & $1$\\ 
        \midrule
        Chi, Gaussian, Gaussian mixtures, Maxwell–Boltzmann,  Rayleigh & $2$\\ 
        \midrule
        Generalized-gamma with shape parameter $k$ & $k$\\ 
        \midrule
        Generalized gaussian with shape parameter $k$ & $k$\\
        \midrule
        Weibull with shape parameter $k$ & $k$ \\
        \bottomrule
    \end{tabularx}
    \begin{tablenotes}
        \small
        \item[1] Unless explicitly specified (e.g., multivariate normal, Gaussian mixtures), these families are constructed from the corresponding univariate marginals, often assuming independence or a specified copula.
    \end{tablenotes}
    \end{threeparttable}
\end{table}

\subsection{Assumptions on the Constraint Function}
We remind the reader that the deterministic feasible set $\mathcal{X}$ is closed and convex, and the constraint function $g: \mathcal{X}\times \supportset \to \mathbb{R}$ is such that $g(\cdot, \vz)$ is closed and convex for any fixed $\vz$.
To characterize its asymptotic structure, we first introduce the following concept of a rate-parameterized asymptotic cone.
\begin{definition}[Asymptotic Cone]\label{def:asymptoset}
    For a nonempty set $\mathcal{A}$ and $\theta\neq0$, an \emph{asymptotic cone} of $\mathcal{A}$ with rate $\theta$ is defined as
    \begin{align}
        \mathcal{A}^{\infty}_{\theta} \coloneqq \left\{ \vy : \exists \{\va_u\} \subset \mathcal{A}, \exists \{\lambda_u\} \subset \mathbb{R}_{>0} \text{ s.t. } \lambda_u \to \infty \text{ and } \frac{\va_u}{(\lambda_u)^\theta} \to \vy \right\}.
    \end{align}
\end{definition}
Definition~\ref{def:asymptoset} unifies two fundamental geometric concepts in optimization.
Specifically, when $\theta>0$, $\mathcal{A}^\infty_\theta$ coincides with the horizon cone~\citep[Definition 3.3]{rockafellar2009variational},
capturing the asymptotic directions along which the set $\mathcal{A}$ extends to infinity.
Conversely, when $\theta<0$, $\mathcal{A}^\infty_\theta$ corresponds to the tangent cone of $\mathcal{A}$ at the origin~\citep[Definition 6.1]{rockafellar2009variational}.
In either case, the set $\mathcal{A}^\infty_\theta$ forms a cone characterizing the asymptotic geometry of $\mathcal{A}$ 
under the scaling regime governed by $\theta$.

Notably, for any closed cone $\mathcal{A}$, we have
$\mathcal{A}^\infty_\theta=\mathcal{A}$ for all $\theta\neq0$. 
In particular, under Assumption~\ref{assum:light-tail}, since the support set $\supportset$ is a closed cone (Remark~\ref{Remark:ClosedCone}), we obtain $\supportset^\infty_\theta=\supportset$ for any $\theta\neq0$.
This unified framework allows us to consistently analyze both the unbounded and infinitesimal properties of the feasible region depending on the problem structure.

\

\begin{assumption}\label{assum:conti-conv}
    The constraint function $g$ satisfies the following properties:
    \begin{enumerate}[label=({A\arabic*}), leftmargin=*]
        \item\label{Eq:Conti-Convergence} There exist constants $\gamma\neq0$, $\rho\geq 0$, and a continuous function $g^*\colon\mathcal{X}^{\infty}_{\gamma}\times\supportset\to \mathbb{R}$
        such that,
        for every $\vy\in\mathcal{X}^{\infty}_{\gamma}$and $\vw\in\supportset$, whenever 
        sequences $\{\vx_u\}\subset\mathcal{X}$ and $\{\vz_u\}\subset\supportset$ satisfy
        \begin{equation}\label{eq:relation_x_and_y}
        \lim_{u \to \infty} \frac{\vx_u}{u^\gamma} = \vy \quad \text{and} \quad \lim_{u \to \infty} \frac{\vz_u}{u} = \vw,
        \end{equation}
        the following limit holds:
        \begin{equation}\label{eq:direct_conti_conv}
        \lim_{u \to \infty} \frac{g(\vx_u, \vz_u)}{u^{\rho}} = g^*(\vy, \vw).
        \end{equation}
        
        \item\label{Assum:NonEmptyHorizon} $\mathcal{X}^{\infty}_{\gamma}\setminus\{\vzero\}\neq\emptyset$;
        
        \item\label{eq:away_from_0} 
        $g^*(\vy,\vzero)< 0$ for all $\vy\in \mathcal{X}^{\infty}_{\gamma}\setminus\{\vzero\}$;

       \item\label{eq:nonempty_set_g^*} 
       $\{\vw\in\supportset: g^*(\vy,\vw)> 0\}\neq \emptyset$ for all $\vy\in \mathcal{X}^{\infty}_{\gamma}\setminus\{\vzero\}$;

        \item\label{assum:trivial_safety_at_0} 
        If $\gamma<0$, then $g(\vzero,\vz)\le0$ for all $\vz\in\supportset$.
    \end{enumerate}
    
\end{assumption}
Assumption~\ref{Eq:Conti-Convergence} requires the constraint function to exhibit an asymptotic homogeneity, characterized by the limit function $g^*$.
When the decision and random vectors scale by factors $u^\gamma$ and $u$ respectively, the constraint function scales by $u^\rho$, ensuring the limit remains nondegenerate. 
The sign of $\gamma$ encodes the problem structure, consistent with Definition~\ref{def:asymptoset}.
A positive $\gamma$ indicates that decisions grow with uncertainty, 
while a negative $\gamma$ indicates that decisions shrink as uncertainty grows.

Assumption~\ref{Assum:NonEmptyHorizon} ensures that the asymptotic cone $\mathcal{X}^{\infty}_{\gamma}$ is non-trivial. %
Assumption~\ref{eq:away_from_0} ensures that asymptotically negligible uncertainty does not trigger constraint violations; in other words, the system remains safe when uncertainty vanishes. 
Conversely, Assumption~\ref{eq:nonempty_set_g^*} 
prevents triviality by guaranteeing that every non-zero decision direction carries inherent risk in the asymptotic limit.
Notably, when $\gamma < 0$, Assumption~\ref{Assum:NonEmptyHorizon} and the closedness of the feasible set $\mathcal{X}$ imply that $\vzero \in \mathcal{X}$. 
This guarantees that $g(\vzero, \vz)$ is well-defined,
allowing Assumption~\ref{assum:trivial_safety_at_0} to further ensure that the origin is feasible for all uncertainty realizations.

These conditions, while technical, encompass a broad class of practical 
constraints and exclude pathological cases. Verification procedures 
and examples are provided in Section~\ref{Sec:VerifyingConditions}. For now, we simply note that the scaling exponents $(\gamma,\rho)$ directly govern the decision-scaled \ac{sa} that we introduce in the next section. Their uniqueness is therefore essential to ensure that the method is well-defined and yields unambiguous sample complexity bounds.
The following result establishes this uniqueness.

\begin{proposition}[Uniqueness of Scaling]\label{Lem:UniqueScalingExponents}
    Under Assumptions~\labelcref{Eq:Conti-Convergence,Assum:NonEmptyHorizon,eq:away_from_0,eq:nonempty_set_g^*}, the pair $(\gamma,\rho)$ is unique.
\end{proposition}

\section{Decision-Scaled \ac{sa}}\label{Sec:ProposedMethod}

\subsection{Overview of Classical \ac{sa}}\label{Scenario_approach}
Let $\{\vz^{(j)}\}_{j=1}^{N}$ denote $N$ \ac{iid} samples, or scenarios, drawn from $\mathbb{P}_{\vxi}$. 
The \ac{sa} approximates \eqref{eqn:cc_basic} by solving the following so-called scenario problem:
\begin{equation}
    \label{eq:sp}
    \tag{$\text{SP}_{N}$}
    \begin{aligned}
        \minimize_{\vx\in \mathcal{X}} & \quad c(\vx)\\
        \text{subject to}   & \quad g\left(\vx,\vz^{(j)}\right) \leq  0, \quad j=1,\ldots,N.
    \end{aligned}
\end{equation}
This problem is convex since $g(\vx,\vz)$ is convex in $\vx$ for every fixed~$\vz$.
\begin{definition}
    For a given $\vx\in\mathcal{X}$, the violation probability is defined as $V(\vx)\coloneqq\mathbb{P}_{\vxi}(\{\vz:g(\vx,\vz)>0\})$.
\end{definition}
Consequently, the condition $V(\vx) \le \varepsilon$ ensures that $\vx$ satisfies the feasibility requirement of~\eqref{eqn:cc_basic}.

Let ${\vx}^*_{N}$ denote an optimal solution of the scenario problem~\eqref{eq:sp}. The central result of the \ac{sa} is to establish the sample requirement that ensures feasibility with respect to the original chance constraint.
\begin{theorem}[\cite{campi2009scenario}, Theorem~1]
\label{Thm:ScenarioApproach}
Given a risk tolerance level $\varepsilon\in (0,1)$ and a confidence parameter $\beta\in (0,1)$, choose
\begin{align}\label{eq:numsample}
    N \geq \mathsf{N}(\varepsilon,\beta) \coloneqq  \left\lceil\frac{2}{\varepsilon} \left( \log \frac{1}{\beta} + n \right)\right\rceil.
\end{align}
If the scenario problem~\eqref{eq:sp} admits an optimal solution ${\vx}^*_{N}$, then 
\begin{align*}
    \mathbb{P}^N_{\vxi}(V({\vx}^*_{N})\le\varepsilon)\ge 1-\beta,
\end{align*}
where $\mathbb{P}^N_{\vxi}$ is the $N$-fold product probability distribution.
\end{theorem}

Theorem~\ref{Thm:ScenarioApproach} establishes that, with probability at least $1-\beta$ over sample realizations, the optimal solution $\vx^*_N$ satisfies the feasibility requirement for~\eqref{eqn:cc_basic}.
Notably, this guarantee is \emph{a priori}---it holds independently of the specific sample realizations of the uncertainty and is \emph{distribution-free}, maintaining validity regardless of the underlying probability distribution.

However, this generality comes at a computational cost.
In particular, for applications where $\varepsilon$ is exceptionally small, typically $10^{-3}$ to $10^{-5}$, the required sample size scales as $1/\varepsilon$. 
For instance, achieving $\varepsilon=10^{-3}$ with 99\% confidence ($\beta=0.01$) requires $N\ge 11,211$ samples, even with only a single decision variable ($n=1$).
Since each sample adds a new constraint in~\eqref{eq:sp}, the resulting optimization problem explodes in size and becomes intractable.

\subsection{Decision-scaled \ac{sa}}
The decision-scale \ac{sa} addresses the computational intractability of classical \ac{sa} in the rare-event regime.
Throughout this section, 
we assume that the distribution $\mathbb{P}_{\vxi}$ satisfies Assumption~\ref{assum:light-tail} with tail index $\alpha>0$, and the constraint function $g$ satisfies Assumption~\ref{assum:conti-conv} with parameters $(\gamma,\rho)$.
Our approach is parameterized by a \emph{scaling hyperparameter} $s \ge 1$, and is formulated as follows:
\begin{equation}
    \label{eq:s-sp}
    \tag{$\text{SSP}_{N, s}$}
    \begin{aligned}
        \minimize_{\vx\in \mathcal{X}}&\quad c(\vx)\\
        \text{subject to}&\quad g\left(s^{-\gamma}\vx, \vz^{(j)}\right) \leq 0, \quad j=1,\ldots,N,
    \end{aligned}
\end{equation}
where $\{\vz^{(j)}\}_{j=1}^{N}$ are i.i.d. samples drawn from $\mathbb{P}_{\vxi}$. 
Here, we choose
\begin{align}\label{eqn:scaled_N}
    N \geq \mathsf{N}\left(\varepsilon^{1/s^{\alpha}},\beta\right) =  \left\lceil\frac{2}{\varepsilon^{1/s^{\alpha}}} \left( \log \frac{1}{\beta} + n \right)\right\rceil.
\end{align}
In contrast to existing methods~\citep{nemirovski2006scenario,blanchet2024efficient} that incur the computational overhead of constructing \ac{is} distributions, our approach is easy to implement as it relies exclusively on samples from the original distribution $\mathbb{P}_{\vxi}$.

When $s=1$, problem~\eqref{eq:s-sp} reduces to classical \ac{sa}. For $s>1$, the key insight is that scaling the decision vectors by $s^{-\gamma}$ tightens each constraint, effectively restricting the feasible region. Intuitively, this conservatism allows for fewer samples while maintaining feasibility guarantees. The following theorem formalizes this intuition.
\begin{theorem}\label{thm:LT_feasible}
    Fix any
    $s \geq 1$ and $\beta \in (0,1)$. 
    Let $\{\vx_{\varepsilon}\}$ be a sequence of optimal solutions to~\eqref{eq:s-sp}.
    Then, 
    \begin{equation}\label{eq:asymptotic_feasibility}
    \mathbb{P}_{\vxi}^\infty\left(\liminf_{\varepsilon \rightarrow 0}\frac{\log V(\vx_{\varepsilon})}{\log \varepsilon} \geq 1\right)\ge 1-\beta,
    \end{equation}
    where $\mathbb{P}_{\vxi}^\infty$ is the infinite product probability distribution.
\end{theorem}

The probability statement in Theorem~\ref{thm:LT_feasible} is defined on the space of infinite sequences of \ac{iid} samples drawn from $\mathbb{P}_{\vxi}$.
This guarantees that, with probability at least $1-\beta$, the stated asymptotic behavior~\eqref{eq:asymptotic_feasibility} holds for any realization of the infinite sample path.
While problem~\eqref{eq:s-sp} may become infeasible for certain sample realizations, Theorem~\ref{thm:LT_feasible} implicitly considers only cases where $\{\vx_\varepsilon\}$ is well-defined for sufficiently small $\varepsilon$. Otherwise, as with the classical \ac{sa}, no conclusion can be provided.

To build intuition, note that since $\varepsilon\in(0,1)$ we have $\log\varepsilon<0$, so the condition $\log V(\vx_\varepsilon)/\log\varepsilon\ge 1$ is equivalent to $V(\vx_\varepsilon)\le\varepsilon$.
Thus, the theorem states that the violation probability of the decision-scaled solution $\vx_\varepsilon$ eventually decays at least as fast as the prescribed risk level $\varepsilon$.
In other words, our method produces solutions that are asymptotically feasible for~\eqref{eqn:cc_basic} as the risk tolerance shrinks to zero.
More precisely, the limit inferior in~\eqref{eq:asymptotic_feasibility} implies that for any $\delta>0$, there exists $\varepsilon_0$ such that $\log V(\vx_\varepsilon)/\log\varepsilon\ge 1-\delta$ for all $\varepsilon\in(0,\varepsilon_0)$. Equivalently,
\begin{align}\label{eq:weakfeasibility}
    V(\vx_\varepsilon)\le \varepsilon^{1-\delta},\quad\forall\varepsilon<\varepsilon_0.
\end{align}
While the bound in~\eqref{eq:weakfeasibility} is slightly weaker than
$V(\vx_\varepsilon)\le\varepsilon$, 
the fact that $\delta$ can be made arbitrarily close to $0$
makes this difference negligible in practice.

The proof relies on establishing an \ac{ldp} for the violation probability $V(\vx_\varepsilon)$. 
The key step involves deriving a rate function that captures the exponential decay of this probability, determined by the tail index $\alpha$ of $\mathbb{P}_{\vxi}$ and the asymptotic structure of $g$.
The detailed theoretical development is provided in Section~\ref{secA1}.

An immediate consequence of Theorem~\ref{thm:LT_feasible} is the following polynomial reduction in sample complexity:
\begin{corollary}
Compared to the classical \ac{sa}, the decision-scaled \ac{sa} achieves a polynomial reduction in sample complexity from $O(\varepsilon^{-1})$ to $O(\varepsilon^{-1/s^{\alpha}})$ for obtaining a feasible solution to~\eqref{eqn:cc_basic} when $\varepsilon$ is sufficiently small.
\end{corollary}
This corollary follows directly from the asymptotic relationship:
\begin{equation}
\lim_{\varepsilon \to 0} \frac{\log \mathsf{N}(\varepsilon^{s^{-\alpha}}, \beta)}{\log \mathsf{N}(\varepsilon, \beta)} = \frac{1}{s^\alpha}.
\end{equation}
The reduction in sample complexity translates directly to computational savings. 
In the \ac{sa} framework, the number of samples equals the number of constraints in the optimization problem, thus a reduction in sample size by a factor of $s^\alpha$
yields proportional computational time reduction.
Consider the example from Section~\ref{Scenario_approach} ($\varepsilon=10^{-3},\beta=0.01,n=1$).
As noted previously, the classical \ac{sa} requires $N\ge 11,211$ samples. In contrast, our decision-scaled approach with $s=1.2$ for $\alpha=2$  requires only $N\ge 1,359$ samples, representing an eight-fold reduction.
Moreover, this efficiency gain is particularly valuable in applications where samples are limited or computationally expensive to generate, such as sampling from high-dimensional distributions.

While Theorem~\ref{thm:LT_feasible} holds for all $s\ge1$, the choice of $s$ introduces a trade-off.
Our empirical findings in Section~\ref{Sec:NumExp} show that increasing $s$ yields more conservative solutions with larger feasibility margins, although detailed analysis of this observation is beyond the scope of this paper and will be investigated in future work.
This trade-off between computational efficiency and solution quality provides flexibility.
Practitioners can select larger scaling factors $s$ for real-time applications or limited computational resources.
This tunability enables solving previously intractable problems by adjusting $s$ to match available computational budget.

\section{Verification of Regularity Conditions and Examples}\label{Sec:VerifyingConditions}

The aforementioned theoretical guarantees rely on Assumption~\ref{assum:conti-conv}, which requires the constraint function to exhibit a certain asymptotic homogeneity. %
In this section, we develop a systematic and practical procedure to verify this assumption.
We first focus on the broad class of \emph{algebraic} constraint functions; that is, functions expressible as finite sums of monomial terms in the decision and uncertainty vectors, which encompasses polynomials, posynomials, and signomials commonly arising in practice.
For this class, we show that %
verifying the required assumption %
is relatively straightforward.
We then extend it to handle non-algebraic constraint functions as well as joint chance constraints.
Throughout, we illustrate the framework through concrete examples.

\subsection{Verification of Assumption~\ref{assum:conti-conv} for Algebraic Functions}

Let $\va=(a_1,\ldots,a_n)\in\mathbb{R}^n$ and $\vb=(b_1,\ldots,b_n)\in\mathbb{R}^m$ denote real exponents.
We adopt the standard multi-index notation
\begin{align}
    \vx^{\va} \coloneqq \prod_{i=1}^n x_i^{a_i}
    \quad\text{and}\quad
    \vz^{\vb} \coloneqq \prod_{i=1}^m z_i^{b_i}.
\end{align}
We consider constraint functions of the form
\begin{align}\label{eq:algebraic_function}
g(\vx,\vz)=\sum_{(\va,\vb)\in\mathcal{J}}C_{\va,\vb}\vx^{\va}\vz^{\vb},
\end{align}
where $\mathcal{J}\coloneqq\{(\va,\vb)\in\mathbb{R}^n\times\mathbb{R}^m\colon C_{\va,\vb}\neq 0\}$ is a finite index set with coefficients $C_{\va,\vb}\in\mathbb{R}$.
Here, the function $g$ in~\eqref{eq:algebraic_function} is defined on a domain $\mathcal{X}\times\supportset$ such that every term $\vx^{\va}\vz^{\vb}$ is well-defined and real-valued for all $(\vx,\vz)\in\mathcal{X}\times\supportset$.
This class encompasses polynomials, posynomials, and signomials.

In Assumption~\ref{assum:conti-conv},
verifying condition~\ref{Eq:Conti-Convergence} is the most challenging.
The remaining conditions~\labelcref{Assum:NonEmptyHorizon,eq:away_from_0,eq:nonempty_set_g^*,assum:trivial_safety_at_0} are typically straightforward to check once scaling exponents $(\gamma,\rho)$ and the limit function $g^*$ are identified.
To make~\ref{Eq:Conti-Convergence} more accessible, we decompose its verification into two steps:
\begin{enumerate}[label=(\roman*),leftmargin=*]
    \item Identify the scaling exponents $(\gamma,\rho)$ and the limit function $g^*$;
    \item Verify that the convergence in~\ref{Eq:Conti-Convergence} holds.
\end{enumerate}

\medskip
\noindent
\textbf{Part (i)}.
The key observation is that condition~\ref{Eq:Conti-Convergence} implies a specific asymptotic structure of the constraint function $g$.
To systematically identify the scaling exponents $(\gamma,\rho)$ and the limit function $g^*$,
we examine the canonical choice of sequences $\vx_u=u^\gamma\vy$ and $\vz_u=u\vw$ for $(\vy,\vw)\in\mathcal{X}^\infty_\gamma\times\supportset$.
This heuristic reduces limit~\eqref{eq:direct_conti_conv} to the following form:
\begin{align}\label{eq:canonical_limit}
    \lim_{u\to\infty}\frac{g(u^\gamma\vy,u\vw)}{u^\rho}=g^*(\vy,\vw).
\end{align}
Substituting $\vx=u^\gamma\vy$ and $\vz=u\vw$  in~\eqref{eq:algebraic_function} yields
\begin{align}
    g(u^\gamma\vy, u\vw)=\sum_{(\va,\vb)\in\mathcal{J}}u^{p_{\va,\vb}(\gamma)}C_{\va,\vb}\vy^{\va}\vw^{\vb},
\end{align}
where the exponent $p_{\va,\vb}(\gamma)$ is defined as
\begin{align}\label{eq:exponent_function}
    p_{\va,\vb}(\gamma)\coloneqq\gamma\vone^\top\va+\vone^\top\vb.
\end{align}

For the limit function $g^*$ to be well-defined (i.e., finite) and non-trivial (i.e., not identically zero), the scaling exponents $\rho\ge0$ must equal the maximum scaling rate among all terms, i.e., $\rho=\max_{(\va,\vb)\in\mathcal{J}}p_{\va,\vb}(\gamma)$.
The limit function $g^*$ then consists of those terms achieving this maximum:
\begin{align}
    g^*(\vy,\vw)=\sum_{(\va,\vb)\in\mathcal{J},p_{\va,\vb}(\gamma)=\rho}C_{\va,\vb}\vy^{\va}\vw^{\vb}.
\end{align}
The value of $\gamma$ is determined by requiring that $g^*$ satisfies conditions~\labelcref{eq:away_from_0,eq:nonempty_set_g^*}.
Specifically, $\gamma$ must be chosen such that at least two terms achieve the maximum scaling rate, which is necessary to simultaneously satisfy conditions~\labelcref{eq:away_from_0} and~\labelcref{eq:nonempty_set_g^*}\footnote{If a single term dominates, then $g^*$ becomes a monomial. In this case, if the exponent of $\vw$ is non-zero, then $g^*(\vy,\vzero)=0$ which violates~\labelcref{eq:away_from_0}; otherwise, the exponent of $\vw$ is zero and $g^*$ becomes independent of $\vw$, failing to simultaneously satisfy both the strictly negative condition of~\labelcref{eq:away_from_0} and the strictly positive condition of~\labelcref{eq:nonempty_set_g^*}.}.

Algorithm~\ref{alg:find_scaling_exponents} systematically identifies the unique pair $(\gamma,\rho)$ and the corresponding limit function $g^*$.
The first part of the algorithm enumerates candidate values of $\gamma$ that equalize the scaling rates of distinct terms.
It then immediately terminates on line~\ref{alg:earlystop} after finding the first candidate that satisfies the conditions of Assumption~\ref{assum:conti-conv}, since its uniqueness is guaranteed by Proposition~\ref{Lem:UniqueScalingExponents}.
If no candidate satisfies all requirements, the algorithm returns \texttt{None}, indicating that the constraint function $g$ does not admit the structure required by Assumption~\ref{assum:conti-conv};
see Remark~\ref{Remark:notsatisfyingAssumption} for further discussion.
\begin{algorithm}
    \caption{Identification of $(\gamma,\rho)$ and $g^*$ for algebraic functions}
    \label{alg:find_scaling_exponents}
    \begin{algorithmic}[1]  
        \Require The finite index set $\mathcal{J}$ and coefficients $C_{\va,\vb}$ defining $g$ as in~\eqref{eq:algebraic_function}.
        \Ensure  $(\gamma,\rho)$ and $g^*$, or \texttt{None}.
        \Statex
        
        \State Initialize a candidate set of $\gamma$, $\Gamma\leftarrow\emptyset$.
        \ForAll{$\{(\va,\vb), (\va',\vb')\} \subset \mathcal{J}$}\label{alg:calculate_gamma1}
        \If{$\mathbf{1}^\top\va \neq \mathbf{1}^\top\va'$}\label{alg:calculate_gamma2}
            \State Solve for $\gamma \leftarrow \frac{\mathbf{1}^\top(\vb' - \vb)}{\mathbf{1}^\top(\va - \va')}$.\label{alg:calculate_gamma3} 
            \State $\Gamma\leftarrow\Gamma\cup\{\gamma\}$\label{alg:calculate_gamma4}
        \EndIf\label{alg:calculate_gamma5}
        \EndFor\label{alg:calculate_gamma6}

        \ForAll{$\gamma\in\Gamma$} 
        \If{$\gamma = 0$} 
        \State \textbf{continue}
        \EndIf
        \State $\rho \leftarrow \max_{(\va,\vb) \in \mathcal{J}} \{p_{\va,\vb}(\gamma)\}$\label{alg:rho_is_max}
        
        \If{$\rho < 0$} 
        \State \textbf{continue}
        \EndIf
        
        \State $\mathcal{J}^* \leftarrow \{(\va,\vb) \in \mathcal{J} : p_{\va,\vb}(\gamma) = \rho\}$ 
        \If{$|\mathcal{J}^*|\ge 2$} \label{alg:line:check_balance}
        \State $g^*(\vy, \vw) \leftarrow \sum_{(\va,\vb) \in \mathcal{J}^*} C_{\va\vb} \vy^{\va} \vw^{\vb}$.
            \If{$g^* \colon \mathcal{X}^\infty_\gamma \times \supportset \to \mathbb{R}$ is continuous; $\mathcal{X}^\infty_\gamma$ satisfies~\labelcref{Assum:NonEmptyHorizon}; $g^*$ satisfies~\labelcref{eq:away_from_0,eq:nonempty_set_g^*}; $g$ satisfies~\labelcref{assum:trivial_safety_at_0}} 
            \State \Return $(\gamma, \rho), g^*$.\label{alg:earlystop}
            \EndIf
        \EndIf
        \EndFor
        \State\Return \texttt{None}.
    \end{algorithmic}
\end{algorithm}

\begin{remark}\label{Remark:notsatisfyingAssumption}
Not all constraint functions of the form~\eqref{eq:algebraic_function} satisfy Assumption~\ref{assum:conti-conv}. 
This assumption characterizes problems that admit a well-defined asymptotic structure in which a unique scaling relationship governs the dominant terms.
Consider $g(x,z) = xz - x + z - 1$.
If $\gamma>0$, then only the term $xz$ dominates with scaling exponent ${\gamma+1}$, whereas if $\gamma<0$, then only the term $z$ dominates with scaling exponent $1$.
Since no nonzero value of $\gamma$ yields a balance of two or more terms, Algorithm~\ref{alg:find_scaling_exponents} returns \texttt{None}, indicating that Assumption~\ref{assum:conti-conv} is not satisfied by this constraint function.
\end{remark}

\begin{remark}
The sequence $\vx_u=u^\gamma\vy$ implicitly employed in Algorithm~\ref{alg:find_scaling_exponents} need not lie within the feasible set $\mathcal{X}$ for all $u$.
However, based on the observation that the scaling exponents $(\gamma, \rho)$ and the limit function $g^*$ characterize the asymptotic structure of the function $g$ itself, Algorithm~\ref{alg:find_scaling_exponents} serves as a heuristic to identify these parameters.
The validity with respect to all other feasible sequences is discussed below.
\end{remark}

\medskip
\noindent
\textbf{Part (ii)}.
Algorithm~\ref{alg:find_scaling_exponents} identifies the scaling exponents $(\gamma,\rho)$ and the continuous limit function $g^*$ by verifying that the limit in~\eqref{eq:direct_conti_conv} holds for the canonical sequences $\vx_u=u^\gamma\vy$ and $\vz_u=u\vw$ for every $(\vy,\vw)\in\mathcal{X}^\infty_\gamma\times\supportset$.
However, condition~\ref{Eq:Conti-Convergence} requires that the convergence in~\eqref{eq:direct_conti_conv} holds for \textit{every} pair of sequences $\{\vx_u\}\subset\mathcal{X}$ and $\{\vz_u\}\subset\supportset$ satisfying~\eqref{eq:relation_x_and_y}, not just the canonical choice. 

While verifying this requirement for all such sequences is generally impractical,
we demonstrate that for the class of algebraic functions defined in~\eqref{eq:algebraic_function}, this strong convergence property is inherent.
The following proposition establishes that the successful identification of $(\gamma,\rho,g^*)$ via Algorithm~\ref{alg:find_scaling_exponents} is sufficient to guarantee the convergence required by~\ref{Eq:Conti-Convergence}.

\begin{proposition}\label{prop:SufficientConditionsOfContiConv}
Suppose the constraint function $g(\vx, \vz)$ is of the form~\eqref{eq:algebraic_function}.
If Algorithm~\ref{alg:find_scaling_exponents} returns a tuple $(\gamma, \rho, g^*)$, then Assumption~\ref{assum:conti-conv} is satisfied. 
\end{proposition}

\paragraph{Illustrative Examples.} Below, we illustrate that Assumption~\ref{assum:conti-conv} encompasses 
representative instances from diverse classes of constraint functions common in \ac{cco}.
These classes include linear, quadratic~\citep{lejeune2016solving}, polynomial~\citep{jasour2015semidefinite, lasserre2021distributionally} and posynomial constraints found in geometric programming~\citep{liu2022distributionally,fontem2023robust}. 

\begin{example}[Linear Constraint]
Consider the function $g(\vx,\vz)=-\va^\top \vx+\vb^\top \vz+\ell$ defined on the feasible region $\mathcal{X}=\mathbb{R}^n_{\ge0}$ and support set $\supportset=\mathbb{R}^m_{\ge0}$.
With strictly positive coefficient vectors $\va\in\mathbb{R}^n_{>0}$ and $\vb\in\mathbb{R}^m_{>0}$, this function satisfies Assumption~\ref{assum:conti-conv} with exponents $(\gamma,\rho)=(1,1)$, yielding the limit function $g^*(\vy,\vw)=-\va^\top \vy+\vb^\top \vw$. This structure arises in resource allocation~\citep{luedtke2008sample} and inventory management~\citep{zhang2023new} problems.
\end{example}

\begin{example}[Bilinear Constraint]
Let $g(\vx,\vz)=\vx^\top A \vz - \ell$ where $A\in\mathbb{R}^{n\times m}$ is a non-zero matrix and $\ell > 0$ is a scalar coefficient. 
We define the support set as $\supportset=\mathbb{R}^m$ and the feasible set $\mathcal{X} \subset \mathbb{R}^n$ as a compact set containing the origin.
This function satisfies Assumption~\ref{assum:conti-conv} with scaling exponents $(\gamma,\rho)=(-1,0)$ and the limit function $g^*(\vy,\vw)=\vy^\top A \vw - \ell$. This formulation frequently appears in Value-at-Risk constraints~\citep{blanchet2024efficient} and robust classification~\citep{wang2018robust}.
\end{example}

\begin{example}[Quadratic Constraint]
Consider the quadratic form $g(\vx,\vz)=-\vx^\top A\vx + \vz^\top B\vz + \ell$ defined on $\mathcal{X}=\mathbb{R}^n_{\ge0}$ and $\supportset=\mathbb{R}^m_{\ge0}$. Assuming $A\in\mathbb{R}^{n\times n}$ is positive definite and $B\in\mathbb{R}^{m\times m}$ is positive semidefinite, this function satisfies Assumption~\ref{assum:conti-conv} with $(\gamma,\rho)=(1,2)$. The corresponding limit function is $g^*(\vy,\vw)=-\vy^\top A\vy + \vw^\top B\vw$. Applications include portfolio optimization~\citep{pagnoncelli2009sample} and power system planning~\citep{qiu2014chance}.
\end{example}

\begin{example}[High-order Polynomial Constraint]
    Consider a constraint with high-order nonlinearities defined by $g(\vx,\vz) = \sum_{i=1}^n \vz_i^2 \vx_i^3 - \ell$,
    with $\ell> 0$. 
    Let $\mathcal{X} \subset \mathbb{R}^n_{\ge 0}$ be a compact set containing the origin and $\supportset = \mathbb{R}^n$. 
    This function satisfies Assumption~\ref{assum:conti-conv} with exponents $(\gamma,\rho)=(-2/3,0)$, resulting in the limit function $g^*(\vy,\vw)= \sum_{i=1}^n \vw_i^2 \vy_i^3 - \ell$.
\end{example}

While the aforementioned examples are often directly amenable to optimization, the functions may sometimes be non-convex. 
The following result provides a way to address this issue while preserving the scaling structure.

\begin{proposition}[Invariance to multiplication]\label{Lem:AlgebraicalEquivalence}
Suppose there exists a function $h:\mathcal{X}\times\supportset\to\mathbb{R}_{>0}$ that satisfies condition~\ref{Eq:Conti-Convergence} with $(\gamma,\rho_h)$ and a strict positive limit function $h^*(\vy,\vw)$.
Then, for any function $f:\mathcal{X}\times\supportset\to\mathbb{R}$
satisfying Assumption~\ref{assum:conti-conv} with $(\gamma,\rho)$ and limit function $f^*(\vy,\vw)$,
the product
\begin{align}
    g(\vx,\vz)\coloneqq h(\vx,\vz)f(\vx,\vz) 
\end{align}
satisfies Assumption~\ref{assum:conti-conv} with $(\gamma,\rho+\rho_h)$ and limit function 
\begin{align*}
    g^*(\vy,\vw)\coloneqq h^*(\vy,\vw)f^*(\vy,\vw).
\end{align*}
\end{proposition}
Proposition~\ref{Lem:AlgebraicalEquivalence} implies that the parameter $\gamma$ is invariant under multiplication by some strictly positive function $h(\vx, \vz)$.
Since the decision-scaled~\eqref{eq:s-sp} depends exclusively on $\gamma$, its formulation remains identical under such positive multiplicative transformations of the constraint function.

The practical implication of this result is significant.
Since $h$ is strictly positive, the feasible regions $\{\vz:f(\vx,\vz)>0\}$ and $\{\vz:g(\vx,\vz)>0\}$ coincide.
Therefore, we can work with $g$ instead of $f$ by choosing an appropriate $h$ that ensures convexity or other desired properties.
Notably, $h$ need only satisfy condition~\ref{Eq:Conti-Convergence} of Assumption~\ref{assum:conti-conv}, not the entire assumption, which provides additional flexibility in implementation.

\begin{example}[Illustration of Proposition~\ref{Lem:AlgebraicalEquivalence}]
Consider the chance constraint
\begin{align*}
    \mathbb{P}_{\vxi}(\{\vz:f(\vx,\vz)\le0\})\ge 1-\varepsilon,
\end{align*}
where the constraint function 
\begin{align}
    f(\vx,\vz)=\frac{\vx^\top\vz-1}{1+\|\vx\|^2}
\end{align}
is defined on $\mathcal{X}=\mathcal{B}_{R}$ for some $R>0$ and $\supportset=\mathbb{R}^n$.
This function satisfies Assumption~~\ref{assum:conti-conv} with $(\gamma,\rho)=(-1,0)$ and limit function $f^*(\vy,\vw)=\vy^\top\vw-1$. 

However, the mapping $\vx\mapsto f(\vx,\vz)$ is nonconvex, violating our preliminary conditions and posing computational challenges.
To address this, multiply by
\begin{align}
    h(\vx,\vz)=1+\|\vx\|^2,
\end{align}
which is strictly positive on $\mathcal{X}$ and satisfies condition~\ref{Eq:Conti-Convergence} with
$(\gamma,\rho_h)=(-1,0)$ and $h^*(\vy,\vw)=1$.
Note that since only~\ref{Eq:Conti-Convergence} is required, this choice is not unique and can be adjusted as needed.
Then, 
\begin{align}  g(\vx,\vz)=f(\vx,\vz)h(\vx,\vz)=\vx^\top\vz-1
\end{align}
satisfies Assumption~\ref{assum:conti-conv} with $(\gamma,\rho+\rho_h)=(-1,0)$ and limit function $g^*=f^*$.
Crucially, the multiplication induces convexity of $g$ in $\vx$ while preserving $\gamma$.
\end{example}

\subsection{Extension via Constraint Decomposition}

We now extend our framework to constraint functions that do not strictly adhere to the algebraic form in \eqref{eq:algebraic_function}, such as those involving trigonometric, logarithmic, and exponential terms.
This is achieved via a \emph{decomposition} that separates a dominant algebraic component from an asymptotically negligible residual.
Specifically, consider a function decomposed as
\begin{align}\label{eq:decomposition}
    g(\vx,\vz)=g_0(\vx,\vz)+h(\vx,\vz).
\end{align}
Here, $g_0$ represents a dominant term satisfying the algebraic structure in~\eqref{eq:algebraic_function}.  
We assume that applying Algorithm~\ref{alg:find_scaling_exponents} to $g_0$ yields exponents $(\gamma, \rho)$ and limit function $g_0^*$, thereby satisfying Assumption~\ref{assum:conti-conv}.

The term $h$ represents a residual component that is \emph{asymptotically negligible} relative to the scaling of $g_0$. 
Formally, we require that for any $\vy \in \mathcal{X}_\gamma^\infty$ and $\vw \in \supportset$, and for any sequences $\{\vx_u\}\subset\mathcal{X}$ and $\{\vz_u\}\subset\supportset$ such that $\lim_{u\to\infty}\vx_u/u^\gamma=\vy$ and $\lim_{u\to\infty}\vz_u/u=\vw$, 
\begin{align}\label{eq:h_decay}
\lim_{u\to\infty}\frac{h(\vx_u, \vz_u)}{u^\rho} = 0.
\end{align}
Under this condition, the limit of the original function $g$ is determined solely by the dominant term:
\begin{align}
\lim_{u\to\infty}\frac{g(\vx_u, \vz_u)}{u^\rho}
= \lim_{u\to\infty}\frac{g_0(\vx_u, \vz_u)}{u^\rho} + \lim_{u\to\infty}\frac{h(\vx_u, \vz_u)}{u^\rho} %
= g_0^*(\vy, \vw).
\end{align}
Consequently, $g$ satisfies Assumption~\ref{assum:conti-conv} with the same parameters $(\gamma, \rho)$ and limit function $g^* = g_0^*$.
This implies that the applicability of our framework relies on the constraint's asymptotic behavior rather than its strict algebraic form.
Thus, for non-algebraic functions, we can verify applicability by examining only the dominant algebraic component $g_0$. 
We illustrate this through the following examples.
\begin{example}\label{ex:lse}
    Consider the function defined on $\mathcal{X}=\mathbb{R}^n_{\ge0}$ and $\supportset=\mathbb{R}^n_{\ge0}$:
    \begin{align}
        g(\vx,\vz)=\log\left(\sum_{i=1}^n \exp(-x_i+z_i)\right),
    \end{align}
    where $\vx=(x_1,\ldots,x_n)^\top$ and $\vz=(z_1,\ldots,z_n)^\top$.
    It can be decomposed into
    \begin{align}
        g(\vx,\vz)=\underbrace{\max_{i=1,\ldots,n}(-x_i+z_i)}_{=g_0(\vx,\vz)}
        +
        \underbrace{\log\left(\sum_{i=1}^n\exp\left(-x_i+z_i-\max_{i=1,\ldots,n}(-x_i+z_i)\right)\right)}_{h(\vx,\vz)}
    \end{align}
    The dominant term $g_0$ satisfies Assumption~\ref{assum:conti-conv} with $(\gamma,\rho)=(1,1)$ and limit function $g^*_0(\vy,\vw)=\max_{i=1,\ldots,n}(-y_i+w_i)$.
    Since $0\le h(\vx,\vz)\le\log n$,
    the residual $h(\vx,\vz)$ is asymptotically negligible.
    Therefore, $g$ satisfies Assumption~\ref{assum:conti-conv} with $(\gamma,\rho)=(1,1)$ and $g^*=g^*_0$.
    Constraint functions of this type, known as log-sum-exp functions, appear in control~\citep{chiang2017power} and structural system reliability~\citep{aoues2010benchmark}.
\end{example}

\begin{example}
    Consider a constraint function involving trigonometric, logarithmic, and exponential terms on $\mathcal{X}=\mathbb{R}^n_{\ge0}$ and $\supportset=\mathbb{R}^m_{\ge0}$:
    \begin{align}
    g(\vx,\vz)=-\va^\top\vx+\vb^\top\vz+\sin(\vz^\top\vz)+\log(1 + \exp(-\|\vz\|)),   
    \end{align}
    where $\va\in\mathbb{R}^n_{>0}$ and $\vb\in\mathbb{R}^m_{>0}$ are coefficient vectors.
    This function admits the following decomposition:
    \begin{align}
        g(\vx,\vz)=
        \underbrace{-\va^\top\vx+\vb^\top\vz}_{=g_0(\vx,\vz)}
        +\underbrace{\sin(\vz^\top\vz)+\log(1 + \exp(-\|\vz\|))}_{=h(\vx,\vz)}.
    \end{align}
    In this case, $g_0$ satisfies Assumption~\ref{assum:conti-conv} with $(\gamma,\rho)=(1,1)$ and limit function $g_0^*(\vy,\vw)=-\va^\top\vy+\vb^\top\vw$, while $h$ is asymptotically negligible.
    Consequently, $g$ satisfies Assumption~\ref{assum:conti-conv} with $(\gamma,\rho)=(1,1)$ and $g^*=g^*_0$.
    Practical applications of such constraints include structural topology optimization~\citep{chun2016structural} and trajectory planning~\citep{wang2016multi}.
\end{example}

\subsection{Extension to Joint Chance Constraints}
The applicability of our framework extends naturally to joint chance constraints of the form
\begin{align}\label{eqn:JCC}
    \mathbb{P}_{\vxi}(\{\vz:g_i(\vx,\vz)\leq 0,\ \forall i=1,\ldots,K\}) \geq 1-\varepsilon,
\end{align}
where $K\in\mathbb{N}$.
This can be equivalently reformulated as a single chance constraint using the function
\begin{align}
    g(\vx,\vz)\coloneqq\max_{i=1,\ldots,K}g_i(\vx,\vz).
\end{align}
The following lemma establishes that if the individual constraints $g_i$, $i=1,\ldots,K$, satisfy Assumption~\ref{assum:conti-conv}, the constraint $g$ does as well.

\begin{proposition}\label{lem:joint_chance_constraint}
Suppose that each function $g_i$,  $i=1,\ldots,K$, satisfies Assumption~\ref{assum:conti-conv} with the same scaling exponents $(\gamma, \rho)$ and limit function $g_i^*$. 
Then, the function 
\begin{align}
g(\vx,\vz) = \max_{i=1,\ldots,K}g_i(\vx,\vz)   
\end{align}
satisfies Assumption~\ref{assum:conti-conv} with parameters $(\gamma, \rho)$ and the limit function 
\begin{align}
g^*(\vy,\vw) \coloneqq \max_{i=1,\ldots,K}g_i^*(\vy,\vw).    
\end{align}
\end{proposition}

\begin{example}[Linear Joint Chance Constraint]
Consider a joint chance constraint of the form $\mathbb{P}_{\vxi}(\{ \vz : A\vx \ge B\vz \}) \ge 1-\varepsilon$, where $A \in \mathbb{R}^{K \times n}$ and $B \in \mathbb{R}^{K \times m}$ are coefficient matrices. 
Let the feasible region be $\mathcal{X}=\mathbb{R}^n_{\ge0}$ and the support be $\supportset=\mathbb{R}^m_{\ge0}$.
The condition $A\vx \ge B\vz$ requires satisfying $K$ linear inequalities simultaneously:
\begin{align*}
(B\vz - A\vx)_i \le 0, \quad\forall i=1,\dots,K.    
\end{align*}
We can reformulate the problem using a single constraint function $g(\vx, \vz) := \max_{i=1,\dots,K} g_i(\vx, \vz)$, where each component is defined as $g_i(\vx,\vz) = - \va_i^\top \vx+\vb_i^\top \vz $ with $\va_i^\top$ and $\vb_i^\top$ denoting the $i$-th rows of $A$ and $B$, respectively.

Since each $g_i$ satisfies Assumption~\ref{assum:conti-conv} with exponents $(\gamma, \rho) = (1, 1)$ and limit function $g_i^*(\vy, \vw) = - \va_i^\top \vy+\vb_i^\top \vw $, Proposition~\ref{lem:joint_chance_constraint} guarantees that $g$ also satisfies Assumption~\ref{assum:conti-conv} with the same exponents and limit function $g^*(\vy, \vw) = \max_{i=1,\dots,K} ( - \va_i^\top \vy+\vb_i^\top \vw)$.
This formulation represents the canonical form of linear joint chance constraints in \ac{cco}~\citep{shapiro2021lectures}.
\end{example}

\section{Numerical Experiments}\label{Sec:NumExp}
In this section, we numerically evaluate the performance of our proposed decision-scaled \ac{sa} against the classical \ac{sa}. We demonstrate the efficacy of our method on three existing benchmark problems from the literature: portfolio optimization, reliability-based short column design, and norm optimization.

Throughout all experiments, we maintain a confidence level of $\beta = 0.01$ and evaluate the decision-scaling approach across three rare-event risk tolerance levels: $\varepsilon\in\{10^{-3},10^{-4},10^{-5}\}$.
To ensure a fair comparison across the three benchmarks, we select the scaling parameter $s$ such that $s^\alpha\in\{1.1, 1.2\}$, where $\alpha$ is the tail index of the respective problem's underlying distribution.
Since our approach reduces sample complexity from $O(\varepsilon^{-1})$ to $O(\varepsilon^{-1/s^\alpha})$, fixing $s^\alpha$ guarantees an identical polynomial reduction in required sample sizes across all three experiments.
For each configuration, the deterministic sample size requirement is presented via line graphs.
To ensure statistical validity, we perform 100 independent trials per configuration to account for sampling variability. 
The experimental results are visualized using line plots, where markers denote the median value across these trials and error bars represent the interquartile range (the 25th and 75th percentiles).
Out-of-sample validation is conducted via Monte Carlo simulation using ${10^4}/{\varepsilon}$ independent samples to provide sufficient statistical precision for estimating violation probabilities at each risk tolerance level.

The numerical experiments were implemented in Julia~1.11.2 using the JuMP~1.30.0 modeling package.
The portfolio optimization problems were solved using Gurobi~13.0.0, while the short column and the norm optimization problems were solved using Mosek~11.0.27. 
A computational time limit of one~hour was imposed on each run. 
All computations were executed on high-performance computing nodes equipped with 2.90~GHz Intel Xeon Gold 6226R CPU and 32~GB of RAM.
The source code and data used to reproduce the numerical experiments are available at \url{https://github.com/Subramanyam-Lab/Decision-Scaled-Scenario-Approach}.

\subsection{Portfolio Optimization}\label{NumExp:Portfolio}
Our first benchmark is a portfolio optimization problem studied by \citet{xie2020bicriteria, blanchet2024efficient}. The goal is to allocate capital across $n$ different assets to maximize the total return while managing downside risk.

Let $\vx=(x_1,\ldots,x_n)^\top\in\mathbb{R}^n_{\ge 0}$ represent the amount of capital invested in each asset.
For each asset $i=1,\ldots,n$, we associate an expected return per dollar invested, $\mu_i$, and a non-negative random loss per dollar invested $\xi_i$. 
Let $\vmu=(\mu_1,\ldots,\mu_n)^\top$ and $\vxi=(\xi_1,\ldots,\xi_n)^\top$ be the corresponding vectors of expected returns and random losses.
The total expected return of the portfolio is thus $\vmu^\top \vx$, and the total random loss is $\vx^\top\vxi $.

We aim to maximize the portfolio's expected return subject to a Value-at-Risk constraint. This constraint requires that the probability of the total portfolio loss $\vx^\top\vxi$ exceeding a predefined threshold $\eta>0$ must be no greater than a small risk tolerance level $\varepsilon$. This leads to the following \ac{cco} problem:
\begin{equation}
\label{eq:portfolio}
\begin{aligned}
\maximize_{\vx\in \mathbb{R}^{n}_{\ge0}} &\quad \vmu^{\top}\vx\\
\text{subject to} &\quad \mathbb{P}_{\vxi}(\{\vz:\vx^\top\vz \leq \eta\}) \geq 1-\epsilon.
\end{aligned}
\end{equation}

The problem's constraint function is
$g(\vx,\vz)\coloneqq\vx^\top\vz-\eta.$
This function satisfies Assumption~\ref{assum:conti-conv} with scaling exponents $\gamma=-1$ and $\rho=0$, which yields the limit function $g^*(\vy,\vw)=\vy^\top\vw-\eta$.
The corresponding decision-scaled scenario problem \eqref{eq:s-sp} for \eqref{eq:portfolio} with parameter $s\ge1$ is
\begin{equation}
\label{eq:ssp-portfolio}
\begin{aligned}
\maximize_{\vx \in \mathbb{R}^n_{\ge0}} &\quad  \vmu^\top \vx \\
\text{subject to} &\quad   \vx^\top\vz^{(j)} \le \frac{\eta}{s}, \quad \forall j=1, \ldots, N,
\end{aligned}
\end{equation}
where the number of samples is $N=\mathsf{N}\left(\varepsilon^{1/s^{\alpha}},\beta\right)$.

\paragraph{Test setup.}
We construct a test instance with $n=20$ assets as follows. 
\begin{itemize}[label=\textbullet]
    \item The components of the expected return vector, $\mu_i$, are drawn independently from a uniform distribution, $\mu_i\sim\operatorname{Uniform}[1, 3]$.
    
    \item The non-negative random loss $\vxi$ is
    assumed to follow an independent Weibull distribution, aligning with standard practices in financial risk management~\citep{malevergne2006extreme}.
    This choice of distribution satisfies Assumption~\ref{assum:light-tail}.

    \begin{itemize}
        \item The shape parameter is set to $k_i=0.9$ for all $i=1,\ldots,n$, which corresponds to a tail index of $\alpha=0.9$.
        
        \item The scale parameters, $\sigma_i$, are drawn independently from a uniform distribution, $\sigma_i\sim\operatorname{Uniform}[2, 10]$.
    \end{itemize}

    \item The total loss threshold is set to  $\eta=1000$.
\end{itemize}

\paragraph{Numerical results.}
\begin{figure}[htbp]
    \centering
    \begin{subfigure}[t]{0.495\textwidth}
        \centering
        \includegraphics[width=\textwidth]{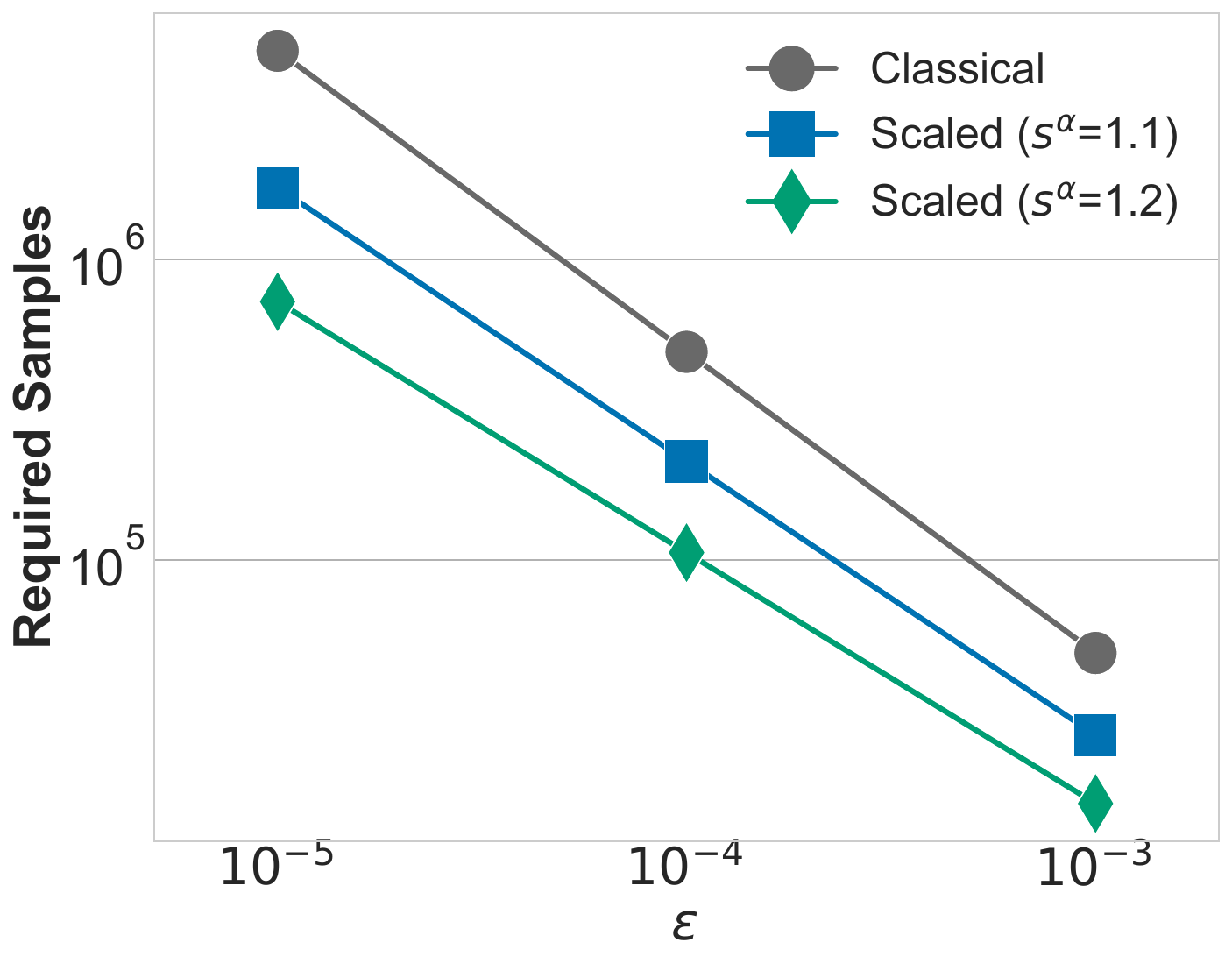}
    \end{subfigure}
    \hfill
    \begin{subfigure}[t]{0.495\textwidth}
        \centering
        \includegraphics[width=\textwidth]{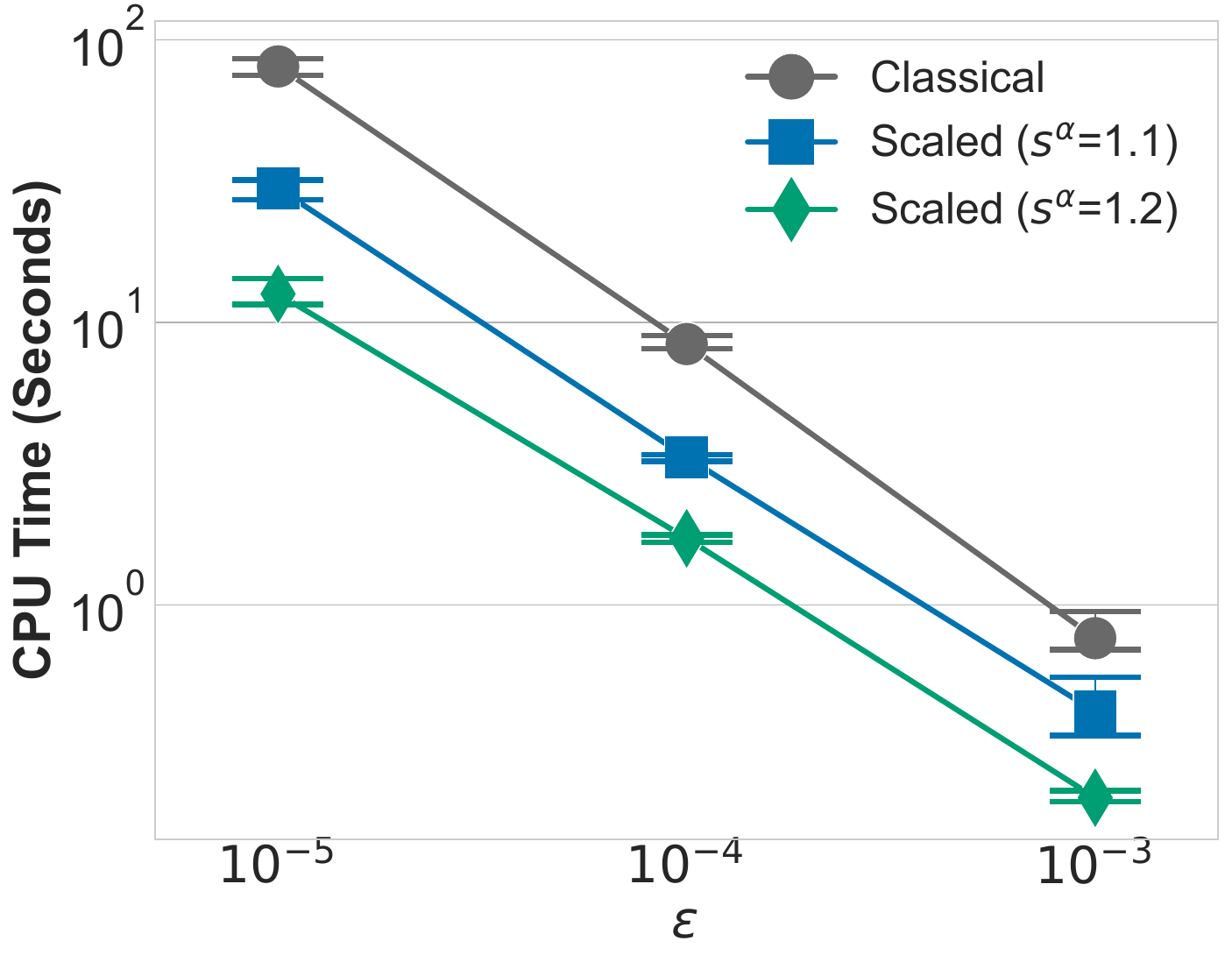}
    \end{subfigure}
    
    \caption{Comparison of computational efficiency for the portfolio optimization problem: (left) Required sample size; (right) CPU time.}
    \label{fig:portfolio_efficiency}
\end{figure}

Figure~\ref{fig:portfolio_efficiency} demonstrates computational efficiency of our proposed decision-scaled \ac{sa}. The left panel illustrates that for any given risk level $\varepsilon$ our approach requires substantially fewer samples than the classical method. This reduction in sample size directly translates to a decrease in computation time, as shown in the right panel, where the efficiency gain increases with the scaling factor $s$.

\begin{figure}[htbp]
    \centering
    \begin{subfigure}[t]{0.495\textwidth}
        \centering
        \includegraphics[width=\textwidth]{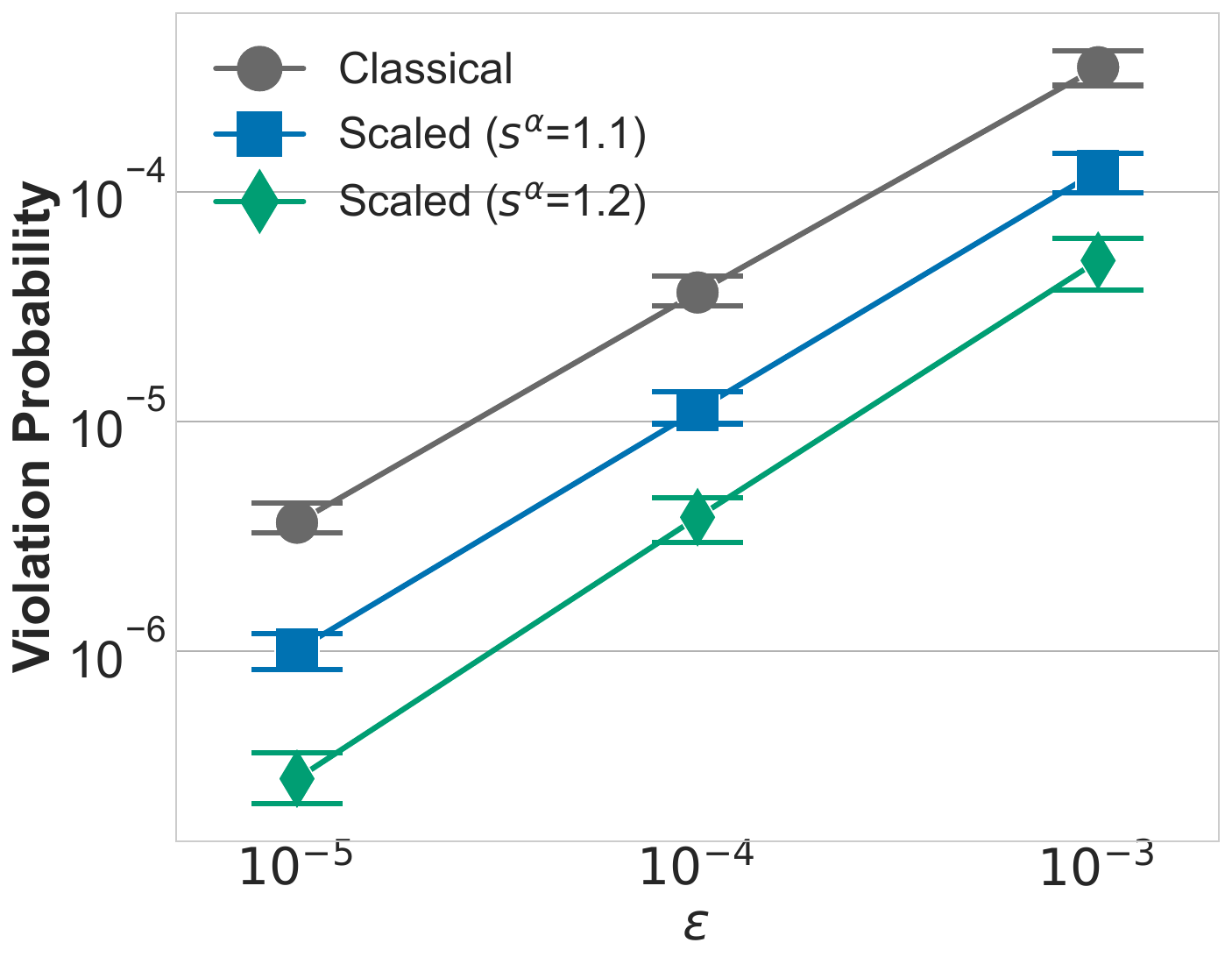}
    \end{subfigure}
    \hfill
    \begin{subfigure}[t]{0.495\textwidth}
        \centering
        \includegraphics[width=\textwidth]{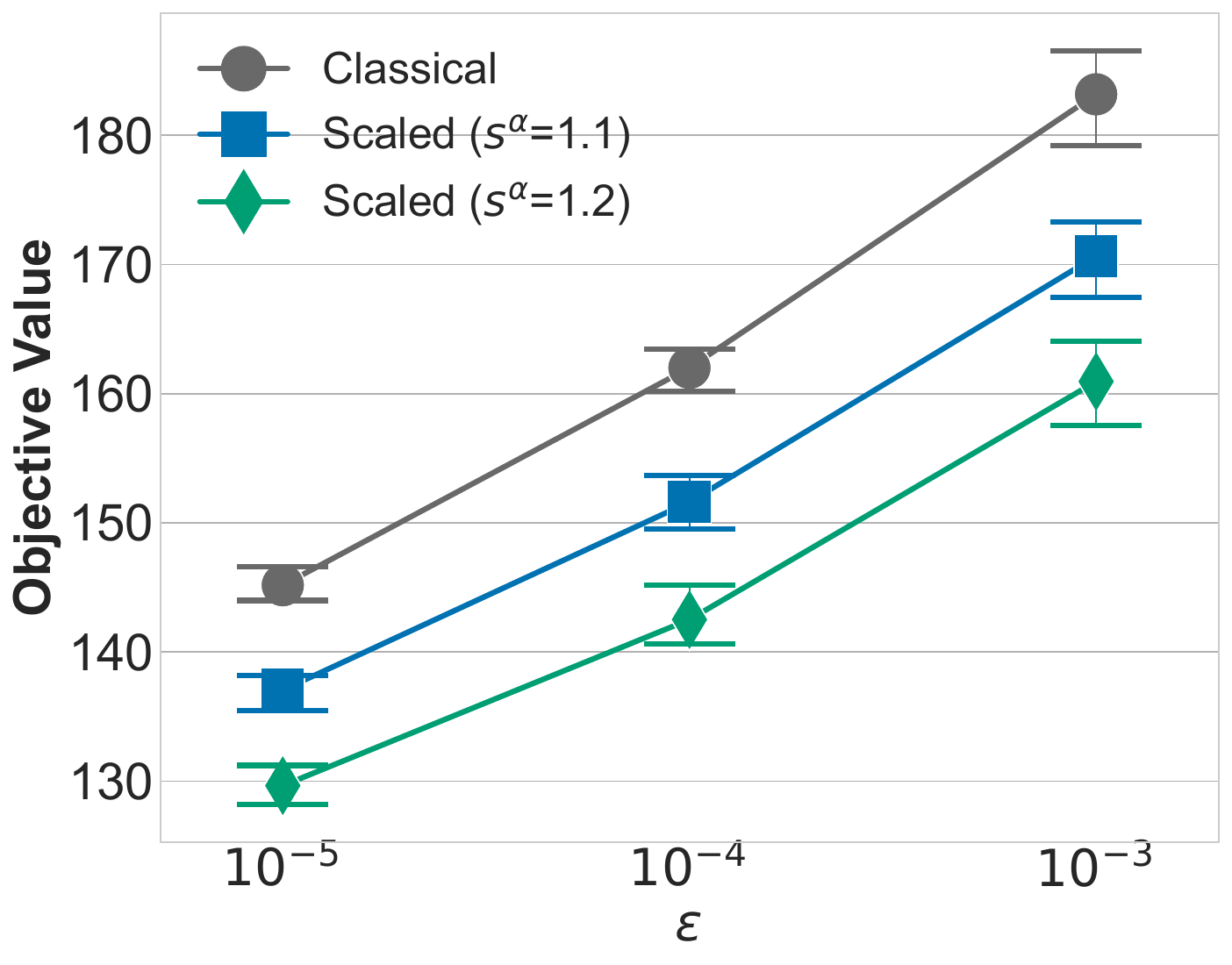}
    \end{subfigure}
    
    \caption{Comparison of solution quality for the portfolio optimization problem: (left) Objective value; (right) Violation probability.}
    \label{fig:portfolio_quality}
\end{figure}
Figure~\ref{fig:portfolio_quality} examines the quality of the obtained solutions. 
The left panel presents the estimated violation probability. All methods reliably produce solutions with estimated violation probabilities below the target level $\varepsilon$.
However, our proposed methodology yields solutions that are slightly more conservative, and this tendency increases with the scaling factor $s$. As shown in the right panel, this increased conservatism results in a marginally lower objective value, representing the computational efficiency premium.

The scaling factor $s$ serves as a tunable hyperparameter. It offers practitioners a clear trade-off, allowing them to strike a balance between computational efficiency and solution conservatism to best suit their specific requirements.

\subsection{Reliability-based short column design}\label{Numexp_shortcolumn}
Next, we consider
the short column design problem, a classic benchmark in reliability-based design~\citep{aoues2010benchmark}, %
following the formulation from~\citet{tong2022optimization}.

The objective is to determine the optimal dimensions of a short column with a rectangular cross-section. The decision vector is $\vx=(x_w,x_h)^\top$, where $x_w$ and $x_h$ are the width and height of the cross-section, respectively. The column is subjected to uncertain loads, represented by the random vector $\vxi=(\xi_M,\xi_F)^\top$, where $\xi_M\ge 0$ is the bending moment and $\xi_F\ge0$ is the axial force. The material's yield stress, $C_Y\ge0$, is a known constant. 

The design goal is to minimize the cross-sectional area $x_wx_h$, while ensuring that the probability of material failure does not exceed a small tolerance $\varepsilon$. To ensure a physically realistic design, the dimensions are constrained to lower bounds $L_w>0$ and $L_h>0$, respectively.
The resulting \ac{cco} problem is formulated as follows.
\begin{equation}
\label{eq:shortcolumn}
\begin{aligned}
\minimize_{\vx} &\quad x_wx_h\\
\text{subject to} &\quad \mathbb{P}_{\vxi}\left(\left\{(z_M,z_F)^\top:\frac{4z_M}{C_Yx_wx_h^2}+\frac{z_F^2}{C_Y^2x_w^2x_h^2} \leq 1\right\}\right) \geq 1-\varepsilon,\\
&\quad x_w \ge L_w, \\
&\quad  x_h \ge L_h.
\end{aligned}
\end{equation}

We apply a change of variables:
$x_w=\exp(\tilde{x}_w), x_h=\exp(\tilde{x}_h), z_M=\exp(\tilde{z}_M)$, and $z_F=\exp(\tilde{z}_F)$.
Applying this transformation to the inequality inside the chance constraint yields:
\begin{align}
    \frac{4\exp{(\tilde{z}_M-\tilde{x}_w-2\tilde{x}_h)}}{C_Y} + \frac{\exp{(2\tilde{z}_F-2\tilde{x}_w-2\tilde{x}_h)}}{C_Y^2}\le 1.
\end{align}
Taking the logarithm on both sides and letting $\tilde{\vx}=(\tilde{x}_w, \tilde{x}_h)^\top$ and $\tilde{\vz}=(\tilde{z}_M, \tilde{z}_F)^\top$, 
the corresponding constraint function can be defined as
\begin{align}
g(\tilde{\vx},\tilde{\vz})\coloneqq\log\left(\frac{4\exp{(\tilde{z}_M-\tilde{x}_w-2\tilde{x}_h)}}{C_Y} + \frac{\exp{(2\tilde{z}_F-2\tilde{x}_w-2\tilde{x}_h)}}{C_Y^2}\right).
\end{align}
Following Example~\ref{ex:lse}, this constraint function $g(\tilde{\vx},\tilde{\vz})$ satisfies Assumption~\ref{assum:conti-conv} with scaling exponents $(\gamma,\rho)=(1,1)$ and the limit function $g^*(\vy,\vw)=\max\{w_M-y_w-2y_h, 2w_F-2y_w-2y_h\}$.

Consequently, applying our decision-scaling method yields the scaled scenario problem with a parameter $s\ge1$.
Introducing auxiliary variables $u_M^{(j)}$ and $u_F^{(j)}$ to ensure compatibility with standard exponential cone solvers, we solve the following optimization problem:
\begin{equation}\label{eq:short_column_cone}
\begin{alignedat}{2}
\minimize_{\tilde{\vx},u_M^{(j)},u_F^{(j)}} &\quad \tilde{x}_w+\tilde{x}_h &&\\
\text{subject to} 
&\quad \exp{\left(\log\frac{4}{C_Y}+\tilde{z}_M^{(j)}-\frac{(\tilde{x}_w+2\tilde{x}_h)}{s}\right)}\le u_M^{(j)},&&\quad\forall j=1,\ldots,N,\\
&\quad \exp{\left(\log\frac{1}{C_Y^2}+2\tilde{z}_F^{(j)}-\frac{2(\tilde{x}_w+\tilde{x}_h)}{s}\right)}\le u_F^{(j)},&&\quad\forall j=1,\ldots,N,\\
&\quad u_M^{(j)} + u_F^{(j)} \le 1,&&\quad\forall j=1,
\ldots,N,\\
&\quad \tilde{x}_w \ge \log L_w,&&\\
&\quad  \tilde{x}_h \ge \log L_h,&&
\end{alignedat}
\end{equation}
where $N=\mathsf{N}\left(\varepsilon^{1/s^{\alpha}},\beta\right)$ as before.

\paragraph{Test setup.} We construct a test instance using the following parameters:
\begin{itemize}[label=\textbullet]
    \item The deterministic parameters are set to $C_Y=5$, $L_w=5$, and $L_h=15$. 
    
    \item The uncertain load vector $\vxi$ is assumed to follow a multivariate lognormal distribution. 
    Specifically, the log-transformed vector $(\log\xi_M, \log\xi_F)^\top$ follows a multivariate normal distribution with a mean vector $\vmu=[6, 7.5]^\top$ and a covariance matrix $\Sigma=\begin{bmatrix} 0.4 & 0.2 \\ 0.2 & 0.4 \end{bmatrix}$.
    
    \item This ensures that the underlying uncertainty in our reformulated problem satisfies Assumption~\ref{assum:light-tail} with a tail index of $\alpha=2$. 

\end{itemize}

\paragraph{Numerical Results.}
\begin{figure}[htbp]
    \centering
    \begin{subfigure}[t]{0.495\textwidth}
        \centering
        \includegraphics[width=\textwidth]{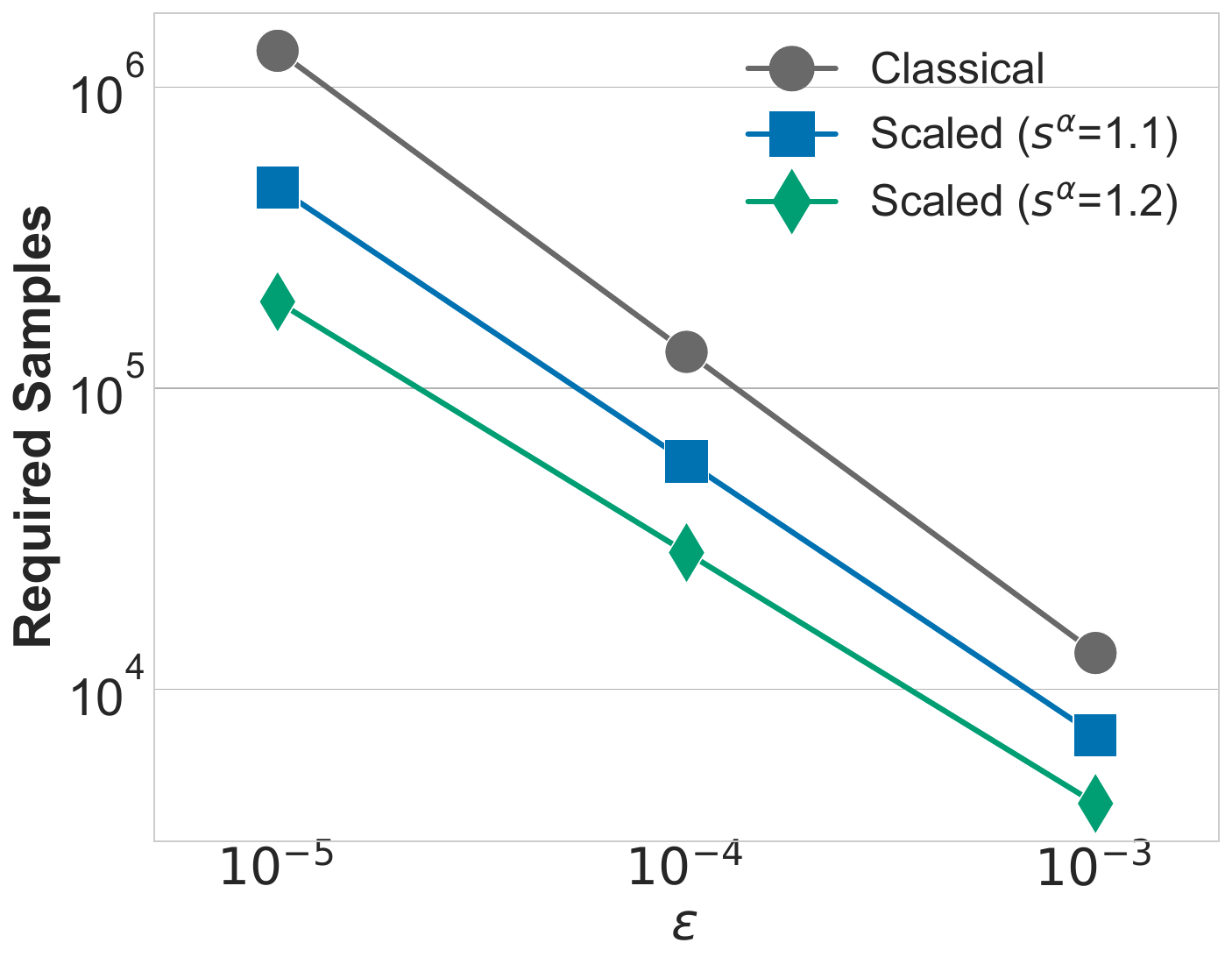}
    \end{subfigure}
    \hfill
    \begin{subfigure}[t]{0.495\textwidth}
        \centering
        \includegraphics[width=\textwidth]{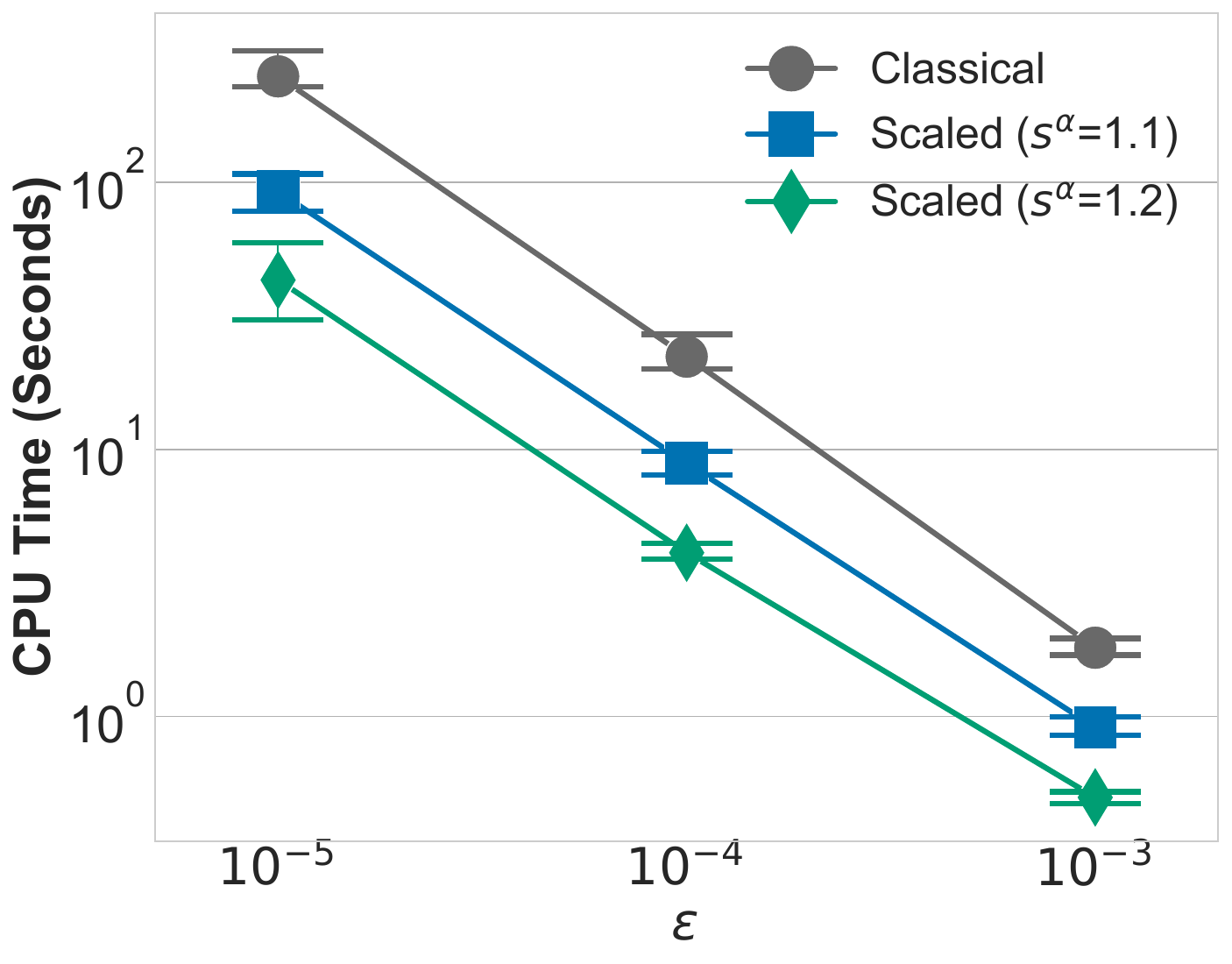}
    \end{subfigure}
    
    \caption{Comparison of computational efficiency for the short column design problem: (left) Required sample size; (right) CPU time.}
    \label{fig:shortcolumn_efficiency}
\end{figure}
Figure~\ref{fig:shortcolumn_efficiency} presents a comparison of the computational efficiency between the classical and decision-scaled \ac{sa}. Consistent with the findings in Section~\ref{NumExp:Portfolio}, the decision-scaled method is more computationally efficient. The left panel shows that for any given risk level $\varepsilon$, the scaled approaches require fewer samples than the classical method, with the required sample size decreasing as the scaling factor $s$ increases. This reduction in sample complexity directly translates to a decrease in computation time, as illustrated in the right panel. 

\begin{figure}[htbp]
    \centering
    \begin{subfigure}[t]{0.495\textwidth}
        \centering
        \includegraphics[width=\textwidth]{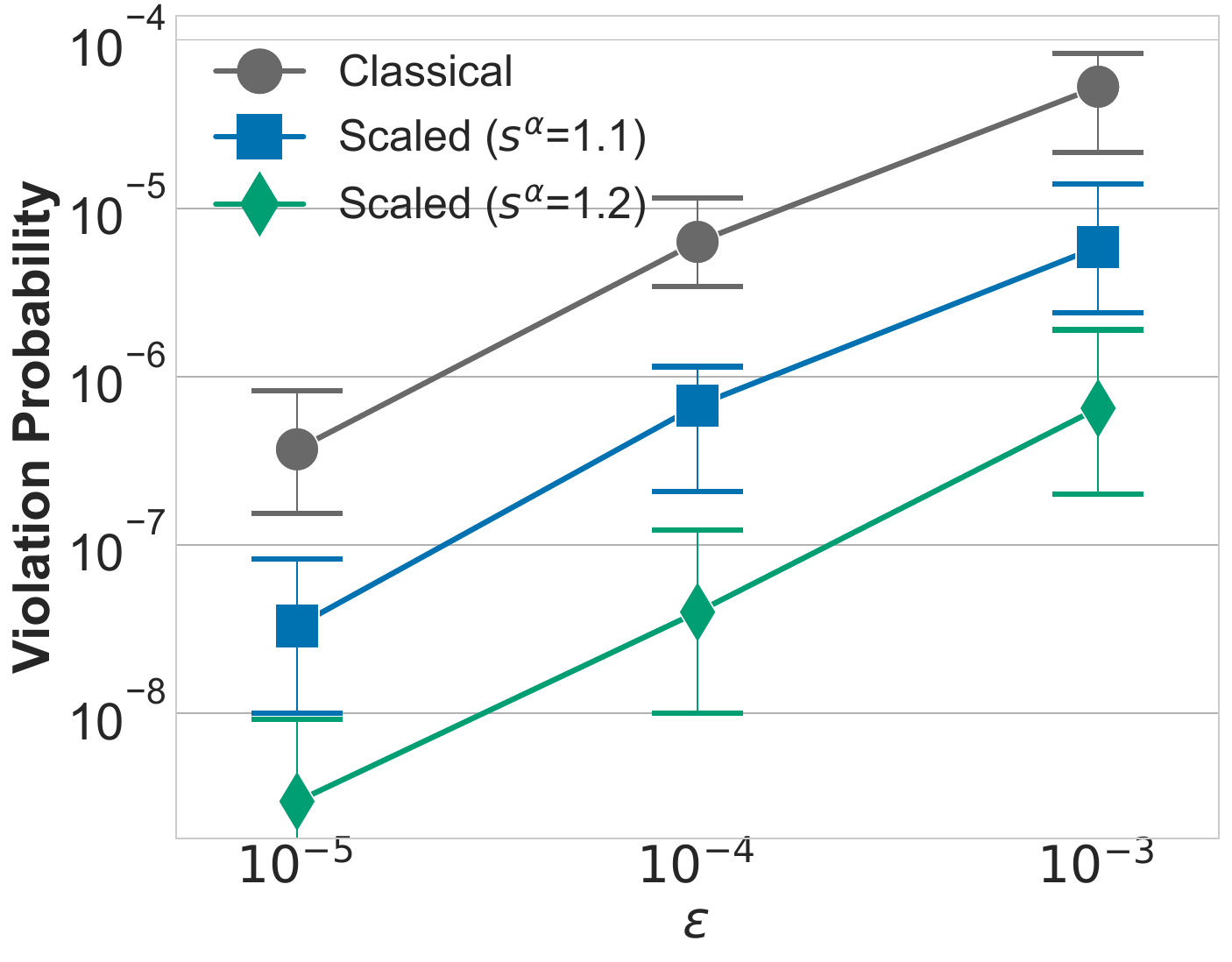}
    \end{subfigure}
    \hfill
    \begin{subfigure}[t]{0.495\textwidth}
        \centering
        \includegraphics[width=\textwidth]{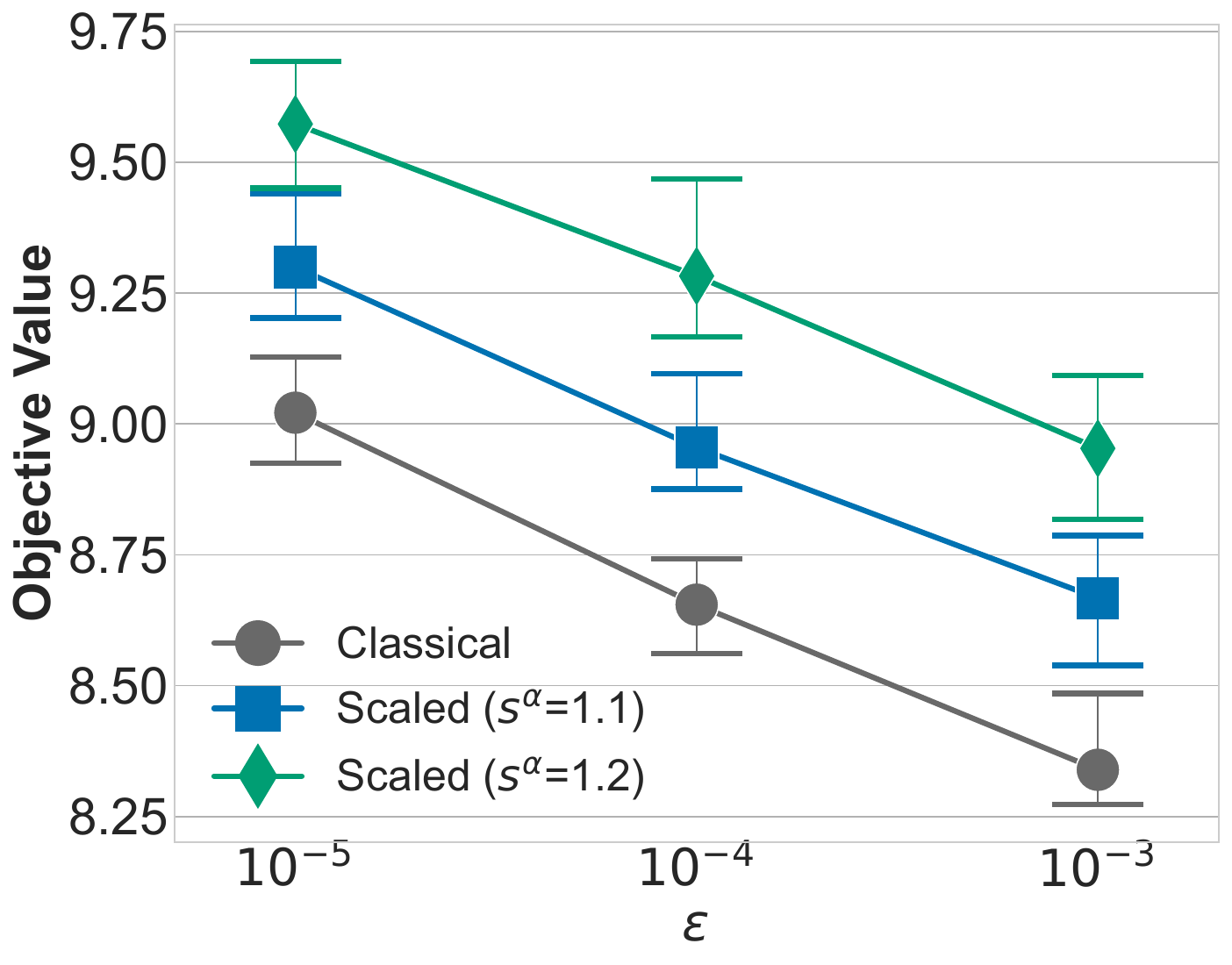}
    \end{subfigure}
    
    \caption{Comparison of solution quality for the short column problem: (left) Objective value; (right) Violation probability.}
    \label{fig:shortcolumn_quality}
\end{figure}
Figure~\ref{fig:shortcolumn_quality} evaluates the quality of the solutions obtained by each method. The left panel shows that all three methods generate conservative solutions, with estimated violation probabilities falling below the prescribed tolerance level $\varepsilon$ across all test cases. 
The decision-scaled method yields solutions that are more conservative than the classical approach, and this conservatism increases with the scaling factor $s$. Consequently, as shown in the right panel, this increased robustness results in a higher objective value for the scaled methods. This marginal cost in the objective function is the trade-off for the reduced computation time.

\subsection{Norm optimization}
Our final numerical experiment examine the norm optimization problem subject to a joint chance constraint, adapted from the formulation presented by~\citet{hong2011sequential}.

Let $\vx\in\mathbb{R}^n_{\ge 0}$ denote a decision vector, and let $\vxi\in\mathbb{R}^{m\times n}$ be a random matrix whose row vectors are defined as
$\vxi_i=(\xi_{i1},\ldots,\xi_{in})$ for $i=1,\ldots,m$.
The objective is to maximize the sum of the decision variables while ensuring that a set of $m$ quadratic inequalities holds simultaneously with a probability of at least $1-\varepsilon$.
The \ac{cco} problem is formulated as follows:
\begin{equation}
\begin{aligned}
\maximize_{\vx\in \mathbb{R}^{n}_{\ge0}} &\quad \vone^{\top}\vx\\
\text{subject to} &\quad \mathbb{P}_{\vxi}\left(\left\{\vz:\sum_{j=1}^n \vz_{ij}^2\vx_{j}^2 \leq 100,\quad\forall i=1,\ldots,m\right\}\right) \geq 1-\epsilon.
\end{aligned}
\end{equation}
This joint chance constraint can be equivalently reformulated using a single maximum function over the individual constraints:
\begin{equation}
\label{eq:norm_opt}
\begin{aligned}
\maximize_{\vx\in \mathbb{R}^{n}_{\ge0}} &\quad \vone^{\top}\vx\\
\text{subject to} &\quad \mathbb{P}_{\vxi}\left(\left\{\vz:\max_{i=1,\ldots,m}\sum_{j=1}^n \vz_{ij}^2\vx_{j}^2-100 \leq 0\right\}\right) \geq 1-\epsilon,
\end{aligned}
\end{equation}

We define the individual constraint functions as $g_i(\vx,\vz)\coloneqq\sum_{j=1}^n z_{ij}^2 x_j^2 - 100$ for $i=1,\ldots,m$. 
Each $g_i$ satisfies Assumption~\ref{assum:conti-conv} with scaling exponents $(\gamma,\rho)=(-1,0)$ and a limit function $g_i^*(\vy,\vw)=\sum_{j=1}^n w_{ij}^2 y_j^2-100$, where $\vy\in\mathbb{R}^n$ and $\vw\in\mathbb{R}^{m\times n}$.
By applying Proposition~\ref{lem:joint_chance_constraint},
the constraint function $g(\vx,\vz)\coloneqq\max_{i=1,\ldots,m}g_i(\vx,\vz)$ satisfies Assumption~\ref{assum:conti-conv} with identical scaling exponents $(\gamma,\rho)=(-1,0)$ and the limit function $g^*(\vy,\vw)=\max_{i=1,\ldots,m}g_i^*(\vw,\vz)$.

Therefore, applying our decision-scaled \ac{sa} with a parameter $s\ge1$ yields the following scenario problem:
\begin{equation}
\label{eq:norm_opt_ssp}
\begin{aligned}
\maximize_{\vx\in \mathbb{R}^{n}_{\ge0}} &\quad \vone^{\top}\vx\\
\text{subject to} &\quad s^2\sum_{j=1}^n {\left(\vz_{ij}^{(k)}\right)}^2\vx_{j}^2 \leq 100 ,\quad\forall i=1,\ldots,m,\quad\forall k=1,\ldots,N,
\end{aligned}
\end{equation}
where $N=\mathsf{N}\left(\varepsilon^{1/s^{\alpha}},\beta\right)$.

\paragraph{Test setup.}
We evaluate the methods on a test instance with $n =5$ variables and $m = 3$ constraints using the following parameters:
\begin{itemize}[label=\textbullet]
    \item Random variables $\xi_{ij}$ for $i=1,\ldots,m$ and $j=1,\ldots,n$ are assumed to follow a normal distribution with mean $j/n$ and variance 1.
    Moreover, $\operatorname{Cov}(\xi_{ij},\xi_{i'j})=0.5$ when $i\neq i'$ and $\operatorname{Cov}(\xi_{ij},\xi_{i'j'})=0$ when $j\neq j'$.

    \item This multivariate normal uncertainty satisfies Assumption~\ref{assum:light-tail} with a tail index of $\alpha=2$.
\end{itemize}

\paragraph{Numerical Results.}
\begin{figure}[htbp]
    \centering
    \begin{subfigure}[t]{0.495\textwidth}
        \centering
        \includegraphics[width=\textwidth]{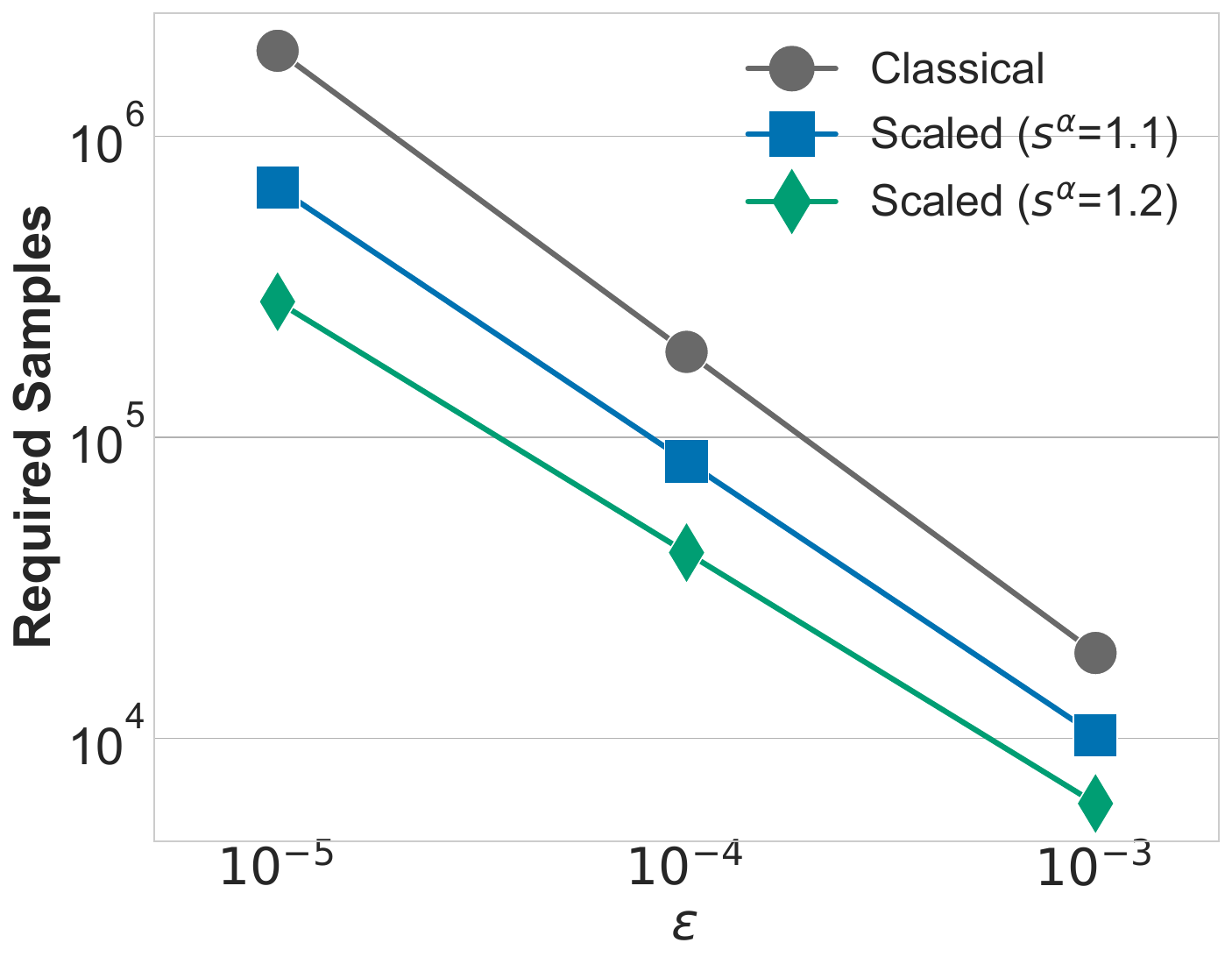}
    \end{subfigure}
    \hfill
    \begin{subfigure}[t]{0.495\textwidth}
        \centering
        \includegraphics[width=\textwidth]{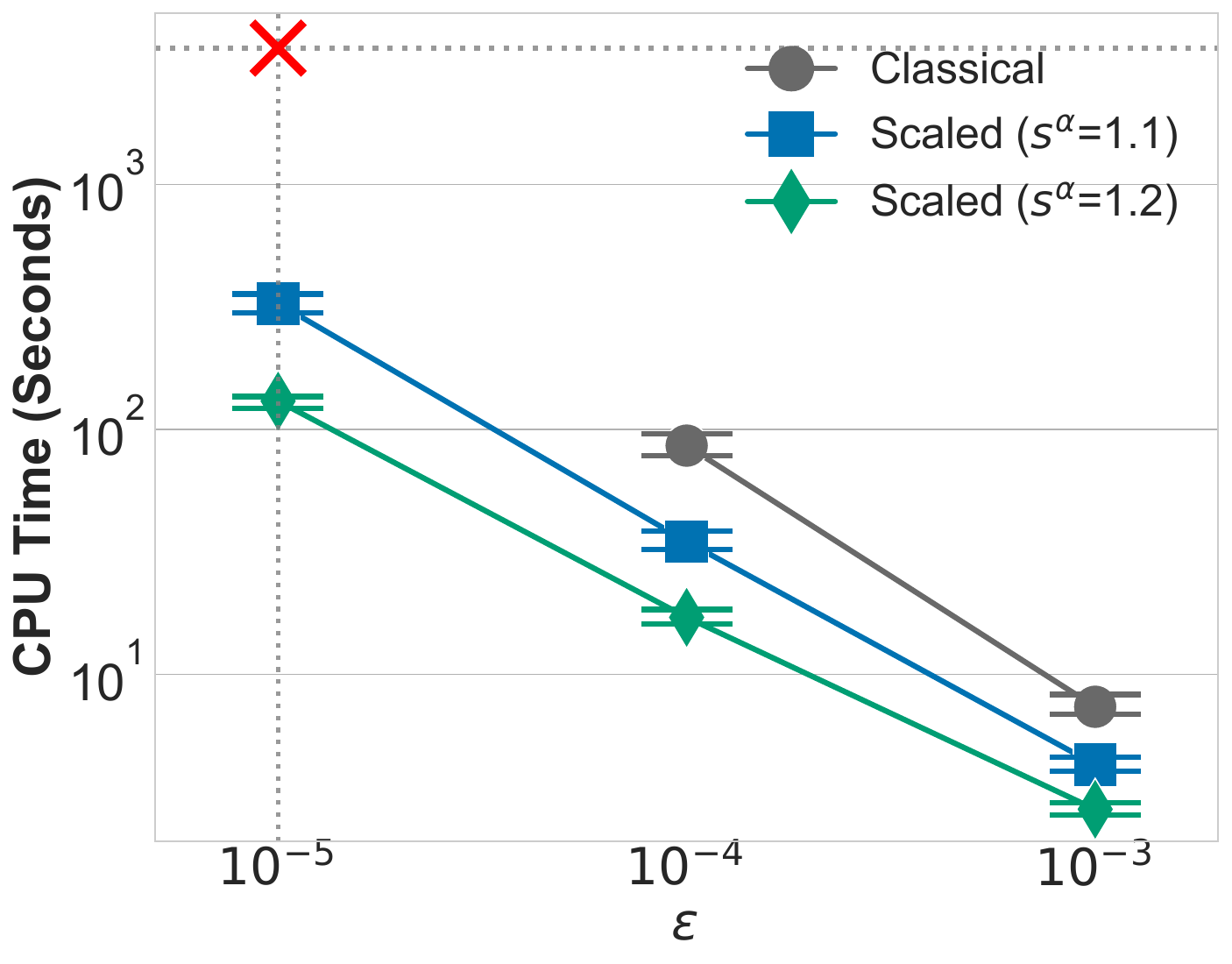}
    \end{subfigure}
    
    \caption{Comparison of computational efficiency for the norm optimization problem: (left) Required sample size; (right) CPU time. At $\varepsilon=10^{-5}$, the classical \ac{sa} failed to find a feasible solution for any of the 100 test instances.}
    \label{fig:norm_efficiency}
\end{figure}
Figure~\ref{fig:norm_efficiency} illustrates the computational performance of
our decision-scaled approach. 
Consistent with prior experiments, the scaled methods achieve substantial reductions in both sample size requirements and CPU time, with these advantages becoming more pronounced at more stringent risk levels.
Notably, because the norm optimization formulation features a joint chance constraint, each generated scenario introduces $m$ distinct constraints.
This results in a total of $m\times N$ constraints in the optimization model, severely compounding the memory burden.
Consequently, at the extreme rare-event tolerance of $\varepsilon=10^{-5}$, the classical \ac{sa} failed to find a feasible solution for any of the 100 test instances.
In contrast, the decision-scaled \ac{sa} successfully solved all instances across every risk level.
This demonstrates that our method, beyond computational efficiency, provides essential robustness for solving problems in extreme rare-event settings where the classical approach becomes intractable.

\begin{figure}[htbp]
    \centering
    \begin{subfigure}[t]{0.495\textwidth}
        \centering
        \includegraphics[width=\textwidth]{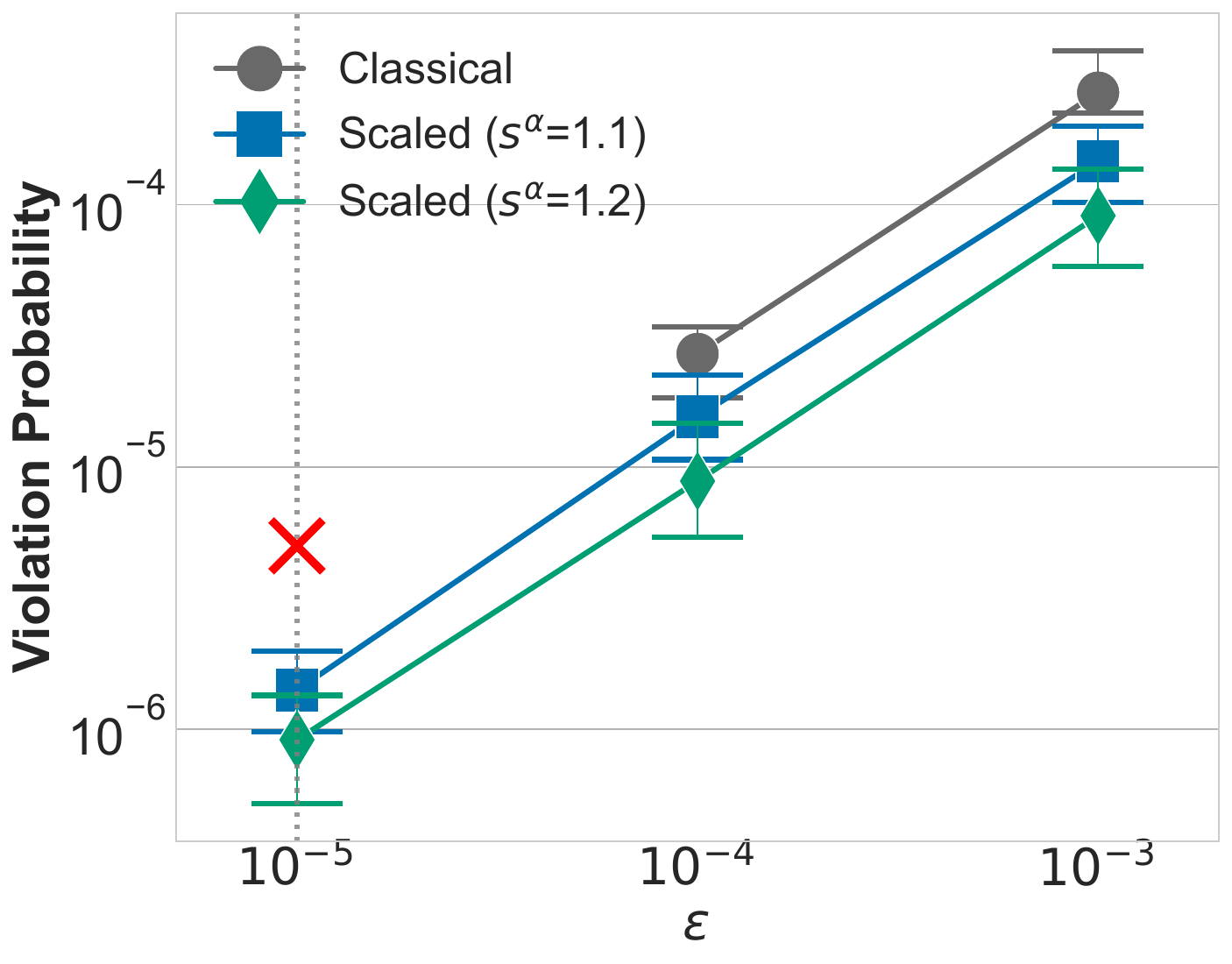}
    \end{subfigure}
    \hfill
    \begin{subfigure}[t]{0.495\textwidth}
        \centering
        \includegraphics[width=\textwidth]{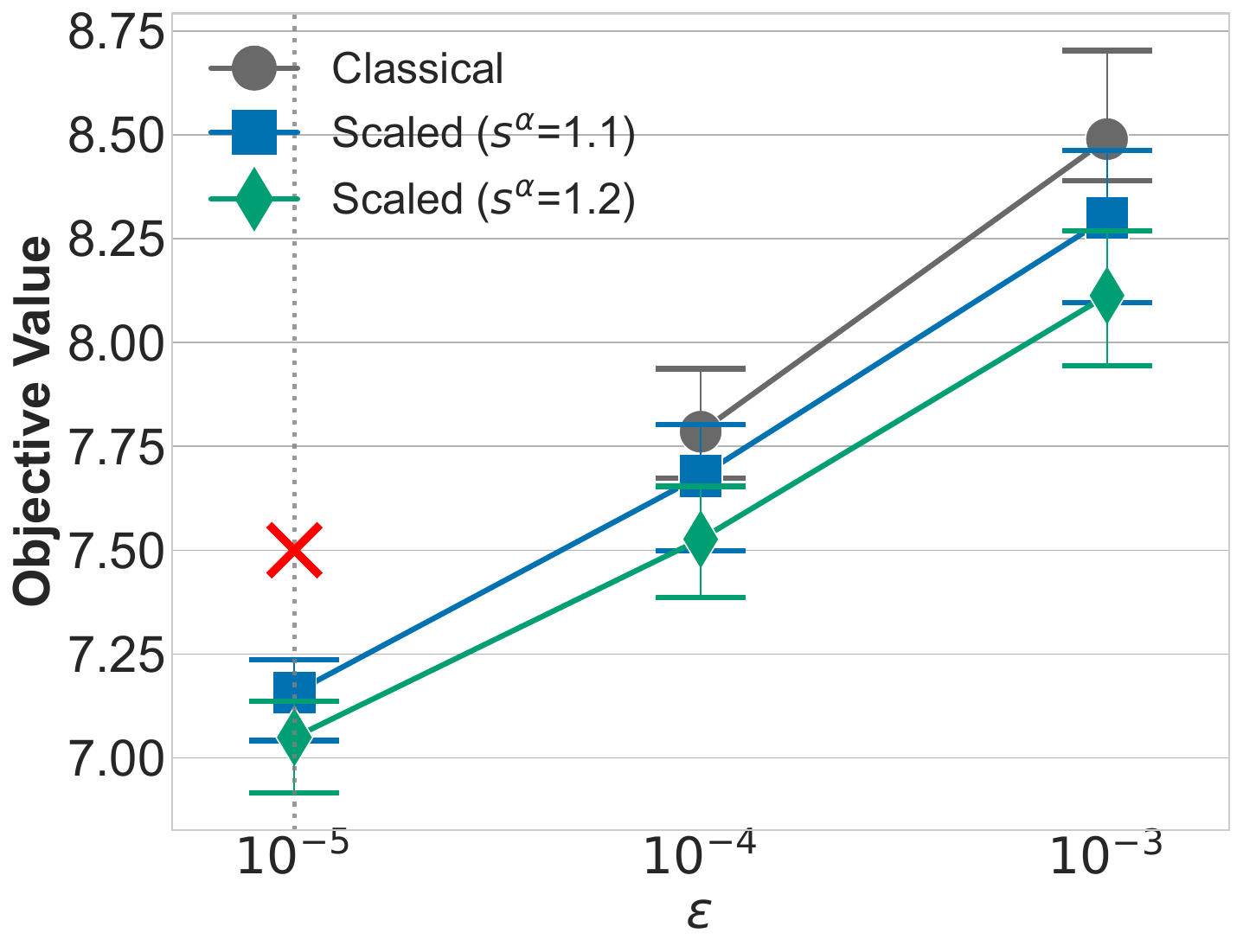}
    \end{subfigure}
    
    \caption{Comparison of solution quality for the norm optimization problem: (left) Objective value; (right) Violation probability. At $\varepsilon=10^{-5}$, the classical \ac{sa} failed to find a feasible solution for any of the 100 test instances.}
    \label{fig:norm_quality}
\end{figure}
Figure~\ref{fig:norm_quality} examines the solution quality trade-offs.
As with the previous
benchmarks, all methods produce solutions with violation probabilities
well below the target $\varepsilon$ (left panel), and the scaled methods yield
slightly lower objective values (right panel).
However, at $\varepsilon=10^{-5}$, this comparison becomes moot since only the scaled methods produce any solutions at all.

\section{Proofs}\label{Sec:Proofs}

\subsection{Supporting Lemmas for the Proof of Theorem~\ref{thm:LT_feasible}}\label{sec:supporting_lemmas}
This section develops the technical machinery needed for the proof of our main result, Theorem~\ref{thm:LT_feasible}.
The development proceeds in three stages.
First, we recall the notion of continuous convergence and establish its key consequences in Lemmas~\ref{Lem:conticonv=localuniform} and~\ref{lem: Levelset Containment}.
Second, we derive structural properties of the limit function $\lambda$ arising from the multivariate regular variation in Assumption~\ref{assum:light-tail} in Lemmas~\ref{lemma:positive_homo} and~\ref{Lemma:lambda_zero}.
Third, we establish a primitive large deviation principle in Lemma~\ref{Lem:primitive_LDP} that underpins the asymptotic analysis of chance constraints; this result forms the foundation on which the proof of Theorem~\ref{thm:LT_feasible} rests.

We begin with continuous convergence, which provides the framework for analyzing the asymptotic behavior of constraint functions. 

\begin{definition}
    A sequence of real-valued functions $\{f_u\}$ is said to converge continuously to a real-valued limit function $f^*$ if, for all convergent sequences $\vx_u\to\vx$,
    \begin{align}
        \lim_{u\to\infty}f_u(\vx_u)=f^*(\vx).
    \end{align}
\end{definition}
Notably, equations~\eqref{eqn:mrv} and~\eqref{eq:direct_conti_conv} are direct applications of this convergence.
Next, we establish useful properties of continuous convergence.

\begin{lemma}[\cite{rockafellar2009variational}, Theorem~7.14]\label{Lem:conticonv=localuniform}\label{Lem:Rockafella_7.14}
A sequence of real-valued functions $\{f_u\}$ converges continuously to $f^*$ if and only if $f^*$ is continuous and $\{f_u\}$ converges uniformly to $f^*$ on all compact sets.
\end{lemma}

\begin{lemma}\label{lem: Levelset Containment}
    If a sequence of functions $\{f_u\}$ converges continuously to a function $f^*$, then for any $M,\delta>0$ and $a\in\mathbb{R}$, there exists $u_0$ such that for all $u>u_0$, we have
    \begin{enumerate}[label=(\roman*), leftmargin=*, ref=\thelemma(\roman*)]
        \item\label{lem: Levelset Containment_part_i} $\mathcal{L}^{M}_{>a}(f_u) \subseteq \mathcal{L}_{\ge a}(f^*)+\mathcal{B}_\delta;$
        \item\label{lem: Levelset Containment_part_ii} $\mathcal{L}^{M}_{>a+\delta}(f^*)\subseteq \mathcal{L}^{M}_{>a}(f_u)$.
    \end{enumerate}
\end{lemma}

\begin{proof}
Let $M,\delta>0$ and $a\in\mathbb{R}$ be arbitrary.

\medskip
\noindent
\textbf{Part (i)}.
Note that $\mathcal{L}_{\ge a}(f^*)+\mathcal{B}_\delta = \{\vy : \|\vy-\vw\|\leq\delta \text{ for some } \vw \text{ with } f^*(\vw)\geq a\}$.
Suppose, for the sake of contradiction, that for every $u$ there exists some $\vz\in \mathcal{L}^{M}_{>a}(f_u)$ with 
$\vz\notin \mathcal{L}_{\ge a}(f^*)+\mathcal{B}_\delta$.
This implies that $\vz\in \mathcal{B}_M$ with $f_u(\vz)>a$, yet $f^*(\vw)<a$ for all $\vw\in \mathcal{B}_{\delta}(\vz)$. In particular, $f^*(\vz)<a$.

Now, set $\epsilon\coloneqq a-f^*(\vz)>0$ and choose $\eta \in (0, \delta)$.
Then, Lemma~\ref{Lem:Rockafella_7.14} implies that there exists $u_0$ such that, for all $u>u_0$, we have
\begin{align*}
    |f_u(\vw)-f^*(\vz)|<\frac{\epsilon}{2}, \quad\forall \vw\in \mathcal{B}_\eta(\vz).
\end{align*}
Consequently, for all $\vw\in \mathcal{B}_\eta(\vz)$ we obtain
\begin{align*}
  f_u(\vw) < f^*(\vz)+\frac{\epsilon}{2} 
= f^*(\vz)+\frac{a-f^*(\vz)}{2}
=\frac{f^*(\vz)+a}{2} < a.  
\end{align*}
Setting $\vw=\vz$ yields $f_u(\vz)<a$, contradicting our assumption that $f_u(\vz)>a$.

\medskip
\noindent
\textbf{Part (ii)}.
From Lemma~\ref{Lem:Rockafella_7.14}, there exists $u_0$ such that, for all $u>u_0$ and all $\vw\in \mathcal{B}_M$,
\begin{align*}
|f_u(\vw)-f^*(\vw)|<\delta.
\end{align*}
Therefore, for any $\vz\in \mathcal{L}^M_{>a+\delta}(f^*)$, i.e., $\vz\in \mathcal{B}_M$ and $f^*(\vz)>a+\delta$, we have
\begin{align*}
f_u(\vz) > f^*(\vz)-\delta> (a+\delta)-\delta = a, 
\end{align*}
which implies $\vz\in \mathcal{L}^M_{>a}(f_u)$.
\end{proof} 

Now, we introduce fundamental properties of $\MRV$ class.
\begin{lemma}\label{lemma:positive_homo}
    If $f\in \MRV(\mathcal{Z},h,f^*)$ with $h\in \RV(\kappa)$, then $f^*$ is positively homogeneous of degree $\kappa$.
\end{lemma}
\begin{proof}
For any $t>0$ and $\vz\in\mathcal{Z}$,
$f^*(t\vz)=\lim_{u\to\infty}{f(ut\vz)}/{h(u)}.$
Substituting $\nu = ut$ yields
\begin{align*}
f^*(t\vz) = \lim_{\nu\to\infty}\frac{f(\nu\vz)}{h({\nu}/{t})}
= t^\kappa\lim_{\nu\to\infty}\frac{f(\nu\vz)}{h(\nu)}
= t^\kappa f^*(\vz).
\end{align*}
\end{proof}

We now examine how Assumption~\ref{assum:light-tail} leverages properties of $\MRV$ classes to characterize the uncertainty distribution.

\begin{lemma}\label{Lemma:lambda_zero}
Under Assumption~\ref{assum:light-tail},
$\lambda(\vz)=0$ if and only if $\vz=\vzero$.
\end{lemma}
\begin{proof}
First, $\lambda(\vzero)=\vzero$ follows directly since
$\lambda(\vzero)=\lim_{u\to\infty}{Q(\vzero)}/{q(u)}=0$
as $Q(\vzero)$ is finite while $q(u)\to\infty$ as $u\to\infty$.

For the converse, let $\mathcal{K}\coloneqq\supportset\cap\{\vtheta\in\mathbb{R}^m:\|\vtheta\|=1\}$ denote the intersection of $\supportset$ with the unit sphere. 
For any nonzero $\vz\in\supportset$,
setting $\vtheta={\vz}/{\|\vz\|}$ gives
$\vtheta\in \mathcal{K}$ because $\supportset$ is a cone.
By Lemma~\ref{lemma:positive_homo}, we have $\lambda(\vz)=\lambda(\|\vz\|\vtheta)=\|\vz\|^\alpha\lambda(\vtheta)$. 
Since $\lambda(\vtheta)>0$ for all $\vtheta\in \mathcal{K}$ by Assumption~\ref{assum:light-tail}, we have $\lambda(\vz)>0$.
\end{proof}

The following lemma establishes that for the primitive scaling $\vz\mapsto \vz/u$, the distribution $\mathbb{P}_{\vxi}$ satisfies a large deviation principle with rate function $\lambda$ and speed $q(u)$.
This result provides the main tool for the asymptotic analysis in Section~\ref{secA1} and ultimately the proof of Theorem~\ref{thm:LT_feasible}.
\begin{lemma}\label{Lem:primitive_LDP}
    Under Assumption~\ref{assum:light-tail},
    for every closed set $\mathcal{C}\subseteq\mathbb{R}^m$,
    \begin{align}
    \limsup_{u\to\infty}\frac{\log\mathbb{P}_{\vxi}\left(\left\{\vz:{\vz}/{u}\in\mathcal{C}\right\}\right)}{q(u)}\le-\inf_{\vz\in\mathcal{C}}\lambda(\vz),
    \end{align}
    and for every open set $\mathcal{O}\subseteq\mathbb{R}^m$,
    \begin{align} \liminf_{u\to\infty}\frac{\log\mathbb{P}_{\vxi}\left(\left\{\vz:{\vz}/{u}\in\mathcal{O}\right\}\right)}{q(u)}\ge-\inf_{\vz\in\mathcal{O}}\lambda(\vz).
    \end{align}
\end{lemma}
\begin{proof}
Consider an arbitrary compact ball $\mathcal{B}_{\delta}(\vc)\subset\mathbb{R}^m$.
Under the change of variables $\vz=u\vnu$ (so that $d\vz = u^m\, d\vnu$ and $\vz/u\in\mathcal{B}_\delta(\vc)$ becomes $\vnu\in\mathcal{B}_\delta(\vc)$), and using Assumption~\ref{assum:light-tail} to write the density as $\exp(-Q(\vz))$,
\begin{align*}
\lim_{u\to\infty}\frac{\log\mathbb{P}_{\vxi}(\{\vz:\vz/u\in\mathcal{B}_{\delta}(\vc)\})}{q(u)}&=\lim_{u\to\infty}\frac{\log\int_{\mathcal{B}_{\delta}(\vc)}\exp(-Q(u\vnu))u^m d\vnu}{q(u)}\\
&=\lim_{u\to\infty}\frac{m\log u+\log\int_{\mathcal{B}_{\delta}(\vc)}\exp(-Q(u\vnu))d\vnu}{q(u)}\\
&=\lim_{u\to\infty}\frac{\log\int_{\mathcal{B}_{\delta}(\vc)}\exp(-Q(u\vnu))d\vnu}{q(u)}.
\end{align*}
The third equality holds because Assumption~\ref{assum:light-tail} implies $q\in\RV(\alpha)$ with $\alpha>0$, which in turn implies $m\log u/q(u)\to 0$ as $u\to\infty$.

Since $Q\in\MRV$,
the continuous convergence in~\eqref{eqn:mrv} implies uniform convergence of $Q(u\vnu)/q(u)\to\lambda(\vnu)$ on the compact set $\mathcal{B}_\delta(\vc)$ by Lemma~\ref{Lem:conticonv=localuniform}.
In particular, for any $\epsilon>0$ and all $u$ sufficiently large, $|Q(u\vnu)/q(u)-\lambda(\vnu)|\le\epsilon$ for all $\vnu\in\mathcal{B}_\delta(\vc)$, yielding uniform bounds on the integrand: $\exp(-q(u)(\lambda(\vnu)+\epsilon))\le\exp(-Q(u\vnu))\le\exp(-q(u)(\lambda(\vnu)-\epsilon))$.
Integrating, taking logarithms, dividing by $q(u)$, and noting that $\epsilon > 0$ is arbitrary gives:
\begin{align*}
\lim_{u\to\infty}\frac{\log\int_{\mathcal{B}_{\delta}(\vc)}\exp(-Q(u\vnu))d\vnu}{q(u)}
&=\lim_{u\to\infty}\frac{\log\int_{\mathcal{B}_{\delta}(\vc)}\exp(-q(u)\lambda(\vnu))d\vnu}{q(u)}. 
\end{align*}
Since $\lambda(\cdot)$ is continuous and the integration on the right-hand side is over a compact set, an application of Varadhan's Integral Lemma~\citep[Theorem~4.3.1]{dembo2009large} implies that it must be equal to:
\begin{align}\label{eq:varadhan}
\lim_{u\to\infty}\frac{\log\int_{\mathcal{B}_{\delta}(\vc)}\exp(-Q(u\vnu))d\vnu}{q(u)}
&=-\inf_{\vnu\in\mathcal{B}_{\delta}(\vc)}\lambda(\vnu).
\end{align}

For any compact $\mathcal{K}\subseteq\mathbb{R}^m$ and $\delta>0$, compactness ensures that $\mathcal{K}$ admits a finite cover; i.e., $\mathcal{K}\subseteq \mathcal{K}_{\delta}\coloneqq\bigcup_{i=1}^{K}\mathcal{B}_{\delta}(\vc_i)$ for some $\vc_i\in\mathcal{K}$, $i=1,\ldots,K$. Thus,
\begin{align*}
    \limsup_{u\to\infty}\frac{\log\mathbb{P}_{\vxi}(\{\vz:\vz/u\in\mathcal{K})\}}{q(u)} &\le  \max_{i=1,\ldots,K}\limsup_{u\to\infty}\frac{\log\mathbb{P}_{\vxi}(\{\vz:\vz/u \in \mathcal{B}_{\delta}(\vc_i)\})}{q(u)}\\
    &=
    -\min_{i=1,\ldots,K}\inf_{\vz\in\mathcal{B}_{\delta}(\vc_i)}\lambda(\vz)\\
    &=
    -\inf_{\vz\in\mathcal{K}_{\delta}}\lambda(\vz).
\end{align*}
The first inequality follows from the union bound $\mathbb{P}_{\vxi}(\mathcal{K})\le\sum_{i=1}^K\mathbb{P}_{\vxi}(\mathcal{B}_\delta(\vc_i))\le K\max_{i=1,\ldots,K}\mathbb{P}_{\vxi}(\mathcal{B}_\delta(\vc_i))$ and by taking logarithms and dividing by $q(u)$ which ensures $\log K/q(u)\to 0$ since $K$ is fixed.
The second equality applies the result~\eqref{eq:varadhan} to each $\mathcal{B}_\delta(\vc_i)$.
The third equality rewrites the minimum of infima over finitely many sets as a single infimum over their union $\mathcal{K}_\delta$.
Since the above bound holds for every $\delta>0$, we may take $\delta\to 0$:
\begin{align*}
\limsup_{u\to\infty}\frac{\log\mathbb{P}_{\vxi}(\{\vz:\vz/u\in\mathcal{K})\}}{q(u)} &\le
-\lim_{\delta\to 0}\inf_{\vz\in\mathcal{K}_{\delta}}\lambda(\vz)\\
&= -\inf_{\vz\in\mathcal{K}}\lambda(\vz).
\end{align*}
The final equality holds because of compactness of $\mathcal{K}$ and continuity of $\lambda$ together with $\bigcap_{\delta >0}\mathcal{K}_{\delta}=\mathcal{K}$.

Now, for any open set $\mathcal{O}$ and any $\delta>0$,
choose a near-minimizer $\vz^*\in\mathcal{O}$ satisfying
$\lambda(\vz^*)\le\inf_{\vz\in\mathcal{O}}\lambda(\vz)+\delta$.
Since $\mathcal{O}$ is open, there exists $d>0$ such that $\mathcal{B}_{d}(\vz^*)\subset\mathcal{O}$.
Therefore,
\begin{align*}
\liminf_{u\to\infty}\frac{\log\mathbb{P}_{\vxi}(\{\vz:\vz/u\in\mathcal{O}\})}{q(u)}
&\ge
\liminf_{u\to\infty}\frac{\log\mathbb{P}_{\vxi}(\{\vz:\vz/u\in\mathcal{B}_{d}(\vz^*)\})}{q(u)}\\
&=-\inf_{\vz\in\mathcal{B}_{d}(\vz^*)}\lambda(\vz)\\
&\ge -\lambda(\vz^*)\\
&\ge -\inf_{\vz\in\mathcal{O}}\lambda(\vz)-\delta.
\end{align*}
The first inequality holds by monotonicity of probability, since $\mathcal{B}_{d}(\vz^*)\subset\mathcal{O}$.
The second equality applies \eqref{eq:varadhan}.
The third inequality uses that $\vz^*\in\mathcal{B}_d(\vz^*)$, so the infimum over the ball is at most $\lambda(\vz^*)$.
The fourth inequality follows from the choice of $\vz^*$.
Since $\delta>0$ was arbitrary, taking $\delta\to 0$ yields
\begin{align*}
\liminf_{u\to\infty}\frac{\log\mathbb{P}_{\vxi}(\{\vz:\vz/u\in\mathcal{O}\})}{q(u)}
\ge
-\inf_{\vz\in\mathcal{O}}\lambda(\vz).
\end{align*}

Finally, we extend the upper bound from compact to arbitrary closed sets. By Lemma~\ref{lemma:positive_homo}, $\lambda$ is positively homogeneous (of degree $\alpha>0$), which together with its continuity implies that $\lambda$ has compact level sets, i.e., $\{\vz:\lambda(\vz)\le a\}$ is compact for all $a>0$. This provides the goodness condition required to apply~\citet[Lemma~1.2.18]{dembo2009large}, which yields the extension from compact to closed sets.
\end{proof}

\subsection{Properties of the Limiting Constraint Function}\label{sec:properties_gstar}
This section establishes structural properties of the limit function $g^*$ arising from Assumption~\ref{assum:conti-conv}.
Specifically, Lemma~\ref{Lem:conti-conv_homogeneous} shows that $g^*$ inherits a homogeneity structure from the scaling in Assumption~\ref{assum:conti-conv},  Lemma~\ref{lem:gstar_origin} characterizes the behavior of $g^*$ near the origin,
and Lemma~\ref{lem:convexity_of_g^*} establishes its inherited convexity.
These properties are used in the proof of Proposition~\ref{Lem:UniqueScalingExponents} (uniqueness of scaling exponents),
with the homogeneity and origin properties also playing a supporting role in the large deviation analysis of Section~\ref{secA1}.

\begin{lemma}\label{Lem:conti-conv_homogeneous}
    Under Assumption~\ref{assum:conti-conv}, the limit function $g^*$ is homogeneous in the sense that,
    for all $\vy\in\mathcal{X}^\infty_\gamma,\vw\in\supportset$, and $t>0$,
    \begin{align*}
        g^*(t^\gamma \vy,t\vw)=t^\rho g^*(\vy,\vw).
    \end{align*}
\end{lemma}
\begin{proof}
Let $\vy\in\mathcal{X}^\infty_\gamma,\vw\in\supportset$, and $t>0$.
By Assumption~\ref{Eq:Conti-Convergence},
\begin{align}
    g^*(t^\gamma\vy,t\vw)=\lim_{u\to\infty}\frac{g(\vx_u',\vz_u')}{u^\rho}
\end{align}
for any sequences $\{\vx_u'\}\in\mathcal{X}$ and $\{\vz_u'\}\in\supportset$ satisfying $\lim_{u\to\infty}{\vx_u'}/{u^\gamma}=t^\gamma\vy$ and $\lim_{u\to\infty}{\vz_u'}/{u}=\vy$.

We apply a change of variable $v=ut$, which implies $u=v/t$. 
As $u\to\infty$, we also have $v\to\infty$ since $t>0$.
Let us define new sequences $\{\vx_v\}=\{\vx_{v/t}'\}$ and $\{\vz_v\}=\{\vz_{v/t}'\}$.
Substituting $u=v/t$ yields
\begin{align}\label{eq:g^*homogeneity}
   g^*(t^\gamma\vy,t\vw) = \lim_{v\to\infty}\frac{g(\vx_v,\vz_v)}{(v/t)^\rho}=t^\rho\lim_{v\to\infty}\frac{g(\vx_v,\vz_v)}{v^\rho}.
\end{align}
Since these new sequences satisfy
\begin{align}
    \lim_{v\to\infty}\frac{\vx_v}{v^\gamma}=\lim_{v\to\infty}\frac{\vx_{v/t}'}{(ut)^\gamma}=\frac{1}{t^\gamma}\lim_{u\to\infty}\frac{\vx_{u}'}{u^\gamma}=\frac{1}{t^\gamma}(t^\gamma\vy)=\vy
\end{align}
and
\begin{align}
    \lim_{v\to\infty}\frac{\vz_v}{v}=\lim_{v\to\infty}\frac{\vz_{v/t}'}{ut}=\frac{1}{t}\lim_{u\to\infty}\frac{\vz_{u}'}{u}=\frac{1}{t}(t\vw)=\vw,
\end{align}
we have $\lim_{v\to\infty}{g(\vx_v,\vz_v)}/{v^\rho}=g^*(\vy,\vw)$.
Combining with~\eqref{eq:g^*homogeneity}, 
$g^*(t^\gamma\vy,t\vw)=t^\rho g^*(\vy,\vw)$ follows.
\end{proof}

\begin{lemma}\label{lem:gstar_origin}
Under Assumption~\ref{assum:conti-conv}, 
\begin{enumerate}[label=(\roman*), leftmargin=*, ref=\thelemma(\roman*)]
    \item\label{lem:gstar_origin_Part(i)} If $\rho > 0$, then $g^*(\vzero,\vzero) = 0$;
    \item\label{lem:gstar_origin_Part(ii)} If $\rho = 0$, then $g^*(\vzero,\vw) < 0$ for all $\vw \in \supportset$. 
\end{enumerate}
\end{lemma}

\begin{proof}
\noindent
\textbf{Part (i)}.
Using the homogeneity property in Lemma~\ref{Lem:conti-conv_homogeneous}, for any $t > 0$, 
\begin{align}
    g^*(\vzero, \vzero)=g^*(t^\gamma \vzero, t\vzero) = t^\rho g^*(\vzero, \vzero).\nonumber
\end{align}
Since this equality holds for all $t>0$, we have $g^*(\vzero, \vzero) = 0$.

\medskip
\noindent
\textbf{Part (ii)}.
We first show $g^*(\vzero,\vw)=g^*(\vzero,\vzero)$, for all $\vw\in\supportset$.
By Lemma~\ref{Lem:conti-conv_homogeneous}, for all $t>0$ and $\vw\in\supportset$,
\begin{align}
    g^*(\vzero,t\vw)=g^*(t^\gamma\vzero,t\vw)
    =t^0g^*(\vzero,\vw) = g^*(\vzero,\vw).\nonumber
\end{align}
Since continuity of $g^*$ implies
\begin{align}
    \lim_{t\to 0^+}g^*(\vzero,t\vw)=g^*(\vzero,\lim_{t\to 0^+}t\vw)
    =g^*(\vzero,\vzero),\nonumber
\end{align}
we have 
\begin{align}\label{eq: 0w=00}
  g^*(\vzero,\vw)=g^*(\vzero,\vzero),\quad\forall \vw\in\supportset. 
\end{align}

Now, we show $g^*(\vzero,\vzero)<0$.
Take any $\vy\in\mathcal{X}^\infty_\gamma\setminus\{\vzero\}$.
From Lemma~\ref{Lem:conti-conv_homogeneous},
\begin{align}\label{eq:homogenity_for(y,0)}
    g^*(t^\gamma\vy,\vzero)=g^*(t^\gamma\vy,t\vzero)=t^\rho g^*(\vy,\vzero)=g^*(\vy,\vzero)\quad\forall t>0.
\end{align}
If $\gamma>0$, since $g^*$ is continuous,
\begin{align}
    g^*(\vzero,\vzero)= \lim_{t\to 0^+}g^*(t^\gamma\vy,\vzero).\nonumber
\end{align}
Similarly, if $\gamma<0$,
\begin{align}
    g^*(\vzero,\vzero)= \lim_{t\to \infty}g^*(t^\gamma\vy,\vzero).\nonumber
\end{align}
In both cases, we have $g^*(\vzero,\vzero)=g^*(\vy,\vzero)$ following from~\eqref{eq:homogenity_for(y,0)}.
Given that $g^*(\vy,\vzero)<0$ from~\ref{eq:away_from_0},
we obtain $g^*(\vzero,\vzero)<0$.

Combining with \eqref{eq: 0w=00}, we conclude $g^*(\vzero,\vw)<0$, for all $\vw\in\supportset$.
\end{proof}

\begin{lemma}\label{lem:convexity_of_g^*}
Under Assumption~\ref{assum:conti-conv}, for any fixed $\vw\in\supportset$, the limit function $g^*(\cdot,\vw)$ is convex on $\mathcal{X}^\infty_\gamma$.
\end{lemma}

\begin{proof}
Let $\vw\in\supportset$ be fixed.
Choose any $\vy_1,\vy_2\in\mathcal{X}^\infty_\gamma$ and any scalar $t\in [0,1]$.
We aim to show that 
\begin{align*}
    g^*(t\vy_1+(1-t)\vy_2,\vw)\le tg^*(\vy_1,\vw)+ (1-t)g^*(\vy_2,\vw).
\end{align*}

By the definition of $\mathcal{X}^\infty_\gamma$, there exist sequences $\{\vx_{1,u}\}\subset\mathcal{X}$ and $\{\vx_{2,u}\}\subset\mathcal{X}$ such that $\lim_{u\to\infty}\vx_{1,u}/u^\gamma=\vy_1$ and $\lim_{u\to\infty}\vx_{2,u}/u^\gamma=\vy_2$.
Because the feasible set $\mathcal{X}$ is convex, the convex combination $\vx_u\coloneqq t\vx_{1,u}+(1-t)\vx_{2,u}$ satisfies $\vx_u\in\mathcal{X}$ for all $u$.
Hence, $\lim_{u \to \infty} {\vx_u}/{u^\gamma} = t \vy_1 + (1-t)\vy_2$.

Now, let $\{\vz_u\} \subset \supportset$ be any sequence satisfying $\lim_{u \to \infty} {\vz_u}/{u} = \vw$. 
Such a sequence always exists, since $\supportset$ is a closed cone (Remark~\ref{Remark:ClosedCone}).
Since the constraint function $g(\cdot, \vz)$ is convex for any fixed $\vz$, for each $u$, we obtain
\begin{align*}
    g(\vx_u, \vz_u) = g(t \vx_{1,u} + (1-t)\vx_{2,u}, \vz_u) \le t g(\vx_{1,u}, \vz_u) + (1-t) g(\vx_{2,u}, \vz_u).
\end{align*}
Dividing both sides of the inequality by $u^\rho$ yields
\begin{align*}
   \frac{g(\vx_u, \vz_u)}{u^\rho} \le t \frac{g(\vx_{1,u}, \vz_u)}{u^\rho} + (1-t) \frac{g(\vx_{2,u}, \vz_u)}{u^\rho}. 
\end{align*}
Taking the limit as $u \to \infty$ on both sides, and applying the continuous convergence property from Assumption~\ref{Eq:Conti-Convergence}, we obtain
\begin{align*}
    \lim_{u \to \infty} \frac{g(\vx_u, \vz_u)}{u^\rho} &\le t \lim_{u \to \infty} \frac{g(\vx_{1,u}, \vz_u)}{u^\rho} + (1-t) \lim_{u \to \infty} \frac{g(\vx_{2,u}, \vz_u)}{u^\rho},
\end{align*}
which simplifies to
\[
    g^*(t \vy_1 + (1-t)\vy_2, \vw) \le t g^*(\vy_1, \vw) + (1-t) g^*(\vy_2, \vw).\eqno
\]
\end{proof}

\subsection{Proof of Proposition~\ref{Lem:UniqueScalingExponents}}\label{sec:proof_of_uniqueness} 
\begin{proof}
Suppose there exist two pairs $(\gamma_1,\rho_1,g_1^*)$ and $(\gamma_2,\rho_2,g^*_2)$ satisfying Assumption~\ref{assum:conti-conv}.
We prove uniqueness by showing that the following cases lead to contradictions:
\begin{enumerate}[label=(\roman*),leftmargin=*]
    \item $\gamma_1=\gamma_2$ and $\rho_1\neq\rho_2$;
    \item $\gamma_1\neq\gamma_2$ and $\rho_1\neq\rho_2$;
    \item $\gamma_1\neq\gamma_2$ and $\rho_1=\rho_2$.
\end{enumerate}
\medskip 
\noindent
\textbf{Case (i)}. 
Without loss of generality, assume $\rho_1>\rho_2$ and $\gamma=\gamma_1=\gamma_2$.
Fix $\vy\in\mathcal{X}^\infty_{\gamma}\setminus\{\vzero\}$ and $\vw=\vzero$. 
By the definition of $\mathcal{X}^\infty_{\gamma}$, there exists a sequence $\{\vx_u\}\subset\mathcal{X}$ such that $\lim_{u\to\infty}\vx_u/u^\gamma=\vy$.
From Assumption~\ref{Eq:Conti-Convergence},
we have
\begin{align*}
g_1^*(\vy,\vzero)=\lim_{u\to\infty}\frac{g(\vx_u,\vzero)}{u^{\rho_1}}
&=\lim_{u\to\infty}\frac{g(\vx_u,\vzero)}{u^{\rho_2}}\cdot\frac{u^{\rho_2}}{u^{\rho_1}}
    \\&=g_2^*(\vy,\vzero)\cdot\lim_{u\to\infty}\frac{u^{\rho_2}}{u^{\rho_1}}
    \\&=0,
\end{align*}
contradicting~\ref{eq:away_from_0}.

\medskip 
\noindent
\textbf{Case (ii)}. Without loss of generality, assume $\rho_1>\rho_2$. We examine four subcases:
\begin{enumerate}[label=(\alph*),leftmargin=*]
    \item $\gamma_1<\gamma_2$;
    \item $0<\gamma_2<\gamma_1$;
    \item $\gamma_2<\gamma_1<0$;
    \item $\gamma_2<0<\gamma_1$.
\end{enumerate}
\noindent
\textbf{Case (iia)}. 
Fix $\vy\in\mathcal{X}^\infty_{\gamma_1}\setminus\{\vzero\}$.
Then, there exists a sequence $\{\vx_u\}\subset\mathcal{X}$ such that $\lim_{u\to\infty}\vx_u/u^{\gamma_1}=\vy$.
From Assumption~\ref{Eq:Conti-Convergence}, we have
\begin{align}\label{eqn:unique_case_ii_a}
    g_1^*(\vy,\vzero)
    =\lim_{u\to\infty}\frac{g(\vx_u,\vzero)}{u^{\rho_1}}
    &=\lim_{u\to\infty}\frac{g(\vx_u,\vzero)}{u^{\rho_2}}\cdot\frac{u^{\rho_2}}{u^{\rho_1}}.
\end{align}
Since $\gamma_1<\gamma_2$,
$\{\vx_u\}$ satisfies
\begin{align*}
    \lim_{u\to\infty}\frac{\vx_u}{u^{\gamma_2}}=\lim_{u\to\infty}\frac{\vx_u}{u^{\gamma_1}}\cdot\frac{u^{\gamma_1}}{u^{\gamma_2}}=\vzero.    
\end{align*}
Note that $\vzero\in\mathcal{X}^\infty_{\gamma_2}$ as it is a cone.
Therefore, we obtain $\lim_{u\to\infty}{g(\vx_u,\vzero)}/{u^{\rho_2}}=g^*_2(\vzero,\vzero)$,
which is a finite value.
Since $\rho_1>\rho_2$,
combining with~\eqref{eqn:unique_case_ii_a} yields $g_1^*(\vy,\vzero)=0,$
contradicting Assumption~\ref{eq:away_from_0}.

\medskip
\noindent
\textbf{Case (iib)}. 
From Assumption~\ref{eq:nonempty_set_g^*}, there exist
$\vy\in\mathcal{X}^\infty_{\gamma_1}\setminus\{\vzero\}$ and $\vw\in\supportset$
such that $g^*_1(\vy,\vw)>0$. Fix such $\vy$ and $\vw$.
Then, there exist sequences $\{\vx_u\}\subset\mathcal{X}$ and $\{\vz_u\}\subset\supportset$ such that
\begin{align}
    \lim_{u\to\infty}\frac{\vx_u}{u^{\gamma_1}}=\vy\quad\text{and}\quad
    \lim_{u\to\infty}\frac{\vz_u}{u}=\vw.
\end{align}
Let $v=u^{\gamma_1/\gamma_2}$. 
Since $\gamma_1/\gamma_2>1$, $v\to\infty$ as $u\to\infty$.
By Assumption~\ref{Eq:Conti-Convergence},
\begin{equation}\label{eq:unique_ii_b}
\begin{aligned}
    g_1^*(\vy,\vw)=\lim_{u\to\infty}\frac{g(\vx_u,\vz_u)}{u^{\rho_1}}
    &=
    \lim_{u\to\infty}\frac{g(\vx_u,\vz_u)}{u^{\left(\frac{\gamma_1}{\gamma_2}\right)\rho_2}}\cdot
    \frac{u^{\left(\frac{\gamma_1}{\gamma_2}\right)\rho_2}}{u^{\rho_1}}\\
    &=\lim_{v\to\infty}\frac{g(\vx_{v^{\gamma_2/\gamma_1}},\vz_{v^{\gamma_2/\gamma_1}})}{v^{\rho_2}}\cdot
    \frac{v^{\rho_2}}{v^{\left(\frac{\gamma_2}{\gamma_1}\right)\rho_1}}.
\end{aligned}    
\end{equation}
Here, $\vx_{v^{\gamma_2/\gamma_1}}$ and $\vz_{v^{\gamma_2/\gamma_1}}$ denote the re-parameterization of the sequence $\{\vx_u\}$ and $\{\vz_u\}$ using $u=v^{\gamma_2/\gamma_1}$.
Given that
\begin{align}    \lim_{v\to\infty}\frac{\vx_{v^{\gamma_2/\gamma_1}}}{v^{\gamma_2}}=\lim_{u\to\infty}\frac{\vx_u}{u^{\gamma_1}}=\vy
\end{align}
and
\begin{align}    \lim_{v\to\infty}\frac{\vz_{v^{\gamma_2/\gamma_1}}}{v}=\lim_{u\to\infty}\frac{\vz_u}{u^{\gamma_1/\gamma_2}}=\lim_{u\to\infty}\frac{\vz_u}{u}\cdot\frac{u}{u^{\gamma_1/\gamma_2}}=\vzero,
\end{align}
we obtain
\begin{align}    \lim_{v\to\infty}\frac{g(\vx_{v^{\gamma_2/\gamma_1}},\vz_{v^{\gamma_2/\gamma_1}})}{v^{\rho_2}}=g_2^*(\vy,\vzero).
\end{align}
Therefore,
equality~\eqref{eq:unique_ii_b} reduces to
\begin{align}
    g^*_1(\vy,\vw)=\begin{cases}
        0,&\quad\text{if }\rho_2<\left(\frac{\gamma_2}{\gamma_1}\right)\rho_1,\\
        g_2^*(\vy,\vzero),&\quad\text{if }\rho_2=\left(\frac{\gamma_2}{\gamma_1}\right)\rho_1,\\
        -\infty,&\quad\text{if }\rho_2>\left(\frac{\gamma_2}{\gamma_1}\right)\rho_1.
    \end{cases}
\end{align}
Furthermore, $g^*_2(\vy,\vzero)<0$ follows from Assumption~\ref{eq:away_from_0}. As a result, in each case, $g_1^*(\vy,\vw)\leq 0$ for all $\vw$, contradicting Assumption~\ref{eq:nonempty_set_g^*}.

\medskip
\noindent
\textbf{Case (iic)}. Let $v=u^{\gamma_2/\gamma_1}$. By the same argument of Case (iib), this case also leads to a contradiction.

\medskip
\noindent
\textbf{Case (iid)}.
Choose any $\vy\in\mathcal{X}^\infty_{\gamma_2}$ and $\vw\in\supportset$.
By the definition~\ref{def:asymptoset} and Remark~\ref{Remark:ClosedCone}, there exist sequences $\{\vx_u\}$ and $\{\vz_u\}$ such that
$\lim_{u\to\infty}\vx_u/u^{\gamma_2}=\vy$ and
$\lim_{u\to\infty}\vz_u/u=\vw$.
Since $\gamma_2<0<\gamma_1$, it follows $\lim_{u\to\infty}\vx_u/u^{\gamma_1}=\vzero$.
Therefore, given $\rho_2<\rho_1$, Assumption~\ref{Eq:Conti-Convergence} yields
\begin{align}
    g_1^*(\vzero,\vw)=\lim_{u\to\infty}\frac{g(\vx_u,\vz_u)}{u^{\rho_1}}
    =\lim_{u\to\infty}\frac{g(\vx_u,\vz_u)}{u^{\rho_2}}\cdot\frac{u^{\rho_2}}{u^{\rho_1}}
    =0.
\end{align}
Because $\vw\in\supportset$ was chosen arbitrarily, this establishes 
\begin{align}\label{eq:proof_unique_case_iid}
  g_1^*(\vzero,\vw)=0,\quad\forall\vw\in\supportset.
\end{align}

Next, we demonstrate this leads to a contradiction.
Choose any $\vy_1\in\mathcal{X}^\infty_{\gamma_1}$.
By Assumption~\ref{eq:nonempty_set_g^*}, there exists a $\vw_1\in\supportset$ such that $g_1^*(\vy_1,\vw_1)>0$.
By Lemma~\ref{lem:convexity_of_g^*}, $g_1^*(\cdot, \vw)$ is convex.
We apply the definition of convexity.
For any $t\in(0,1)$, we have $t^{\gamma_1} \in (0, 1)$ since $\gamma_1 > 0$. 
This gives
\begin{align}
    g_1^*(t^{\gamma_1}\vy_1,t\vw_1)&\le t^{\gamma_1}g_1^*(\vy_1,t\vw_1)+ (1-t^{\gamma_1})g_1^*(\vzero,t\vw_1)\\
    &=t^{\gamma_1}g_1^*(\vy_1,t\vw_1),
\end{align}
where the last equality follows from~\eqref{eq:proof_unique_case_iid}.
Furthermore, since Lemma~\ref{Lem:conti-conv_homogeneous} states
$t^{\rho_1}g_1^*(\vy_1,\vw_1)=g_1^*(t^{\gamma_1}\vy_1,t\vw_1)$, we obtain
\begin{align}
    t^{\rho_1}g_1^*(\vy_1,\vw_1)\le t^{\gamma_1}g_1^*(\vy_1,t\vw_1),\quad\forall t\in(0,1).
\end{align}
Given that $g_1^*$ is continuous, taking the limit as $t \to 0^+$ yields
\begin{align}
    \lim_{t\to 0^+}g_1^*(\vy_1,t\vw_1)=g_1^*(\vy_1,\vzero)<0,
\end{align}
where the strict negativity follows from Assumption~\ref{eq:away_from_0}.
This implies there exists $t_0\in(0,1)$ such that $g_1^*(\vy_1,t\vw_1)<0$ for all $t\in(0,t_0)$.
Consequently, it follows that
\begin{align}
    t^{\rho_1}g_1^*(\vy_1,\vw_1)\le t^{\gamma_1}g_1^*(\vy_1,t\vw_1)<0,\quad\forall t\in(0,t_0).
\end{align}
However, this contradicts our initial choice of $\vy_1$ and $\vw_1$, which guaranteed $g_1^*(\vy_1,\vw_1)>0$ and, by extension, $t^{\rho_1}g_1^*(\vy_1,\vw_1)>0$ for all $t>0$.

\medskip 
\noindent
\textbf{Case (iii)}. Without loss of generality, assume $\gamma_1<\gamma_2$ and $\rho=\rho_1=\rho_2$. 
We divide the proof into two subcases:
\begin{enumerate}[label=(\alph*),leftmargin=*]
    \item $\rho>0$;
    \item $\rho=0$.
\end{enumerate}
\noindent
\textbf{Case (iiia)}.
Fix $\vy\in\mathcal{X}^\infty_{\gamma_1}$ and $\vw=\vzero$. 
By Assumption~\ref{Eq:Conti-Convergence}, there exist $\{\vx_u\}\in\mathcal{X}$ and $\{\vz_u\}\in\supportset$ such that
\begin{align}
    \lim_{u\to\infty}\frac{\vx_u}{u^{\gamma_1}}=\vy\quad\text{and}\quad \lim_{u\to\infty}\frac{\vz_u}{u}=\vw=\vzero
\end{align}
and
\begin{align} \lim_{u\to\infty}\frac{g(\vx_u,\vz_u)}{u^{\rho}}=g_1^*(\vy,\vzero).
\end{align}

Since $\gamma_1<\gamma_2$, we have
\begin{align}
    \lim_{u\to\infty}\frac{\vx_u}{u^{\gamma_2}}=\lim_{u\to\infty}\frac{\vx_u}{u^{\gamma_1}}\cdot\frac
    {u^{\gamma_1}}{u^{\gamma_2}}=\vzero,
\end{align}
thus, by Assumption~\ref{Eq:Conti-Convergence},
\begin{align}
\lim_{u\to\infty}\frac{g(\vx_u,\vz_u)}{u^\rho}=g_2^*(\vzero,\vzero),
\end{align}
which implies $g_1^*(\vy,\vzero)=g_2^*(\vzero,\vzero)$.
However, this leads to a contradiction since $g_1^*(\vy,\vzero)<0$ from Assumption~\ref{eq:away_from_0} and $g_2^*(\vzero,\vzero)=0$ from Lemma~\ref{lem:gstar_origin_Part(i)}.

\medskip
\noindent
\textbf{Case (iiib)}.
Assumption~\ref{eq:nonempty_set_g^*} ensures that there exist $\vy_0\in\mathcal{X}^\infty_{\gamma_1}\setminus\{\vzero\}$ and $\vw_0\in\supportset$ such that $g_1^*(\vy_0,\vw_0)>0$. 
With the exactly same argument as Case (iiia), we derive $g_1^*(\vy_0,\vw_0)=g_2^*(\vzero,\vw_0)$.
However, Lemma~\ref{lem:gstar_origin_Part(ii)} states that $g_2^*(\vzero,\vw_0)<0$, contradicting our selection that $g_1^*(\vy_0,\vw_0)>0$.
\end{proof}

\subsection{Large Deviation Principle for Chance Constraints}\label{secA1}
This section establishes an \ac{ldp} that characterizes the asymptotic behavior of chance constraints.
Building on the primitive \ac{ldp} of Lemma~\ref{Lem:primitive_LDP} and the properties of $g^*$ from Section~\ref{sec:properties_gstar}, we define a rate function $I(\vy)$ that captures the exponential decay rate of the violation probability.
The main result of this section is Proposition~\ref{Prop:LDP}, which provides matching upper and lower bounds on this decay.
This serves as theoretical foundations for proving Theorem~\ref{thm:LT_feasible}.

\begin{definition}
    For every $\vy\in\mathcal{X}^\infty_\gamma$, the rate function is defined as
    \begin{align}
    I(\vy)\coloneqq\inf_{\vw\in\supportset}\{\lambda(\vw):g^*(\vy,\vw)\geq 0\}.
    \end{align}
\end{definition}
For a given $\{\vx_u\}\subset\mathcal{X}$, rate function quantifies the asymptotic decay rate of the violation probability 
$\mathbb{P}_{\vxi}(\{\vz:g(\vx_u,\vz)>0\})$ as $u\to\infty$.
We begin by establishing key properties of $g^*$ and $I(\vy)$.

\begin{lemma}\label{lem:uniform_separation}
Let $\mathcal{K}\subset\mathcal{X}^\infty_\gamma\setminus\{\vzero\}$ be any non-empty compact set.
Then, there exists a constant $\delta>0$ such that, for  all $\vy\in \mathcal{K}$ and for all $\vw\in\supportset$ with $\|\vw\|<\delta$, $g^*(\vy,\vw)<0$.
\end{lemma}
\begin{proof}
    We prove by contradiction.
    Suppose, for any $\delta>0$, there exist some $\vy\in \mathcal{K}$ and $\vw\in\supportset$ with $\|\vw\|<\delta$ such that $g^*(\vy,\vw)\geq 0$.
    
    Let us construct a sequence by choosing $\delta_k={1}/{k}$, $k\in\mathbb{N}$. 
    For each $k$, there exists a pair $(\vy_k,\vw_k)$ such that $\vy_k\in \mathcal{K}$, $\|\vw_k\|<{1}/{k}$ and $g^*(\vy_k,\vw_k)\ge 0$. 
    Since $\mathcal{K}$ is compact, there exists a subsequence of $\{\vy_k\}$ that converges to a limit within $\mathcal{K}$. 
    For notational simplicity, we denote this convergent subsequence by $\{\vy_k\}$ as well, and its limit by $\vy_0\in \mathcal{K}$. 
    For $\{\vw_k\}$, the condition $\|\vw_k\|<{1}/{k}$ implies that $\vw_k\to\vzero$ as $k\to\infty$.

    Since $g^*$ is continuous from Assumption~\ref{Eq:Conti-Convergence}, we have
    \begin{align*}
        \lim_{k\to\infty}g^*(\vy_k,\vw_k)=g^*\left(\lim_{k\to\infty}\vy_k, \lim_{k\to\infty}\vw_k \right)=g^*(\vy_0,\vzero).
    \end{align*}
    Since $g^*(\vy_k,\vw_k)\geq0$ for all $k$, 
    $g^*(\vy_0,\vzero)\geq0.$
    However, this is a contradiction of~\ref{eq:away_from_0}.
\end{proof}

\begin{lemma}\label{Lem:UniformPostiiveRateFunction}
Let $\mathcal{K}\subset\mathcal{X}^\infty_\gamma\setminus\{\vzero\}$ be any non-empty compact set.
Then, there exists constant $c>0$ and $M>0$ such that $c\le I(\vy)\le M$, for all $\vy\in \mathcal{K}$.
\end{lemma}
\begin{proof}
    We begin by establishing the lower bound.
    From Lemma~\ref{lem:uniform_separation}, there exists a $\delta>0$ such that $\vw\in\supportset$ with $\|\vw\|<\delta$ implies $g^*(\vy,\vw)<0$, for all $\vy\in \mathcal{K}$. 
    In other words, for all $\vy\in\mathcal{K}$,
    \begin{align*}
        \{\vw\in\supportset:g^*(\vy,\vw)\geq 0\}\subseteq\{\vw\in\supportset:\|\vw\|\geq\delta\}.
    \end{align*}
    Hence, $I(\vy)\geq\inf_{\vw\in\supportset}\{\lambda(\vw):\|\vw\|\ge\delta\}.$

    Set $c\coloneqq\inf_{\vw\in\supportset}\{\lambda(\vw):\|\vw\|\ge\delta\}$.
    Since $\lambda$ is positively homogeneous of degree $\alpha>0$ from Lemma~\ref{lemma:positive_homo} and $\|\vw\|^\alpha\ge\delta^\alpha$ for any $\vw$ with $\|\vw\|\geq\delta$,
    \begin{align*}
    c=\inf_{\vw\in\supportset,\|\vw\|\ge\delta}\lambda\left(\|\vw\|\cdot \frac{\vw}{\|\vw\|}\right)
    &=\inf_{\vw\in\supportset,\|\vw\|\ge\delta}\|\vw\|^\alpha\cdot\lambda\left(\frac{\vw}{\|\vw\|}\right)\\
    &\geq\delta^\alpha\inf_{\vw\in\supportset\cap\{\vw':\|\vw'\|=1\}}\lambda(\vw).
    \end{align*}
    By Assumption~\ref{assum:light-tail}, $\lambda(\vw)>0$ for all $\vw$ on the unit sphere within $\supportset$. Since $\lambda$ is continuous and the set $\supportset\cap\{\vw':\|\vw'\|=1\}$ is compact, the minimum of $\lambda$ on this set is attained and is strictly positive. As $\delta>0$ and $\alpha>0$, we conclude that $c>0$.

    Next, we derive the upper bound. 
    For any $\vy$, Assumption~\ref{eq:nonempty_set_g^*} ensures the existence of a point $\vw_{\vy} \in \supportset$ such that $g^*(\vy, \vw_{\vy}) > 0$. 
    Furthermore, since $g^*$ is continuous by Assumption~\ref{Eq:Conti-Convergence},
    there exists an open neighborhood $\mathcal{N}_{\vy}$ of $\vy$ such that
    \begin{align}
        g^*(\vy',\vw_{\vy})>0,\quad\forall \vy'\in\mathcal{N}_{\vy}.
    \end{align}
    Since $\mathcal{K}$ is compact, there exists a finite set $\{\vy_1,\ldots,\vy_m\}$ such that $\mathcal{K}\subseteq\cup_{i=1}^m \mathcal{N}_{\vy_i}$.
    Let us define the finite constant $M\coloneqq\max_{i=1,\ldots,m}\lambda(\vw_{\vy_i})$.
    
    Choose any $\vy\in\mathcal{K}$.
    There must exist an index $j\in\{1,\ldots,m\}$ such that $\vy\in\mathcal{N}_{\vy_j}$.
    Thus, it follows that
    \begin{align}
        I(\vy)=\inf_{\vw\in\supportset}\{\lambda(\vw)\colon g^*(\vy,\vw)\ge 0\}\le
        \lambda(\vw_{\vy_j})
        \le M.
    \end{align}   
\end{proof}

\begin{proposition}\label{Prop:LDP}
    Let $\{\vy_u\}$ be any sequence such that $u^\gamma\vy_u \in \mathcal{X}$ for all $u$, and $\lim_{u\to\infty}\vy_u = \vy_0 \in \mathcal{X}^\infty_\gamma$.
    Then, the following limits hold:
    \begin{align}\label{eq:Uniform_LDP_LB}
    \liminf_{u\to\infty}\frac{\log{\mathbb{P}_{\vxi}\left(\left\{\vz:g(u^\gamma\vy_u,\vz)>0\right\}\right)}}{-q(u)}\ge
    I(\vy_0),
    \end{align}
     and
    \begin{align}\label{eq:Uniform_LDP_UB}
    \limsup_{u\to\infty}\frac{\log{\mathbb{P}_{\vxi}\left(\left\{\vz:g(u^\gamma\vy_u,\vz)>0\right\}\right)}}{-q(u)}\leq
    I(\vy_0).
    \end{align}
\end{proposition}

\begin{proof}
From Lemma~\ref{Lem:primitive_LDP},
for every closed set $\mathcal{C}\subseteq\mathbb{R}^m$,
    \begin{align}\label{eq:primitiveLDP_closed}
    \limsup_{u\to\infty}\frac{\log\mathbb{P}_{\vxi}\left(\left\{\vz:{\vz}/{u}\in\mathcal{C}\right\}\right)}{q(u)}\le-\inf_{\vw\in\mathcal{C}}\lambda(\vw),
    \end{align}
    and for every open set $\mathcal{O}\subseteq\mathbb{R}^m$,
    \begin{align}\label{eq:primitiveLDP_open} \liminf_{u\to\infty}\frac{\log\mathbb{P}_{\vxi}\left(\left\{\vz:{\vz}/{u}\in\mathcal{O}\right\}\right)}{q(u)}\ge-\inf_{\vw\in\mathcal{O}}\lambda(\vw).
    \end{align}

We define
\begin{align}
g_{u}(\vw) \coloneqq \frac{g(u^\gamma \vy_u, u\vw)}{u^\rho}\quad\text{and}\quad 
g_{0}^*(\vw) \coloneqq g^*(\vy_0,\vw).
\end{align}
Let $\{\vw_u\}\subset\supportset$ be an arbitrary convergent sequence.
Since $\supportset$ is closed, there exists some $\vw_0\in\supportset$ such that $\lim_{u\to\infty}\vw_u=\vw_0$.
Setting $\vx_u\coloneqq u^\gamma\vy_u$ and $\vz_u=u\vw_u$ guarantee
\begin{align}
    \lim_{u\to\infty}\frac{\vx_u}{u^\gamma}=\lim_{u\to\infty}\vy_u=\vy_0
    \quad
    \text{and}
    \quad
    \lim_{u\to\infty}\frac{\vz_u}{u}=\lim_{u\to\infty}\vw_u=\vw_0.
\end{align}
Thus, by Assumption~\ref{Eq:Conti-Convergence}, we have
\begin{align}
    \lim_{u\to\infty}g_u(\vw_u)=\lim_{u\to\infty}\frac{g(u^\gamma\vy_u,u\vw_u)}{u^\rho}=\lim_{u\to\infty}\frac{g(\vx_u,\vz_u)}{u^\rho}
    &=g^*(\vy_0,\vw_0)\\
    &=g_0^*(\vw_0).
\end{align}
Since $\{\vw_u\}$ was chosen arbitrarily, the sequence $\{g_u\}$ converges continuously to $g^*_0$.

Now, we analyze the asymptotic behavior of 
$\mathbb{P}_{\vxi}(\{\vz:g(u^\gamma \vy_u,\vz)>0\})$.
\begin{align}
\mathbb{P}_{\vxi}(\{\vz:g(u^\gamma \vy_u,\vz)>0\}) 
&= \mathbb{P}_{\vxi}\left(\left\{\vz:\frac{g\left(u^\gamma \vy_u, u\,\frac{\vz}{u}\right)}{u^\rho}>0\right\}\right) \\
&= \mathbb{P}_{\vxi}\left(\{\vz:g_{u}\left({\vz}/{u}\right)>0\}\right) \\
&= \mathbb{P}_{\vxi}(\{\vz:{\vz}/{u}\in \mathcal{L}_{>0}(g_u)\}).
\end{align}
\medskip
\noindent
\textbf{Lower bound}. For any $M>0$, we decompose the event by splitting $\mathbb{R}^m$ into the ball $\mathcal{B}_M$ and its complement. Using $\mathcal{A}\subseteq (\mathcal{A}\cap \mathcal{B}_M)\cup \overline{\mathcal{B}_M^c}$ for any set $\mathcal{A}$, the union bound gives
\begin{equation}
\begin{aligned}
    \mathbb{P}_{\vxi}\left(\left\{\vz:{\vz}/{u}\in \mathcal{L}_{>0}(g_u) \right\}\right)
    &\leq 
    \mathbb{P}_{\vxi}\left(\left\{\vz:{\vz}/{u}\in \mathcal{L}^M_{>0}(g_u)\right\}\right)\\
    &\quad +
    \mathbb{P}_{\vxi}\left(\left\{\vz:{\vz}/{u}\in \overline{B_M^c}\right\}\right).
\end{aligned}    
\end{equation}
By Lemma~\ref{lem: Levelset Containment_part_i} (applied with $a=0$ and the continuous convergence of $\{g_u\}$ to $g_0^*$), for any $\delta>0$ and all $u$ sufficiently large, $\mathcal{L}^M_{>0}(g_u)\subseteq \mathcal{L}_{\ge 0}(g_0^*)+\mathcal{B}_\delta$. Substituting this containment into the first term above and using $\log(A+B)\le\max\{\log 2A,\log 2B\}$ for $A,B>0$, we obtain
\begin{align*}
&\log\left(\mathbb{P}_{\vxi}\left(\left\{\vz:{\vz}/{u}\in \mathcal{L}^M_{>0}(g_u)\right\}\right)+\mathbb{P}_{\vxi}\left(\left\{\vz:{\vz}/{u}\in \overline{\mathcal{B}_M^c}\right\}\right)\right)
\leq\\ &\max\left\{
\log2\mathbb{P}_{\vxi}\left(\left\{\vz:{\vz}/{u}\in \mathcal{L}_{\ge 0}\left(g_{0}^*\right)+\mathcal{B}_{\delta}\right\}\right),
\log2\mathbb{P}_{\vxi}\left(\left\{\vz:{\vz}/{u}\in \overline{\mathcal{B}_M^c}\right\}\right)
\right\}.
\end{align*}
Note that $\mathcal{L}_{\ge0}\left(g_0^*\right)+\mathcal{B}_\delta$ is closed, being the Minkowski sum of the closed set $\mathcal{L}_{\ge 0}(g_0^*)$ (a superlevel set of the continuous function $g_0^*$) and the compact set $\mathcal{B}_\delta$.
Similarly, $\overline{\mathcal{B}_M^c}$ is closed.
Thus, applying the primitive \ac{ldp} upper bound~\eqref{eq:primitiveLDP_closed} to each term, dividing by $q(u)$, and noting $\log 2/q(u)\to 0$, gives
\begin{align*}
&\limsup_{u\to\infty}\frac{\log{\mathbb{P}_{\vxi}\left(\left\{\vz:{\vz}/{u}\in \mathcal{L}_{>0}(g_u)\right\}\right)}}{q(u)} \\&\leq
\max\left\{
\limsup_{u\to\infty}\frac{\log{\mathbb{P}_{\vxi}\left(\left\{\vz:{\vz}/{u}\in \mathcal{L}_{\ge0}(g_0^*)+\mathcal{B}_\delta\right\}\right)}}{q(u)},\right.\\
&\quad\left.\limsup_{u\to\infty}\frac{\log{\mathbb{P}_{\vxi}\left(\left\{\vz:{\vz}/{u}\in \overline{\mathcal{B}_M^c}\right\}\right)}}{q(u)}\right\}\\
&\leq -\min\left\{
\inf_{\vw\in \mathcal{L}_{\ge0}(g_0^*)+\mathcal{B}_\delta}\lambda(\vw),
\inf_{\|\vw\|\geq M}\lambda(\vw)
\right\}.
\end{align*}
Note that $\lambda(\vw)\to\infty$ as $\|\vw\|\to\infty$ 
since $\lambda$ is positively homogeneous of degree $\alpha>0$ (Lemma~\ref{lemma:positive_homo}) and strictly positive on the unit sphere (Assumption~\ref{assum:light-tail}).
Since $\lambda$ is coercive, $\inf_{\|\vw\|\ge M}\lambda(\vw)\to\infty$ as $M\to\infty$, so the second term in the minimum becomes negligible.
For the first term, taking $\delta\to 0$ yields $\inf_{\vw\in \mathcal{L}_{\ge0}(g_0^*)+\mathcal{B}_\delta}\lambda(\vw)\to\inf_{\vw\in \mathcal{L}_{\ge0}(g_0^*)}\lambda(\vw)$ by continuity of $\lambda$.
Therefore, taking $M\to\infty$ and $\delta\to 0$, we obtain
\begin{equation}
\begin{aligned}
\limsup_{u\to\infty}\frac{\log{\mathbb{P}_{\vxi}\left(\left\{\vz:{\vz}/{u}\in \mathcal{L}_{>0}(g_u)\right\}\right)}}{q(u)}&\leq
-\inf_{\vw\in \mathcal{L}_{\ge0}(g_0^*)}\lambda(\vw)\\
&= -I(\vy_0).
\end{aligned}    
\end{equation}
Multiplying both sides by $-1$ and applying the identity $-\limsup_{u\to\infty} a_u = \liminf_{u\to\infty} (-a_u)$, we conclude that
\begin{align}
\liminf_{u\to\infty}\frac{\log{\mathbb{P}_{\vxi}\left(\left\{\vz:g(u^\gamma\vy_u,\vz)>0\right\}\right)}}{-q(u)} \geq I(\vy_0).
\end{align}

\medskip
\noindent
\textbf{Upper bound}. 
For any $\delta>0$ and $M>0$, the strict superlevel set $\mathcal{L}_{>\delta}(g_0^*)=\{\vw:g_0^*(\vw)>\delta\}$ is open since it is the preimage of the open set $(\delta,\infty)$ under the continuous function $g_0^*$.
Its intersection with the open ball $(\mathcal{B}_M)^\circ$ is therefore also open.
Applying the primitive \ac{ldp} lower bound~\eqref{eq:primitiveLDP_open} yields
\begin{align*}   \liminf_{u\to\infty}\frac{\log\mathbb{P}_{\vxi}\left(\left\{\vz:{\vz}/{u}\in \mathcal{L}_{>\delta}(g_0^*)\cap(\mathcal{B}_M)^\circ\right\}\right)}{q(u)}\ge -\inf_{\vw\in \mathcal{L}_{>\delta}(g_0^*)\cap(\mathcal{B}_M)^\circ}\lambda(\vw).
\end{align*}
We now establish the chain of inclusions $\mathcal{L}_{>\delta}(g_0^*)\cap(\mathcal{B}_M)^\circ\subseteq\mathcal{L}^M_{>\delta}(g_0^*)\subseteq\mathcal{L}^M_{>0}(g_u)\subseteq\mathcal{L}_{>0}(g_u)$:
the first inclusion holds because $(\mathcal{B}_M)^\circ\subseteq\mathcal{B}_M$;
the second follows from Lemma~\ref{lem: Levelset Containment_part_ii} (applied with $a=0$, so that the $a+\delta$ superlevel set of $g_0^*$ is contained in the $a$ superlevel set of $g_u$ for large $u$);
the third is immediate since the restricted superlevel set is contained in the unrestricted one.
By monotonicity of probability, we therefore have
\begin{align*}
\liminf_{u\to\infty}\frac{\log\mathbb{P}_{\vxi}\left(\left\{\vz:{\vz}/{u}\in \mathcal{L}_{>0}(g_u)\right\}\right)}{q(u)}\ge -\inf_{\vw\in \mathcal{L}_{>\delta}(g_0^*)\cap(\mathcal{B}_M)^\circ}\lambda(\vw).
\end{align*}
As $M\to\infty$, the constraint $\vw\in(\mathcal{B}_M)^\circ$ becomes vacuous, and as $\delta\to 0$, the strict superlevel set $\mathcal{L}_{>\delta}(g_0^*)$ exhausts $\mathcal{L}_{>0}(g_0^*)$. Therefore,
\begin{equation}
\begin{aligned}
\liminf_{u\to\infty}\frac{\log\mathbb{P}_{\vxi}\left(\left\{\vz:{\vz}/{u}\in \mathcal{L}_{>0}(g_u)\right\}\right)}{q(u)}
&\ge -\inf_{\vw\in \mathcal{L}_{>0}(g_0^*)}\lambda(\vw)\\
&=-\inf_{\vw\in \mathcal{L}_{\ge0}(g_0^*)}\lambda(\vw)\\
&= -I(\vy_0).
\end{aligned}    
\end{equation}
The second equality holds because the infimum of the continuous function $\lambda$ over $\mathcal{L}_{>0}(g_0^*)$ equals that over its closure $\mathcal{L}_{\ge0}(g_0^*)$; indeed, for any $\vw\in\mathcal{L}_{\ge0}(g_0^*)$, there exist $\vw_k\in\mathcal{L}_{>0}(g_0^*)$ with $\vw_k\to\vw$, and continuity of $\lambda$ gives $\lambda(\vw_k)\to\lambda(\vw)$.
The last equality follows from the definition of $I(\vy_0)$.
Multiplying both sides by $-1$ and applying the identity $-\liminf_{u\to\infty} a_u = \limsup_{u\to\infty} (-a_u)$ yields
\begin{align}
\limsup_{u\to\infty}\frac{\log\mathbb{P}_{\vxi}\left(\left\{\vz:g(u^\gamma\vy_u,\vz)>0\right\}\right)}{-q(u)} \leq I(\vy_0).
\end{align}
\end{proof}

\subsection{Proof of Theorem~\ref{thm:LT_feasible}}\label{proof_of_main_result} 

\begin{proof}
Let $u_{\varepsilon}\coloneqq q^{\leftarrow}(\log{1}/{\varepsilon})$, where $q^{\leftarrow}(t) \coloneqq \inf\{u : q(u) \ge t\}$ is the generalized inverse of $q$.
Since $q\in\RV(\alpha)$ with $\alpha>0$, \citet[Theorem 1.5.12]{bingham1989regular} guarantees that $q(u_{\varepsilon}) \sim \log(1/\varepsilon)$ as $\varepsilon \to 0$, where $f(\varepsilon) \sim g(\varepsilon)$ denotes $\lim_{\varepsilon \to 0} f(\varepsilon)/g(\varepsilon) = 1$.
We define
\begin{align}
\vy_\varepsilon\coloneqq \frac{\vx_\varepsilon}{ (u_{\varepsilon})^\gamma}.
\end{align}
The rescaled sequence $\{\vy_\varepsilon\}$ captures the growth of $\vx_\varepsilon$ relative to the natural scale $u_\varepsilon^\gamma$, enabling us to apply the large deviation principles developed in Section~\ref{secA1}.

We divide the proof into two exhaustive cases:
\begin{enumerate}[label=(\roman*),leftmargin=*]
    \item The sequence $\{\vy_\varepsilon\}$ is bounded;
    \item The sequence $\{\vy_\varepsilon\}$ is unbounded.
\end{enumerate}
\medskip 
\noindent
\textbf{Case (i)}.
Boundedness implies that the sequence $\{\vy_{\varepsilon}\}$ is contained in a compact set. 
We separate this case into two subcases:
\begin{enumerate}[label=(\alph*),leftmargin=*]
    \item The sequence is eventually bounded away from the origin \\i.e., $\liminf_{\varepsilon\to 0}\|\vy_\varepsilon\|>0$;
    \item The sequence admits a subsequence converging to the origin \\i.e., $\liminf_{\varepsilon\to 0}\|\vy_\varepsilon\|=0$.
\end{enumerate}

\medskip 
\noindent
\textbf{Case (ia)}.
Since $\{\vy_\varepsilon\}$ is eventually bounded away from the origin, any limit point $\vy_0$ satisfies $\vy_0\neq\vzero$.
By the definition of a limit point, there exists a sequence $\{\varepsilon_k\} \subset (0,1)$ converging to $0$ as $k\to\infty$ such that the corresponding subsequence $\{\vy_{\varepsilon_k}\}$ converges to $\vy_0$.
Moreover, it follows that $\vy_0\in\mathcal{X}^\infty_\gamma$ by Definition~\ref{def:asymptoset}. 
Indeed, the sequence $\vx_{\varepsilon_k}\in\mathcal{X}$ paired with $u_{\varepsilon_k}\to\infty$ satisfies $\vx_{\varepsilon_k}/u_{\varepsilon_k}^\gamma=\vy_{\varepsilon_k}\to\vy_0$, which witnesses the membership.
Hence, the set of limit points of $\{\vy_\varepsilon\}$ forms a nonempty compact subset $\mathcal{K}\subset\mathcal{X}^\infty_\gamma\setminus\{\vzero\}$.

We now analyze the asymptotic behavior of the constraint violation probability.
Define
\begin{align}
    I_{u}(\vy)\coloneqq\frac{\log\mathbb{P}_{\vxi}(\{\vz:g(u^\gamma \vy,\vz)>0\})}{-q(u)},
\end{align}
which is well-defined for any $\vy$ satisfying $u^\gamma\vy\in\mathcal{X}$.
In particular, $I_{u_\varepsilon}(\vy_\varepsilon)$ is well-defined because $u_\varepsilon^\gamma\vy_\varepsilon=\vx_\varepsilon\in\mathcal{X}$.
With this notation, we can write
\begin{align}
    \frac{\log V(\vx_\varepsilon)}{\log\varepsilon}
    =\frac{\log\mathbb{P}_{\vxi}(\{\vz:g(u_\varepsilon^\gamma \vy_\varepsilon,\vz)>0\})}{-\log(1/\varepsilon)}
    =\frac{-q(u_\varepsilon)}{-\log(1/\varepsilon)}\cdot I_{u_\varepsilon}(\vy_\varepsilon).
\end{align}
Since $q(u_\varepsilon) \sim \log(1/\varepsilon)$ by the definition of $u_\varepsilon$,
taking the limit inferior as $\varepsilon \to 0$ on both sides simplifies the asymptotic relationship to
\begin{align}\label{eq:log_violation_rewrite}
    \liminf_{\varepsilon \to 0} \frac{\log V(\vx_\varepsilon)}{\log\varepsilon} = \liminf_{\varepsilon \to 0} I_{u_\varepsilon}(\vy_\varepsilon).
\end{align}
For any convergent subsequence $\{\vy_{\varepsilon_k}\}\to \vy_0$, by Proposition~\ref{Prop:LDP}, it follows that
\begin{align}\label{eq:case_ia_LDP_bounds}
    \liminf_{k\to\infty}I_{u_{\varepsilon_k}}(\vy_{\varepsilon_k})\ge I(\vy_0).
\end{align}
Since this holds for every convergent subsequence of the bounded sequence $\{\vy_\varepsilon\}$, we obtain
\begin{align}
\liminf_{\varepsilon\to 0}I_{u_\varepsilon}(\vy_\varepsilon)
\ge \inf_{\vy\in\mathcal{K}}I(\vy).
\end{align}
Together with~\eqref{eq:log_violation_rewrite}, we obtain
\begin{align}\label{eq:log_violation_prob}
\liminf_{\varepsilon\to 0}\frac{\log V(\vx_\varepsilon)}{\log\varepsilon}
\ge \inf_{\vy\in\mathcal{K}}I(\vy).
\end{align}

We show $\inf_{\vy\in\mathcal{K}}I(\vy)\ge 1$ by establishing a lower bound $I(\vy)\ge 1$ for every $\vy\in\mathcal{K}$.
Choose any convergent subsequence $\{\vy_{\varepsilon_k}\}\to\vy_0$ of $\{\vy_\varepsilon\}$.
Then, Proposition~\ref{Prop:LDP} gives
\begin{align}\label{eq:case_ia_proposition_gives}
I(\vy_0)\ge
\limsup_{k\to\infty}\frac{\log\mathbb{P}_{\vxi}\left(\left\{\vz:g\left(\left({u_{\varepsilon_k}}/{s}\right)^\gamma \vy_{\varepsilon_k},\vz\right)>0\right\}\right)}{-q(u_{\varepsilon_k}/s)}
\end{align}
Since $\vx_{\varepsilon_k}$ solves~\eqref{eq:s-sp}, Theorem~\ref{Thm:ScenarioApproach} guarantees that, with probability $1-\beta$,
\begin{align}
\mathbb{P}_{\vxi}\left(\left\{\vz:g\left(\left({u_{\varepsilon_k}}/{s}\right)^\gamma \vy_{\varepsilon_k},\vz\right)>0\right\}\right)\leq \varepsilon_k^{s^{-\alpha}}.
\end{align}
Taking logarithm and multiplying by -1 on both sides yields
\begin{align}
    -\log\mathbb{P}_{\vxi}\left(\left\{\vz:g\left(\left({u_{\varepsilon_k}}/{s}\right)^\gamma \vy_{\varepsilon_k},\vz\right)>0\right\}\right)
    \ge
    s^{-\alpha}\log(1/\varepsilon_k).
\end{align}
Combining with~\eqref{eq:case_ia_proposition_gives}, we have
\begin{align}
    I(\vy_0)\ge
    \limsup_{k\to\infty}\frac{s^{-\alpha}\log(1/\varepsilon_k)}{q(u_{\varepsilon_k}/s)}
    =\limsup_{k\to\infty}\frac{s^{-\alpha}\log(1/\varepsilon_k)}{q(u_{\varepsilon_k}/s)}\cdot \frac{q(u_{\varepsilon_k})}{q(u_{\varepsilon_k})}
    =1,
\end{align}
where the last equality follows since $q(u_{\varepsilon_k})\sim\log(1/\varepsilon_k)$ and $q(u_{\varepsilon_k}/s)\sim s^{-\alpha}q(u_{\varepsilon_k})$ due to the regular variation property $q\in\RV(\alpha)$. 
Since $\vy_0\in\mathcal{K}$ was arbitrary, it follows that
\begin{align}
    \inf_{\vy\in\mathcal{K}}I(\vy)\ge1,
\end{align}
which proves the claimed result.

\medskip 
\noindent
\textbf{Case (ib)}.
Any subsequence of $\{\vy_\varepsilon\}$ that does not converge to $\vzero$ is covered by Case~(ia). Thus, it suffices to consider the case where the sequence converges to the origin.
By passing to a subsequence if necessary, we assume that $\lim_{\varepsilon \to 0} \vy_\varepsilon = \vzero$.
When $\vy_\varepsilon\to\vzero$, the rate function $I$ is evaluated at the origin, where its value depends on whether $\rho>0$ or $\rho=0$ since these reflect qualitatively different asymptotic behavior of $g^*$ at $\vzero$. We analyze the asymptotic behavior based on the value of $\rho$:
\begin{enumerate}[label=(\arabic*),leftmargin=*]
    \item $\rho>0$;
    \item $\rho=0$.
\end{enumerate}

\medskip 
\noindent
\textbf{Case (ib1)}.
We show by contradiction that this case cannot occur.
Since $\vx_{\varepsilon}$ solves~\eqref{eq:s-sp}, Theorem~\ref{Thm:ScenarioApproach} guarantees that, with probability at least $1-\beta$, the violation probability decays at least as fast as $\varepsilon^{s^{-\alpha}}$. Noting that $\vx_\varepsilon/s^\gamma=u_\varepsilon^\gamma\vy_\varepsilon/s^\gamma$, this translates to
\begin{align}
\liminf_{\varepsilon\to 0}\frac{\log\mathbb{P}_{\vxi}\left(\left\{\vz:g\left(u_\varepsilon^\gamma\vy_\varepsilon/s^\gamma,\vz\right)>0\right\}\right)}{-\log(1/\varepsilon)}\ge {s^{-\alpha}}>0.
\end{align}
Since $q(u_\varepsilon)\sim\log(1/\varepsilon)$, it follows
\begin{align}\label{eq:Feasibility_Case_ib1}
\liminf_{\varepsilon\to 0}\frac{\log\mathbb{P}_{\vxi}\left(\left\{\vz:g\left(u_\varepsilon^\gamma\vy_\varepsilon/s^\gamma,\vz\right)>0\right\}\right)}{-q(u_\varepsilon)}\ge {s^{-\alpha}}>0.
\end{align}
This bound says the log-probability vanishes at rate at least $s^{-\alpha}$ relative to $-q(u_\varepsilon)$.

On the other hand, we derive a conflicting upper bound. Since the sequence $\{\vy_\varepsilon/s^\gamma\}$ converges to $\vzero\in\mathcal{X}^\infty_\gamma$, the upper bound~\eqref{eq:Uniform_LDP_UB} of Proposition~\ref{Prop:LDP} yields
\begin{equation}
\begin{aligned}
\limsup_{\varepsilon\to 0}\frac{\log\mathbb{P}_{\vxi}\left(\left\{\vz:g\left(u_\varepsilon^\gamma\vy_\varepsilon/s^\gamma,\vz\right)>0\right\}\right)}{-q(u_\varepsilon)}&\le I(\vzero)\\
&=\inf_{\vw:g^*(\vzero,\vw)\ge 0}\lambda(\vw).
\end{aligned}
\end{equation}
For $\rho>0$, Lemma~\ref{lem:gstar_origin_Part(i)} implies $g^*(\vzero,\vzero)=0$, so $\vzero$ belongs to the feasible set $\{\vw:g^*(\vzero,\vw)\ge 0\}$.
Furthermore, since $\lambda(\vzero)=0$ by Lemma~\ref{Lemma:lambda_zero} and $\lambda$ is nonnegative, the infimum is attained at $\vzero$, giving $I(\vzero)=0$.
This means the log-probability does not decay at all relative to $-q(u_\varepsilon)$, which contradicts~\eqref{eq:Feasibility_Case_ib1}.

\medskip 
\noindent
\textbf{Case (ib2)}.
In contrast to the previous subcase, when $\rho=0$ the constraint is ``infinitely safe'' at the origin.
By Lemma~\ref{lem:gstar_origin_Part(ii)}, we have $g^*(\vzero,\vw)<0$ for all $\vw\in\supportset$.
Consequently, the set $\{\vw:g^*(\vzero,\vw)\ge0\}$ is empty.
Adopting the convention that the infimum over an empty set is $+\infty$, we have $I(\vzero) = \infty$, meaning the violation probability decays super-exponentially at the origin.
Applying the lower bound~\eqref{eq:Uniform_LDP_LB} of Proposition~\ref{Prop:LDP} to the sequence $\{\vy_{\varepsilon}\}\to\vzero$ yields:
\begin{equation}\label{eq:Feasibility_Case_ib2}
\begin{aligned}
\liminf_{\varepsilon\to 0}\frac{\log\mathbb{P}_{\vxi}\left(\left\{\vz:g\left(u_\varepsilon^\gamma\vy_\varepsilon,\vz\right)>0\right\}\right)}{-q(u_\varepsilon)}&\ge I(\vzero)=+\infty.
\end{aligned}
\end{equation}
Given that $q(u_\varepsilon)\sim\log(1/\varepsilon)$, we obtain
\begin{equation}
\begin{aligned}
\liminf_{\varepsilon\to 0}\frac{\log V(\vx_{\varepsilon})}{\log\varepsilon}
&=
\liminf_{\varepsilon\to 0}\frac{\log\mathbb{P}_{\vxi}\left(\left\{\vz\!:\!g\left(u_\varepsilon^\gamma\vy_\varepsilon,\vz\right)\!>\!0\right\}\right)}{-\log(1/\varepsilon)}\\
&=\liminf_{\varepsilon\to 0}\frac{\log\mathbb{P}_{\vxi}\left(\left\{\vz\!:\!g\left(u_\varepsilon^\gamma\vy_\varepsilon,\vz\right)\!>\!0\right\}\right)}{-q(u_\varepsilon)}\\
&\ge+\infty.
\end{aligned}
\end{equation}
This completes the proof for the bounded case.

\medskip 
\noindent
\textbf{Case (ii)}.
Since any bounded subsequence of $\{\vy_\varepsilon\}$ is covered by Case~(i), we assume without loss of generality that $\|\vy_{\varepsilon}\|\to\infty$ as $\varepsilon\to 0$.
This means $\vx_\varepsilon$ grows faster than the characteristic scale $u_\varepsilon^\gamma$.
Define the normalized vectors $\vd_\varepsilon\coloneqq{\vy_{\varepsilon}}/{\|\vy_{\varepsilon}\|}$, which capture the asymptotic direction of $\vy_\varepsilon$ and lie on the unit sphere.

The proof considers two subcases:
\begin{enumerate}[label=(\alph*),leftmargin=*]
    \item $\gamma>0$;
    \item $\gamma<0$.
\end{enumerate}

\medskip 
\noindent
\textbf{Case (iia)}.
The key idea is to separate the magnitude $\|\vy_\varepsilon\|$ from the direction $\vd_\varepsilon$ by writing the decision vector as $\vx_\varepsilon = u_\varepsilon^\gamma \vy_\varepsilon = (u_\varepsilon \|\vy_\varepsilon\|^{1/\gamma})^\gamma \vd_\varepsilon$. This representation identifies $\tilde{u}_\varepsilon\coloneqq u_\varepsilon\|\vy_\varepsilon\|^{1/\gamma}$ as the effective scale, enabling us to apply the large deviation machinery with scale $\tilde{u}_\varepsilon$ and direction $\vd_\varepsilon$.
We analyze the asymptotic decay of the violation probability by factoring the ratio through this effective scale:
\begin{align*}
\liminf_{\varepsilon\to 0}\frac{\log V(\vx_{\varepsilon})}{\log\varepsilon}
&=
\liminf_{\varepsilon\to 0}\frac{\log\mathbb{P}_{\vxi}(\{\vz\!:\!g((\tilde{u}_\varepsilon)^\gamma \vd_\varepsilon,\vz)>0\})}{-\log(1/\varepsilon)}\\
&= \liminf_{\varepsilon\to 0}\frac{q(u_\varepsilon)}{\log(1/\varepsilon)}\cdot
\frac{q(\tilde{u}_\varepsilon)}{q(u_\varepsilon)}
\cdot\frac{\log\mathbb{P}_{\vxi}(\{\vz\!:\!g((\tilde{u}_\varepsilon)^\gamma \vd_\varepsilon,\vz)>0\})}{-q(\tilde{u}_\varepsilon)}.
\end{align*}

We evaluate the limit inferior of each factored term as $\varepsilon \to 0$.
The first term converges to 1 since $q(u_{\varepsilon}) \sim \log(1/\varepsilon)$ by the definition of $u_\varepsilon$.
For the remaining terms, observe that
$\|\vy_{\varepsilon}\|^{{1}/{\gamma}}\to\infty$ because $\|\vy_{\varepsilon}\| \to \infty$ and $\gamma>0$, so the effective scale $\tilde{u}_\varepsilon=u_\varepsilon\|\vy_\varepsilon\|^{1/\gamma}\to\infty$.
Since $q\in\RV(\alpha)$ with $\alpha>0$ and the multiplier $\|\vy_\varepsilon\|^{1/\gamma}\to\infty$, the ratio $q(\tilde{u}_\varepsilon)/q(u_\varepsilon)\to+\infty$ by the definition of regular variation.

Finally, we show the LDP ratio in the last term is strictly positive.
By the definition of limit inferior, there exists a subsequence $\{\vd_{\varepsilon_k}\}$ of $\{\vd_\varepsilon\}$ such that
\begin{align*}
\liminf_{\varepsilon\to 0}\frac{\log\mathbb{P}_{\vxi}(\{\vz\!:\!g((\tilde{u}_\varepsilon)^\gamma \vd_\varepsilon,\vz)\!>0\})}{-q(\tilde{u}_\varepsilon)}
=
\lim_{k\to \infty}\frac{\log\mathbb{P}_{\vxi}(\{\vz\!:\!g((\tilde{u}_{\varepsilon_k})^\gamma \vd_{\varepsilon_k},\vz)\!>0\})}{-q(\tilde{u}_{\varepsilon_k})}.
\end{align*}
Furthermore, since $\{\vd_{\varepsilon_k}\}$ lies on a compact unit sphere, there exists a further subsequence $\{\vd_{\varepsilon_{k_\ell}}\}$ of $\{\vd_{\varepsilon_k}\}$ converging to some limit point $\vd_0$.
This implies, by Proposition~\ref{Prop:LDP},
\begin{equation}
\begin{aligned}
&\liminf_{\varepsilon\to 0}\frac{\log\mathbb{P}_{\vxi}(\{\vz\!:\!g((\tilde{u}_\varepsilon)^\gamma \vd_\varepsilon,\vz)>0\})}{-q(\tilde{u}_\varepsilon)}\\
=&
\lim_{\ell\to \infty}\frac{\log\mathbb{P}_{\vxi}(\{\vz\!:\!g((\tilde{u}_{\varepsilon_{k_\ell}})^\gamma \vd_{\varepsilon_{k_\ell}},\vz)>0\})}{-q(\tilde{u}_{\varepsilon_{k_\ell}})}\\
\ge&
\liminf_{\ell\to \infty}\frac{\log\mathbb{P}_{\vxi}(\{\vz\!:\!g((\tilde{u}_{\varepsilon_{k_\ell}})^\gamma \vd_{\varepsilon_{k_\ell}},\vz)>0\})}{-q(\tilde{u}_{\varepsilon_{k_\ell}})}
\ge I(\vd_0).
\end{aligned}   
\end{equation}
By Lemma~\ref{Lem:UniformPostiiveRateFunction}, $I(\vd_0)>0$ holds since $\vd_0$ is a non-zero element lying on the unit sphere of $\mathcal{X}^\infty_\gamma$.

Consequently, we have a product of three terms where the first converges to $1$, the second diverges to $+\infty$, and the limit inferior of the third is strictly positive.
Given all three terms are nonnegative, we can apply the super-multiplicativity of the limit inferior to obtain
\begin{equation}
\begin{aligned}
\liminf_{\varepsilon\to 0}\frac{\log V(\vx_{\varepsilon})}{\log\varepsilon} \ge& \liminf_{\varepsilon\to 0}\frac{q(u_\varepsilon)}{\log(1/\varepsilon)} \cdot \liminf_{\varepsilon\to 0}\frac{q(\tilde{u}_\varepsilon)}{q(u_\varepsilon)} \\
&\cdot 
\liminf_{\varepsilon\to 0}\frac{\log\mathbb{P}_{\vxi}(\{\vz : g((\tilde{u}_\varepsilon)^\gamma \vd_\varepsilon, \vz) > 0\})}{-q(\tilde{u}_\varepsilon)}\\
=&+\infty,
\end{aligned}    
\end{equation}
establishing the feasibility guarantee.

\medskip 
\noindent
\textbf{Case (iib)}.
We show by contradiction that this case cannot occur.
Since $u_\varepsilon\|\vy_{\varepsilon}\|^{{1}/{\gamma}}\ge0$, by passing to a subsequence if necessary,
we may assume that $\lim_{\varepsilon\to0}u_\varepsilon\|\vy_{\varepsilon}\|^{{1}/{\gamma}}$ exists in $[0,+\infty]$.
Thus, we consider the following three exhaustive subcases:
\begin{enumerate}[label=(\arabic*),leftmargin=*]
    \item $\lim_{\varepsilon\to0}u_\varepsilon\|\vy_{\varepsilon}\|^{{1}/{\gamma}}=+\infty$;
    \item $\lim_{\varepsilon\to0}u_\varepsilon\|\vy_{\varepsilon}\|^{{1}/{\gamma}}=L\in(0,+\infty)$;
    \item $\lim_{\varepsilon\to0}u_\varepsilon\|\vy_{\varepsilon}\|^{{1}/{\gamma}}=0$.
\end{enumerate}

\medskip 
\noindent
\textbf{Case (iib1)}.
We show by contradiction that the violation probability decays too slowly to be consistent with the feasibility guarantee of the scenario approach.
Since $\vx_{\varepsilon}$ solves~\eqref{eq:s-sp}, Theorem~\ref{Thm:ScenarioApproach} yields
\begin{align}\label{Eq:case2a_greaterthan_salpha}
  \limsup_{\varepsilon\to0}\frac{\log V(\vx_{\varepsilon})}{\log\varepsilon}
\geq s^{-\alpha}>0.
\end{align}
We derive an upper bound on the left-hand side that contradicts this inequality.
As in Case~(iia), we use the effective scale $\tilde{u}_\varepsilon = u_\varepsilon\|\vy_\varepsilon\|^{1/\gamma}$ and write $\vx_\varepsilon = (\tilde{u}_\varepsilon)^\gamma \vd_\varepsilon$, so that the limit superior of the ratio $\log V(\vx_\varepsilon)/\log\varepsilon$ can be decomposed as follows:
\begin{align*}
\limsup_{\varepsilon\to 0}\frac{\log V(\vx_\varepsilon)}{\log\varepsilon}
&= \limsup_{\varepsilon\to 0}\frac{\log\mathbb{P}_{\vxi}(\{\vz\!:\!g((\tilde{u}_\varepsilon)^\gamma \vd_\varepsilon,\vz)>0\})}{-\log(1/\varepsilon)}\\
&= \limsup_{\varepsilon\to 0}  \frac{q(u_\varepsilon)}{\log(1/\varepsilon)}\cdot
\frac{q(\tilde{u}_\varepsilon)}{q(u_\varepsilon)}
\cdot \frac{\log\mathbb{P}_{\vxi}\left(\left\{\vz\!:g\left((\tilde{u}_\varepsilon)^\gamma \vd_\varepsilon,\vz\right)>0\right\}\right)}{-q(\tilde{u}_\varepsilon)}.
\end{align*}

Next, we analyze the asymptotic behavior of each factor in turn.
The first term converges to 1.
For the second term, since $\gamma<0$ and $\|\vy_\varepsilon\|\to\infty$, the multiplier $\|\vy_\varepsilon\|^{1/\gamma}\to 0$.
Since $q\in\RV(\alpha)$ with $\alpha>0$, applying Definition~\ref{def:RV} yields $q(\tilde{u}_\varepsilon)/q(u_\varepsilon)\to 0$.

Regarding the third term, we establish that its limit superior is bounded above by a finite constant.
By the definition of limit superior, there exists a subsequence $\{\vd_{\varepsilon_k}\}$ of $\{\vd_\varepsilon\}$ such that
\begin{align*}
\limsup_{\varepsilon\to 0}\frac{\log\mathbb{P}_{\vxi}(\{\vz\!:\!g((\tilde{u}_\varepsilon)^\gamma \vd_\varepsilon,\vz)\!>\!0\})}{-q(\tilde{u}_\varepsilon)}
=\!
\lim_{k\to \infty}\frac{\log\mathbb{P}_{\vxi}(\{\vz\!:\!g((\tilde{u}_{\varepsilon_k})^\gamma \vd_{\varepsilon_k},\vz)\!>\!0\})}{-q(\tilde{u}_{\varepsilon_k})}.
\end{align*}
Because the normalized vectors $\{\vd_{\varepsilon_k}\}$ reside on a compact unit sphere, we can extract a further subsequence $\{\vd_{\varepsilon_{k_\ell}}\}$ of $\{\vd_{\varepsilon_k}\}$ converging to some limit point $\vd_0$.
Applying Proposition~\ref{Prop:LDP} to this sequence yields
\begin{equation}
\begin{aligned}
&\limsup_{\varepsilon\to 0}\frac{\log\mathbb{P}_{\vxi}(\{\vz\!:\!g((\tilde{u}_\varepsilon)^\gamma \vd_\varepsilon,\vz)>0\})}{-q(\tilde{u}_\varepsilon)}\\
=&
\lim_{\ell\to \infty}\frac{\log\mathbb{P}_{\vxi}(\{\vz\!:\!g((\tilde{u}_{\varepsilon_{k_\ell}})^\gamma \vd_{\varepsilon_{k_\ell}},\vz)>0\})}{-q(\tilde{u}_{\varepsilon_{k_\ell}})}\\
\le&
\limsup_{\ell\to \infty}\frac{\log\mathbb{P}_{\vxi}(\{\vz\!:\!g((\tilde{u}_{\varepsilon_{k_\ell}})^\gamma \vd_{\varepsilon_{k_\ell}},\vz)>0\})}{-q(\tilde{u}_{\varepsilon_{k_\ell}})}
\le I(\vd_0).
\end{aligned}    
\end{equation}
Lemma~\ref{Lem:UniformPostiiveRateFunction} guarantees that $I(\vd_0)$ is bounded above by a finite constant, as $\vd_0$ is a non-zero element on the unit sphere of $\mathcal{X}^\infty_\gamma$.

We have thus decomposed the original ratio into a product of three terms:
the first component converges to $1$, the second converges to $0$, and 
the limit superior of the last is bounded above by a constant.
Because all three factors are nonnegative, we apply the sub-multiplicativity of the limit superior to obtain
\begin{equation}
\begin{aligned}
\limsup_{\varepsilon\to0}\frac{\log V(\vx_{\varepsilon})}{\log\varepsilon}\le& \limsup_{\varepsilon\to 0}\frac{q(u_\varepsilon)}{\log(1/\varepsilon)} \cdot \limsup_{\varepsilon\to 0}\frac{q(\tilde{u}_\varepsilon)}{q(u_\varepsilon)} \\
&\cdot 
\limsup_{\varepsilon\to 0}\frac{\log\mathbb{P}_{\vxi}(\{\vz : g((\tilde{u}_\varepsilon)^\gamma \vd_\varepsilon, \vz) > 0\})}{-q(\tilde{u}_\varepsilon)}\\
=&0,
\end{aligned}    
\end{equation}
which contradicts~\eqref{Eq:case2a_greaterthan_salpha}.

\medskip 
\noindent
\textbf{Case (iib2)}.
Since the sequence of normalized vectors $\{\vd_\varepsilon\}$ lies on the compact unit sphere, by passing to a further subsequence if necessary, we assume that $\lim_{\varepsilon\to 0} \vd_\varepsilon = \bar{\vd}$ for some unit vector $\bar{\vd}$.
Recalling the effective scale $\tilde{u}_\varepsilon = u_\varepsilon\|\vy_\varepsilon\|^{1/\gamma}\to L$ from the hypothesis of this subcase, we have $\vx_\varepsilon = (\tilde{u}_\varepsilon)^\gamma \vd_\varepsilon \to L^\gamma \bar{\vd}$.
Since $\mathcal{X}$ is closed, $L^\gamma\bar{\vd}\in\mathcal{X}$.
The strategy is to leverage the scenario feasibility guarantee to show that $g$ is nonpositive along a limiting direction, then extend this via convexity and the asymptotic homogeneity~\ref{Eq:Conti-Convergence} to contradict Assumption~\ref{eq:nonempty_set_g^*}.

We define the limiting violation set
\begin{align*}
\mathcal{V}\coloneqq\left\{\vz:g\left(\left({L}/{s}\right)^\gamma \bar{\vd},\vz\right)>0\right\}.
\end{align*}
Consider an arbitrary point $\vz\in\mathcal{V}$.
Since $(\tilde{u}_\varepsilon/s)^\gamma \vd_\varepsilon \to (L/s)^\gamma \bar{\vd}$ and $g(\cdot,\vz)$ is closed (i.e., lower semicontinuous) for each fixed $\vz$, we have
\begin{align*}
\liminf_{\varepsilon\to 0} g\left( \left(\frac{\tilde{u}_\varepsilon}{s}\right)^\gamma \vd_\varepsilon, \vz \right) \ge g\left( (L/s)^\gamma \bar{\vd}, \vz \right) > 0. \end{align*}
This implies that for all sufficiently small $\varepsilon$, $\vz$ belongs to the event
\begin{align*}
\mathcal{V}_\varepsilon\coloneqq\left\{\vz:g\left(\left(\frac{\tilde{u}_\varepsilon}{s}\right)^\gamma \vd_\varepsilon,\vz\right)>0\right\}.
\end{align*}
Consequently, we have the inclusion $\mathcal{V} \subseteq \liminf_{\varepsilon\to 0} \mathcal{V}_\varepsilon$.\footnote{The limit inferior of a sequence of subsets $\{\mathcal{A}_u\}$ is defined as
$\liminf_{u\to\infty}\mathcal{A}_u\coloneqq\left\{\vz\in\mathbb{R}^m: \liminf_{u\to\infty}\mathbf{1}_{\mathcal{A}_u}(\vz)=1\right\}$.
Here, for a given set $\mathcal{A}\subset\mathbb{R}^m$,  $\mathbf{1}_{\mathcal{A}}(\vz)$ is an indicator function which equals 1 if $\vz\in \mathcal{A}$, and 0 otherwise.}
Invoking Fatou's lemma, we obtain
\begin{align}\label{eq:fatou}
\mathbb{P}_{\vxi}(\mathcal{V})\leq\liminf_{\varepsilon\to 0}\mathbb{P}_{\vxi}(\mathcal{V}_\varepsilon).
\end{align}

On the other hand, observing that $(\tilde{u}_\varepsilon/s)^\gamma \vd_\varepsilon = (\tilde{u}_\varepsilon)^\gamma \vd_\varepsilon / s^\gamma = \vx_\varepsilon/s^\gamma$, we recognize $\mathcal{V}_\varepsilon = \{\vz : g(\vx_\varepsilon/s^\gamma, \vz) > 0\}$ as the violation event of the scaled solution.
Since $\vx_\varepsilon$ solves~\eqref{eq:s-sp} with $N \ge \mathsf{N}(\varepsilon^{s^{-\alpha}}, \beta)$ samples, Theorem~\ref{Thm:ScenarioApproach} ensures that $\mathbb{P}_{\vxi}(\mathcal{V}_\varepsilon)\le\varepsilon^{s^{-\alpha}}$.
Taking the limit as $\varepsilon \to 0$ in~\eqref{eq:fatou}, it follows that $\mathbb{P}_{\vxi}(\mathcal{V}) = 0$.
Since $g((L/s)^\gamma \bar{\vd}, \cdot)$ is continuous, the set $\mathcal{V}$ is open (as the preimage of $(0,\infty)$ under a continuous function).
An open subset of the support $\supportset$ that carries zero probability must be empty, so $\mathcal{V}\cap\supportset = \emptyset$, which gives
\begin{align*}
    g((L/s)^\gamma \bar{\vd}, \vz)\le0,\quad\forall\vz\in\supportset.
\end{align*}

Let $\hat{\vy}\coloneqq(L/s)^\gamma\bar{\vd}$.
We now extend the nonpositivity of $g$ from the fixed point $\hat{\vy}$ to all points $\{u^\gamma \hat{\vy}: u\ge 1\}$, which trace a path from $\hat{\vy}$ toward $\vzero$ (as $u\to\infty$, since $\gamma<0$).
Since $\gamma<0$, Assumption~\ref{assum:trivial_safety_at_0} guarantees that $g(\vzero,\vz)\le 0$ for all $\vz\in\supportset$.
For any $u \geq 1$, $u^\gamma \in (0, 1]$ since $\gamma<0$, so $u^\gamma \hat{\vy} = u^\gamma \hat{\vy} + (1-u^\gamma)\vzero$ is a convex combination of $\hat{\vy}$ and $\vzero$.
Because $g(\cdot, \vz)$ is convex,
\begin{align*}
g(u^\gamma \hat{\vy}, \vz) &\leq u^\gamma g(\hat{\vy}, \vz) + (1-u^\gamma) g(\vzero, \vz) \\
&\leq \max\{g(\hat{\vy}, \vz), g(\vzero, \vz)\} \\
&\leq 0, \quad \forall \vz \in \supportset.
\end{align*}

Finally, fix any $\vw \in \supportset$.
Setting $\vx_u \coloneqq u^\gamma \hat{\vy}$, we have $\vx_u/u^\gamma = \hat{\vy}$ and, for all $u \ge s$, $\vx_u = (uL/s)^\gamma\bar{\vd}$ lies on the segment $[\vzero, L^\gamma\bar{\vd}]\subset\mathcal{X}$ (by closedness and convexity), so $\vx_u\in\mathcal{X}$.
By Assumption~\ref{Eq:Conti-Convergence}, for any sequence $\{\vz_u\}\subset\supportset$ such that $\lim_{u\to\infty}\vz_u/u=\vw$, we obtain
\begin{align}
g^*\left(\hat{\vy},\vw\right)=\lim_{u\to\infty}\frac{g\left(u^\gamma\hat{\vy},\vz_u\right)}{u^\rho}\le 0.
\end{align}
The inequality holds because $g(u^\gamma\hat{\vy},\vz_u)\le 0$ for all sufficiently large $u$ (since $\vz_u\in\supportset$) while $u^\rho > 0$.
Since this holds for all $\vw \in \supportset$, we conclude that $\{\vw \in \supportset : g^*(\hat{\vy}, \vw) > 0\} = \emptyset$.
Moreover, $\hat{\vy}\in\mathcal{X}^\infty_\gamma\setminus\{\vzero\}$: the sequence $\vx_u = u^\gamma\hat{\vy}\in\mathcal{X}$ with $u\to\infty$ satisfies $\vx_u/u^\gamma=\hat{\vy}$, witnessing membership via Definition~\ref{def:asymptoset}, and $\hat{\vy}\neq\vzero$ since $L>0$.
This contradicts Assumption~\ref{eq:nonempty_set_g^*}.

\medskip 
\noindent
\textbf{Case (iib3)}.
Unlike Case~(iib2), where $\tilde{u}_\varepsilon\to L>0$ ensures that $\vx_\varepsilon=(\tilde{u}_\varepsilon)^\gamma\vd_\varepsilon$ converges to the finite limit $L^\gamma\bar{\vd}\in\mathcal{X}$, here $\tilde{u}_\varepsilon\to 0$ with $\gamma<0$ forces $\|\vx_\varepsilon\|=(\tilde{u}_\varepsilon)^\gamma\to\infty$, so there is no finite limiting point.
We instead exploit the divergence of $\|\vx_\varepsilon\|$ directly.

Recall from Case~(ii) that $\vd_\varepsilon=\vy_\varepsilon/\|\vy_\varepsilon\|$.
Since $\vx_\varepsilon=u_\varepsilon^\gamma\vy_\varepsilon$ with $u_\varepsilon^\gamma>0$, we have $\vx_\varepsilon/\|\vx_\varepsilon\|=\vd_\varepsilon$.
By passing to a further subsequence if necessary as in Case~(iib2), we may assume that $\vd_\varepsilon\to\bar{\vd}$ for some unit vector $\bar{\vd}$.

Fix any $t>0$.
Since $\gamma<0$ and $s\ge 1$, the scaled solution satisfies $\|\vx_\varepsilon/s^\gamma\|=\|\vx_\varepsilon\|\cdot s^{-\gamma}\to\infty$, so $\|\vx_\varepsilon/s^\gamma\|>t$ for all sufficiently small $\varepsilon>0$.
Noting that $\vx_\varepsilon/s^\gamma$ has the same direction as $\vx_\varepsilon$, namely $\vd_\varepsilon$, we write $t\vd_\varepsilon$ as a convex combination of $\vzero$ and $\vx_\varepsilon/s^\gamma$:
\begin{align*}
t\vd_\varepsilon = \left(1-\frac{t}{\|\vx_\varepsilon/s^\gamma\|}\right)\vzero + \frac{t}{\|\vx_\varepsilon/s^\gamma\|}\cdot\frac{\vx_\varepsilon}{s^\gamma}.
\end{align*}
By the convexity of $g(\cdot,\vz)$ and Assumption~\ref{assum:trivial_safety_at_0} (which gives $g(\vzero,\vz)\le 0$), we obtain
\begin{equation}
\begin{aligned}
    g(t\vd_\varepsilon, \vz)
    &\le
    \left(1 - \frac{t}{\|\vx_\varepsilon/s^\gamma\|}\right)g(\vzero, \vz) + \frac{t}{\|\vx_\varepsilon/s^\gamma\|}g\!\left(\frac{\vx_\varepsilon}{s^\gamma}, \vz\right) \\
    &\le \frac{t}{\|\vx_\varepsilon/s^\gamma\|}g\!\left(\frac{\vx_\varepsilon}{s^\gamma}, \vz\right).
\end{aligned}
\end{equation}
Since $t/\|\vx_\varepsilon/s^\gamma\|>0$, this inequality implies the set inclusion
\begin{align}\label{eq:caseiib3_set_inclusion}
    \left\{\vz:g(t\vd_\varepsilon,\vz)>0\right\}\subseteq \left\{\vz:g\!\left(\frac{\vx_\varepsilon}{s^\gamma},\vz\right)>0\right\}.
\end{align}

We now show $g(t\bar{\vd},\vz)\le 0$ for all $\vz\in\supportset$, proceeding analogously to Case~(iib2).
Consider an arbitrary $\vz\in\supportset$ with $g(t\bar{\vd},\vz)>0$.
Since $\vd_\varepsilon\to\bar{\vd}$ and $g(\cdot,\vz)$ is closed (i.e., lower semicontinuous),
\begin{align*}
\liminf_{\varepsilon\to 0} g(t\vd_\varepsilon,\vz) \ge g(t\bar{\vd},\vz)>0.
\end{align*}
Hence $g(t\vd_\varepsilon,\vz)>0$ for all sufficiently small $\varepsilon$, so it follows $\vz\in\liminf_{\varepsilon\to 0}\{\vz':g(t\vd_\varepsilon,\vz')>0\}$.
Invoking Fatou's lemma and the set inclusion~\eqref{eq:caseiib3_set_inclusion}:
\begin{equation}
\begin{aligned}\label{eq:caseiib3_inequality}
    \mathbb{P}_{\vxi}\!\left(\left\{\vz:g(t\bar{\vd},\vz)>0\right\}\right)
    &\le
    \liminf_{\varepsilon\to 0}\mathbb{P}_{\vxi}\!\left(\left\{\vz:g(t\vd_\varepsilon,\vz)>0\right\}\right)\\
    &\le \liminf_{\varepsilon\to 0}\mathbb{P}_{\vxi}\!\left(\left\{\vz:g\!\left(\frac{\vx_\varepsilon}{s^\gamma},\vz\right)>0\right\}\right)\\
    &\le \liminf_{\varepsilon\to 0}\varepsilon^{s^{-\alpha}} = 0.
\end{aligned}
\end{equation}
The first inequality combines the lower semicontinuity of $g(\cdot,\vz)$ and Fatou's lemma (exactly as in Case~(iib2)), the second uses the set inclusion~\eqref{eq:caseiib3_set_inclusion}, and the third follows from Theorem~\ref{Thm:ScenarioApproach} since $\vx_\varepsilon$ solves~\eqref{eq:s-sp} with $N\ge\mathsf{N}(\varepsilon^{s^{-\alpha}},\beta)$ samples.
Since $g(t\bar{\vd},\cdot)$ is continuous, the set $\{\vz:g(t\bar{\vd},\vz)>0\}$ is open; as before, an open subset of $\supportset$ with zero probability must be empty, so
\begin{align} g(t\bar{\vd},\vz)\le0,\quad\forall\vz\in\supportset.
\end{align}
As $t>0$ was chosen arbitrarily, this nonpositivity holds for all $t > 0$.

Finally, we derive the desired contradiction via Assumption~\ref{Eq:Conti-Convergence}, following the same structure as Case~(iib2).
Fix any $\vw\in\supportset$ and choose a sequence $\{\vz_u\}\subset\supportset$ with $\lim_{u\to\infty}\vz_u/u=\vw$.
Setting $\vx_u\coloneqq u^\gamma\bar{\vd}$, we have $\vx_u/u^\gamma=\bar{\vd}$, so the sequences $\{\vx_u\}$ and $\{\vz_u\}$ satisfy~\eqref{eq:relation_x_and_y} with limit $(\bar{\vd},\vw)$.
Moreover, $\vx_u\in\mathcal{X}$ for all $u>0$: for any $c>0$ and $\|\vx_\varepsilon\|>c$, the point $c\vd_\varepsilon=(c/\|\vx_\varepsilon\|)\vx_\varepsilon+(1-c/\|\vx_\varepsilon\|)\vzero$ is a convex combination of $\vx_\varepsilon\in\mathcal{X}$ and $\vzero\in\mathcal{X}$, so $c\vd_\varepsilon\in\mathcal{X}$; taking $\varepsilon\to 0$ and using the closedness of $\mathcal{X}$ gives $c\bar{\vd}\in\mathcal{X}$ for all $c>0$.
Assumption~\ref{Eq:Conti-Convergence} then yields
\begin{align}
  g^*(\bar{\vd},\vw) = \lim_{u\to\infty}\frac{g(u^\gamma\bar{\vd},\vz_u)}{u^\rho}\le 0,
\end{align}
where the inequality holds because $g(u^\gamma\bar{\vd},\vz_u)\le 0$ for all $u>0$.
Since this holds for all $\vw\in\supportset$, $\{\vw\in\supportset:g^*(\bar{\vd},\vw)>0\}=\emptyset$.
Moreover, $\bar{\vd}\in\mathcal{X}^\infty_\gamma\setminus\{\vzero\}$: to see this, note that the sequence $\vx_u=u^\gamma\bar{\vd}\in\mathcal{X}$ with $\lambda_u=u\to\infty$ satisfies $\vx_u/\lambda_u^\gamma=\bar{\vd}$, witnessing membership in $\mathcal{X}^\infty_\gamma$ via Definition~\ref{def:asymptoset}, and $\bar{\vd}\neq\vzero$ since it is a unit vector.
This contradicts Assumption~\ref{eq:nonempty_set_g^*}.
\end{proof}

\subsection{Proofs of Section~\ref{Sec:VerifyingConditions}}
\label{Sec:Proof_SufficientCondition}
\begin{proof}[Proof of Proposition~\ref{prop:SufficientConditionsOfContiConv}]
Let $\vy \in \mathcal{X}_{\gamma}^{\infty}$ and $\vw \in \supportset$ be arbitrary. 
Consider any pair of sequences $\{\vx_u\} \subset \mathcal{X}$ and $\{\vz_u\} \subset \supportset$ satisfying~\eqref{eq:relation_x_and_y}.
We introduce the sequences $\vy_u := \vx_u / u^\gamma$ and $\vw_u := \vz_u / u$. 
By construction, these sequences converge to $\vy$ and $\vw$, respectively, as $u \to \infty$.
Substituting $\vx_u = u^\gamma \vy_u$ and $\vz_u = u \vw_u$ into the expression for $g(x_u, z_u)$ given in~\eqref{eq:algebraic_function} for~\eqref{eq:direct_conti_conv}, we obtain
\begin{align}
    \frac{g(\vx_u,\vz_u)}{u^\rho}=\sum_{(\va,\vb)\in\mathcal{J}}C_{\va,\vb}u^{p_{\va,\vb}(\gamma)-\rho}\vy_u^{\va}\vw_u^{\vb},
\end{align}
where $p_{\va,\vb}(\gamma)$ is the exponent of each term, defined as~\eqref{eq:exponent_function}.

Algorithm~\ref{alg:find_scaling_exponents} determines $\rho$ as the maximum scaling exponent, i.e., $\rho = \max_{(\va,\vb) \in \mathcal{J}} p_{\va,\vb}(\gamma)$; see line~\ref{alg:rho_is_max} in Algorithm~\ref{alg:find_scaling_exponents}.
Consequently, for every index $(\va,\vb) \in \mathcal{J}$, the exponent of $u$ satisfies $p_{\va,\vb}(\gamma) - \rho \le 0$.
We partition the index set $\mathcal{J}$ into the active set $\mathcal{J}^* := \{(\va,\vb) \in \mathcal{J} : p_{\va,\vb}(\gamma) = \rho\}$ and the inactive set $\mathcal{J} \setminus \mathcal{J}^*$.
The limit in~\eqref{eq:direct_conti_conv} then can be analyzed as follows:
\begin{equation}
\begin{aligned}
    \lim_{u \to \infty} \frac{g(\vx_u, \vz_u)}{u^\rho} =& 
    \sum_{(\va,\vb) \in \mathcal{J}^*} C_{\va,\vb} \lim_{u \to \infty} (\vy_u^{\va} \vw_u^{\vb}) 
    \\
    &+\sum_{(\va,\vb) \in \mathcal{J} \setminus \mathcal{J}^*} C_{\va,\vb} \lim_{u \to \infty} (u^{p_{\va,\vb}(\gamma) - \rho} \vy_u^{\va} \vw_u^{\vb}).
\end{aligned}    
\end{equation}
Since power functions are continuous, and $\vy_u \to \vy, \vw_u \to \vw$, it follows that $\lim_{u \to \infty} \vy_u^{\va} \vw_u^{\vb} = \vy^{\va} \vw^{\vb}$.
For terms in $\mathcal{J} \setminus \mathcal{J}^*$, $u^{p_{\va,\vb}(\gamma)-\rho}\to 0$ as $u\to\infty$.
Therefore, we have
\begin{align}
    \lim_{u \to \infty} \frac{g(\vx_u, \vz_u)}{u^\rho} = \sum_{(\va,\vb) \in \mathcal{J}^*} C_{\va,\vb} \vy^{\va} \vw^{\vb} = g^*(\vy, \vw).
\end{align}
This confirms that the convergence in~\eqref{eq:direct_conti_conv}  holds for any sequences satisfying~\eqref{eq:relation_x_and_y}, thereby verifying condition~\ref{Eq:Conti-Convergence}.
Finally, we note that Algorithm~\ref{alg:find_scaling_exponents} explicitly verifies conditions~\labelcref{Assum:NonEmptyHorizon,eq:away_from_0,eq:nonempty_set_g^*,assum:trivial_safety_at_0} before returning the tuple $(\gamma, \rho, g^*)$.
Since~\ref{Eq:Conti-Convergence} is established by the argument above, Assumption~\ref{assum:conti-conv} is satisfied.
\end{proof}

\begin{proof}[Proof of Proposition~\ref{Lem:AlgebraicalEquivalence}]
We demonstrate that $g(\vx,\vz)$ satisfies all conditions of Assumption~\ref{assum:conti-conv} with parameters $(\gamma,\rho+\rho_h)$
and the limit function $g^*(\vy,\vw)=f^*(\vy,\vw)h^*(\vy,\vw)$ for all $\vy\in\mathcal{X}^\infty_\gamma$ and $\vw\in\supportset$. 

\medskip
\noindent
\textbf{\ref{Eq:Conti-Convergence}}.
Fix arbitrary $\vy\in\mathcal{X}^\infty_\gamma$ and $\vw\in\supportset$.
Consider any sequences $\{\vx_u\}$ and $\{\vz_u\}$ satisfying
$\lim_{u\to\infty}\vx_u/u^\gamma=\vy$ and
$\lim_{u\to\infty}\vz_u/u=\vw$, respectively.
Given that both $f$ and $h$ satisfy condition~\ref{Eq:Conti-Convergence}
with parameters $(\gamma,\rho)$ and $(\gamma,\rho_h)$, respectively,
we obtain
\begin{equation}
\begin{aligned}
\lim_{u\to\infty}\frac{g(\vx_u,\vz_u)}{u^{\rho+\rho_h}}&=\lim_{u\to\infty}
\frac{f(\vx_u,\vz_u)}{u^\rho}\cdot
\frac{h(\vx_u,\vz_u)}{u^{\rho_h}}\\
&=\lim_{u\to\infty}
\frac{f(\vx_u,\vz_u)}{u^\rho}\cdot\lim_{u\to\infty}\frac{h(\vx_u,\vz_u)}{u^{\rho_h}}\\
&=f^*(\vy,\vw)\cdot h^*(\vy,\vw)\\
&=g^*(\vy,\vw).    
\end{aligned}
\end{equation}
Moreover, continuity of $f^*$ and $h^*$ on $\mathcal{X}^\infty_\gamma\times\supportset$ implies that their product $g^*$ is continuous as well.

\medskip
\noindent
\textbf{\ref{Assum:NonEmptyHorizon}}.
As $f$ and $g$ are defined on the same domain $\mathcal{X}$ and share the same $\gamma$, these conditions are satisfied by the hypothesis that $f$ satisfies Assumption~\ref{assum:conti-conv}.

\medskip
\noindent
\textbf{\ref{eq:away_from_0}}. 
Fix any $\vy\in\mathcal{X}^\infty_\gamma\setminus\{\vzero\}$.
Since $f^*(\vy,\vzero)<0$ by condition~\ref{eq:away_from_0} and $h^*(\vy,\vw)>0$ for all $\vw\in\supportset$, we have
\begin{align}
    g^*(\vy,\vzero)=f^*(\vy,\vzero)h^*(\vy,\vzero)<0.
\end{align}

\medskip
\noindent
\textbf{\ref{eq:nonempty_set_g^*}}
Fix any $\vy\in\mathcal{X}^\infty_\gamma\setminus\{\vzero\}$.
Given that $f$ satisfies condition~\ref{eq:nonempty_set_g^*}, the set
$\{\vw\in\supportset:f^*(\vy,\vw)>0\}$ is nonempty.
Let $\vw_0$ be an element of this set.
Strict positivity of $h^*$ ensures that
\begin{align}
    g^*(\vy,\vw_0)=h^*(\vy,\vw_0)f^*(\vy,\vw_0)>0.
\end{align}
Therefore, $\{\vw\in\supportset:g^*(\vy,\vw)>0\}$ is nonempty.

\medskip
\noindent
\textbf{\ref{assum:trivial_safety_at_0}}.
Suppose $\gamma < 0$. By hypothesis, $f$ satisfies~\ref{assum:trivial_safety_at_0}, meaning $f(\vzero, \vz) \le 0$ for all $\vz \in \Xi$. 
Since $h$ is strictly positive on its domain, we have
\begin{align}
    g(\vzero, \vz) = f(\vzero, \vz)h(\vzero, \vz) \le 0, \quad \forall \vz \in \Xi.
\end{align}
\end{proof}

\begin{proof}[Proof of Proposition~\ref{lem:joint_chance_constraint}]
    We verify that $g$ satisfies~\labelcref{Eq:Conti-Convergence,Assum:NonEmptyHorizon,eq:away_from_0,eq:nonempty_set_g^*,assum:trivial_safety_at_0}
    of Assumption~\ref{assum:conti-conv}.
    
\medskip
\noindent
\textbf{\ref{Eq:Conti-Convergence}}.
Fix arbitrary $\vy \in \mathcal{X}_\gamma^\infty$ and $\vw \in \supportset$.
Let $\{\vx_u\} \subset \mathcal{X}$ and $\{\vz_u\} \subset \supportset$ be any sequences satisfying~\eqref{eq:relation_x_and_y}.
Since the limit operator commutes with the maximum over a finite set, we have
\begin{equation}
\begin{aligned}
\lim_{u\to\infty} \frac{g(\vx_u,\vz_u)}{u^\rho} &= \lim_{u\to\infty} \max_{i=1,\ldots,K}  \frac{g_i(\vx_u,\vz_u)}{u^\rho} 
= 
\max_{i=1,\ldots,K}  \lim_{u\to\infty} \frac{g_i(\vx_u,\vz_u)}{u^\rho}.  
\end{aligned}    
\end{equation}
By the hypothesis that each $g_i$ satisfies~\ref{Eq:Conti-Convergence}, the inner limits converge to $g^*_i(\vy,\vw)$.
Thus, the limit equals $g^*(\vy,\vw)$, satisfying~\ref{Eq:Conti-Convergence}.

\medskip
\noindent
\textbf{\labelcref{Assum:NonEmptyHorizon}}.
Since the domain $\mathcal{X}$ remains unchanged, this conditions are inherited directly.

\medskip
\noindent
\textbf{\ref{eq:away_from_0}}.
By hypothesis, $g_i^*(\vy, \vzero) < 0$ for all $i=1,\ldots,K$ and any $\vy \in \mathcal{X}_\gamma^\infty \setminus \{\vzero\}$.
It follows immediately that $g^*(\vy, \vzero) = \max_{i=1,\ldots,K} g_i^*(\vy, \vzero) < 0$.

\medskip
\noindent
\textbf{\ref{eq:nonempty_set_g^*}}.
Fix any $\vy \in \mathcal{X}_\gamma^\infty \setminus \{0\}$. 
For each $i$, let $\mathcal{W}_i = \{\vw \in \supportset : g_i^*(\vy, \vw) > 0\}$. 
By hypothesis, $\mathcal{W}_i \neq \emptyset$. 
Since $g^*(\vy, \vw) \ge g_i^*(\vy, \vw)$ for all $i$, we have the inclusion $\bigcup_{i=1}^K \mathcal{W}_i \subseteq \{\vw \in \supportset : g^*(\vy, \vw) > 0\}$.
This guarantees that the latter set is nonempty.

\medskip
\noindent
\textbf{\ref{assum:trivial_safety_at_0}}.
Suppose $\gamma < 0$.
By hypothesis, each $g_i$ satisfies~\ref{assum:trivial_safety_at_0}, meaning $g_i(\mathbf{0}, \mathbf{z}) \le 0$ for all $\mathbf{z} \in \supportset$. 
Consequently, $g(\vzero,\vz)=\max_{i=1,\ldots,K}g_i(\vzero,\vz)\le 0$ for all $\vz\in \supportset$.
\end{proof}

\bibliographystyle{plainnat}
\bibliography{references}

\begin{thebibliography}{64}
\providecommand{\natexlab}[1]{#1}
\providecommand{\url}[1]{\texttt{#1}}
\expandafter\ifx\csname urlstyle\endcsname\relax
  \providecommand{\doi}[1]{doi: #1}\else
  \providecommand{\doi}{doi: \begingroup \urlstyle{rm}\Url}\fi

\bibitem[Ahmed and Shapiro(2008)]{ahmed2008solving}
Shabbir Ahmed and Alexander Shapiro.
\newblock Solving chance-constrained stochastic programs via sampling and integer programming.
\newblock In \emph{State-of-the-Art Decision-Making Tools in the Information-Intensive Age}, pages 261--269. INFORMS, 2008.

\bibitem[Aoues and Chateauneuf(2010)]{aoues2010benchmark}
Younes Aoues and Alaa Chateauneuf.
\newblock Benchmark study of numerical methods for reliability-based design optimization.
\newblock \emph{Structural and Multidisciplinary Optimization}, 41\penalty0 (2):\penalty0 277--294, 2010.

\bibitem[Bajo-Buenestado(2025)]{bajo2025iberian}
Ra\'{u}l Bajo-Buenestado.
\newblock {The Iberian Peninsula Blackout -- Causes, Consequences, and Challenges Ahead}.
\newblock \url{https://doi.org/10.25613/EC9T-QJ89}, 2025.

\bibitem[Barrera et~al.(2016)Barrera, Homem-de Mello, Moreno, Pagnoncelli, and Canessa]{barrera2016chance}
Javiera Barrera, Tito Homem-de Mello, Eduardo Moreno, Bernardo~K Pagnoncelli, and Gianpiero Canessa.
\newblock Chance-constrained problems and rare events: an importance sampling approach.
\newblock \emph{Mathematical Programming}, 157:\penalty0 153--189, 2016.

\bibitem[Beraldi et~al.(2004)Beraldi, Bruni, and Conforti]{beraldi2004designing}
Patrizia Beraldi, Maria~Elena Bruni, and Domenico Conforti.
\newblock Designing robust emergency medical service via stochastic programming.
\newblock \emph{European Journal of Operational Research}, 158\penalty0 (1):\penalty0 183--193, 2004.

\bibitem[Bienstock et~al.(2014)Bienstock, Chertkov, and Harnett]{bienstock2014chance}
Daniel Bienstock, Michael Chertkov, and Sean Harnett.
\newblock Chance-constrained optimal power flow: Risk-aware network control under uncertainty.
\newblock \emph{SIAM Review}, 56\penalty0 (3):\penalty0 461--495, 2014.

\bibitem[Bingham et~al.(1989)Bingham, Goldie, and Teugels]{bingham1989regular}
Nicholas~H Bingham, Charles~M Goldie, and Jef~L Teugels.
\newblock \emph{Regular variation}, volume~27.
\newblock Cambridge University Press, 1989.

\bibitem[Blanchet et~al.(2024{\natexlab{a}})Blanchet, Jorritsma, and Zwart]{blanchet2024optimization}
Jose Blanchet, Joost Jorritsma, and Bert Zwart.
\newblock Optimization under rare events: scaling laws for linear chance-constrained programs.
\newblock \emph{arXiv preprint arXiv:2407.11825}, 2024{\natexlab{a}}.

\bibitem[Blanchet et~al.(2024{\natexlab{b}})Blanchet, Zhang, and Zwart]{blanchet2024efficient}
Jose Blanchet, Fan Zhang, and Bert Zwart.
\newblock Efficient scenario generation for heavy-tailed chance constrained optimization.
\newblock \emph{Stochastic Systems}, 14\penalty0 (1):\penalty0 22--46, 2024{\natexlab{b}}.

\bibitem[Bonami and Lejeune(2009)]{bonami2009exact}
Pierre Bonami and Miguel~A Lejeune.
\newblock An exact solution approach for portfolio optimization problems under stochastic and integer constraints.
\newblock \emph{Operations Research}, 57\penalty0 (3):\penalty0 650--670, 2009.

\bibitem[Calafiore and Campi(2005)]{calafiore2005uncertain}
Giuseppe Calafiore and Marco~C Campi.
\newblock Uncertain convex programs: randomized solutions and confidence levels.
\newblock \emph{Mathematical Programming}, 102:\penalty0 25--46, 2005.

\bibitem[Calafiore(2010)]{calafiore2010random}
Giuseppe~Carlo Calafiore.
\newblock Random convex programs.
\newblock \emph{SIAM Journal on Optimization}, 20\penalty0 (6):\penalty0 3427--3464, 2010.

\bibitem[Calafiore and Campi(2006)]{calafiore2006scenario}
Giuseppe~Carlo Calafiore and Marco~C Campi.
\newblock The scenario approach to robust control design.
\newblock \emph{IEEE Transactions on Automatic Control}, 51\penalty0 (5):\penalty0 742--753, 2006.

\bibitem[Campi and Garatti(2008)]{campi2008exact}
Marco~C Campi and Simone Garatti.
\newblock The exact feasibility of randomized solutions of uncertain convex programs.
\newblock \emph{SIAM Journal on Optimization}, 19\penalty0 (3):\penalty0 1211--1230, 2008.

\bibitem[Campi and Garatti(2011)]{campi2011sampling}
Marco~C Campi and Simone Garatti.
\newblock A sampling-and-discarding approach to chance-constrained optimization: feasibility and optimality.
\newblock \emph{Journal of Optimization Theory and Applications}, 148\penalty0 (2):\penalty0 257--280, 2011.

\bibitem[Campi et~al.(2009)Campi, Garatti, and Prandini]{campi2009scenario}
Marco~C Campi, Simone Garatti, and Maria Prandini.
\newblock The scenario approach for systems and control design.
\newblock \emph{Annual Reviews in Control}, 33\penalty0 (2):\penalty0 149--157, 2009.

\bibitem[Car{\`e} et~al.(2014)Car{\`e}, Garatti, and Campi]{care2014fast}
Algo Car{\`e}, Simone Garatti, and Marco~C Campi.
\newblock Fast—fast algorithm for the scenario technique.
\newblock \emph{Operations Research}, 62\penalty0 (3):\penalty0 662--671, 2014.

\bibitem[Charnes et~al.(1958)Charnes, Cooper, and Symonds]{charnes1958cost}
Abraham Charnes, William~W Cooper, and Gifford~H Symonds.
\newblock Cost horizons and certainty equivalents: an approach to stochastic programming of heating oil.
\newblock \emph{Management Science}, 4\penalty0 (3):\penalty0 235--263, 1958.

\bibitem[Chiang et~al.(2017)Chiang, Tan, Palomar, O'neill, and Julian]{chiang2017power}
Mung Chiang, Chee~Wei Tan, Daniel~P Palomar, Daniel O'neill, and David Julian.
\newblock Power control by geometric programming.
\newblock \emph{IEEE Transactions on Wireless Communications}, 6\penalty0 (7):\penalty0 2640--2651, 2017.

\bibitem[Choi et~al.(2024)Choi, Deo, Lagoa, and Subramanyam]{choi2024reduced}
Jaeseok Choi, Anand Deo, Constantino Lagoa, and Anirudh Subramanyam.
\newblock Reduced sample complexity in scenario-based control system design via constraint scaling.
\newblock \emph{IEEE Control Systems Letters}, 8:\penalty0 2793--2798, 2024.

\bibitem[Chun et~al.(2016)Chun, Song, and Paulino]{chun2016structural}
Junho Chun, Junho Song, and Glaucio~H Paulino.
\newblock Structural topology optimization under constraints on instantaneous failure probability.
\newblock \emph{Structural and Multidisciplinary Optimization}, 53\penalty0 (4):\penalty0 773--799, 2016.

\bibitem[Dembo and Zeitouni(2009)]{dembo2009large}
Amir Dembo and Ofer Zeitouni.
\newblock \emph{Large Deviations Techniques and Applications}, volume~38.
\newblock Springer Science \& Business Media, 2009.

\bibitem[Deo and Murthy(2023)]{deo2023achieving}
Anand Deo and Karthyek Murthy.
\newblock Achieving efficiency in black-box simulation of distribution tails with self-structuring importance samplers.
\newblock \emph{Operations Research}, 2023.

\bibitem[Deo and Murthy(2025)]{deo2025scaling}
Anand Deo and Karthyek Murthy.
\newblock The scaling behaviors in achieving high reliability via chance-constrained optimization.
\newblock \emph{arXiv preprint arXiv:2504.07728}, 2025.

\bibitem[Duckett(2005)]{duckett2005risk}
Will Duckett.
\newblock Risk analysis and the acceptable probability of failure.
\newblock \emph{Structural Engineer}, 83\penalty0 (15), 2005.

\bibitem[Einmahl et~al.(2021)Einmahl, Yang, and Zhou]{einmahl2021testing}
John~HJ Einmahl, Fan Yang, and Chen Zhou.
\newblock Testing the multivariate regular variation model.
\newblock \emph{Journal of Business \& Economic Statistics}, 39\penalty0 (4):\penalty0 907--919, 2021.

\bibitem[Einmahl et~al.(2025)Einmahl, Krajina, and Cai]{einmahl2025empirical}
John~HJ Einmahl, Andrea Krajina, and Juan~Juan Cai.
\newblock Empirical likelihood based testing for multivariate regular variation.
\newblock \emph{The Annals of Statistics}, 53\penalty0 (1):\penalty0 352--373, 2025.

\bibitem[El{\c{c}}i and Noyan(2018)]{elcci2018chance}
{\"O}zg{\"u}n El{\c{c}}i and Nilay Noyan.
\newblock A chance-constrained two-stage stochastic programming model for humanitarian relief network design.
\newblock \emph{Transportation Research Part B: Methodological}, 108:\penalty0 55--83, 2018.

\bibitem[Fontem(2023)]{fontem2023robust}
Belleh Fontem.
\newblock Robust chance-constrained geometric programming with application to demand risk mitigation.
\newblock \emph{Journal of Optimization Theory and Applications}, 197\penalty0 (2):\penalty0 765--797, 2023.

\bibitem[Geletu et~al.(2017)Geletu, Hoffmann, Kloppel, and Li]{geletu2017inner}
Abebe Geletu, Armin Hoffmann, Michael Kloppel, and Pu~Li.
\newblock An inner-outer approximation approach to chance constrained optimization.
\newblock \emph{SIAM Journal on Optimization}, 27\penalty0 (3):\penalty0 1834--1857, 2017.

\bibitem[Hong et~al.(2011)Hong, Yang, and Zhang]{hong2011sequential}
L~Jeff Hong, Yi~Yang, and Liwei Zhang.
\newblock Sequential convex approximations to joint chance constrained programs: A monte carlo approach.
\newblock \emph{Operations Research}, 59\penalty0 (3):\penalty0 617--630, 2011.

\bibitem[Ivanov and Dolgui(2020)]{ivanov2020viability}
Dmitry Ivanov and Alexandre Dolgui.
\newblock Viability of intertwined supply networks: extending the supply chain resilience angles towards survivability. a position paper motivated by covid-19 outbreak.
\newblock \emph{International Journal of Production Research}, 58\penalty0 (10):\penalty0 2904--2915, 2020.

\bibitem[Jasour et~al.(2015)Jasour, Aybat, and Lagoa]{jasour2015semidefinite}
Ashkan~M Jasour, Necdet~S Aybat, and Constantino~M Lagoa.
\newblock Semidefinite programming for chance constrained optimization over semialgebraic sets.
\newblock \emph{SIAM Journal on Optimization}, 25\penalty0 (3):\penalty0 1411--1440, 2015.

\bibitem[Lagoa et~al.(2005)Lagoa, Li, and Sznaier]{lagoa2005probabilistically}
Constantino~M Lagoa, Xiang Li, and Mario Sznaier.
\newblock Probabilistically constrained linear programs and risk-adjusted controller design.
\newblock \emph{SIAM Journal on Optimization}, 15\penalty0 (3):\penalty0 938--951, 2005.

\bibitem[Lasserre and Weisser(2021)]{lasserre2021distributionally}
Jean~B Lasserre and Tillmann Weisser.
\newblock Distributionally robust polynomial chance-constraints under mixture ambiguity sets.
\newblock \emph{Mathematical Programming}, 185\penalty0 (1):\penalty0 409--453, 2021.

\bibitem[Lejeune and Margot(2016)]{lejeune2016solving}
Miguel~A Lejeune and Fran{\c{c}}ois Margot.
\newblock Solving chance-constrained optimization problems with stochastic quadratic inequalities.
\newblock \emph{Operations Research}, 64\penalty0 (4):\penalty0 939--957, 2016.

\bibitem[Liu et~al.(2022)Liu, Lisser, and Chen]{liu2022distributionally}
Jia Liu, Abdel Lisser, and Zhiping Chen.
\newblock Distributionally robust chance constrained geometric optimization.
\newblock \emph{Mathematics of Operations Research}, 47\penalty0 (4):\penalty0 2950--2988, 2022.

\bibitem[Luedtke and Ahmed(2008)]{luedtke2008sample}
James Luedtke and Shabbir Ahmed.
\newblock A sample approximation approach for optimization with probabilistic constraints.
\newblock \emph{SIAM Journal on Optimization}, 19\penalty0 (2):\penalty0 674--699, 2008.

\bibitem[Luedtke et~al.(2010)Luedtke, Ahmed, and Nemhauser]{luedtke2010integer}
James Luedtke, Shabbir Ahmed, and George~L Nemhauser.
\newblock An integer programming approach for linear programs with probabilistic constraints.
\newblock \emph{Mathematical Programming}, 122\penalty0 (2):\penalty0 247--272, 2010.

\bibitem[Lukashevich et~al.(2023)Lukashevich, Gorchakov, Vorobev, Deka, and Maximov]{lukashevich2023importance}
Aleksander Lukashevich, Vyacheslav Gorchakov, Petr Vorobev, Deepjyoti Deka, and Yury Maximov.
\newblock {Importance Sampling Approach to Chance-Constrained DC Optimal Power Flow}.
\newblock \emph{IEEE Transactions on Control of Network Systems}, 11\penalty0 (2):\penalty0 928--937, 2023.

\bibitem[Malevergne and Sornette(2006)]{malevergne2006extreme}
Yannick Malevergne and Didier Sornette.
\newblock \emph{Extreme financial risks: From dependence to risk management}.
\newblock Springer, 2006.

\bibitem[Milligan(2011)]{milligan2011methods}
Michael Milligan.
\newblock Methods to model and calculate capacity contributions of variable generation for resource adequacy planning (ivgtf1-2): Additional discussion (presentation).
\newblock Technical report, National Renewable Energy Lab.(NREL), Golden, CO (United States), 2011.

\bibitem[Nemirovski and Shapiro(2006)]{nemirovski2006scenario}
Arkadi Nemirovski and Alexander Shapiro.
\newblock Scenario approximations of chance constraints.
\newblock In Giuseppe Calafiore and Fabrizio Dabbene, editors, \emph{Probabilistic and Randomized Methods for Design Under Uncertainty}, pages 3--47. Springer, 2006.

\bibitem[Nemirovski and Shapiro(2007)]{nemirovski2007convex}
Arkadi Nemirovski and Alexander Shapiro.
\newblock Convex approximations of chance constrained programs.
\newblock \emph{SIAM Journal on Optimization}, 17\penalty0 (4):\penalty0 969--996, 2007.

\bibitem[Pagnoncelli et~al.(2009)Pagnoncelli, Ahmed, and Shapiro]{pagnoncelli2009sample}
Bernardo~K Pagnoncelli, Shabbir Ahmed, and Alexander Shapiro.
\newblock Sample average approximation method for chance constrained programming: theory and applications.
\newblock \emph{Journal of Optimization Theory and Applications}, 142\penalty0 (2):\penalty0 399--416, 2009.

\bibitem[Pe{\~n}a-Ordieres et~al.(2020)Pe{\~n}a-Ordieres, Luedtke, and W{\"a}chter]{pena2020solving}
Alejandra Pe{\~n}a-Ordieres, James~R Luedtke, and Andreas W{\"a}chter.
\newblock Solving chance-constrained problems via a smooth sample-based nonlinear approximation.
\newblock \emph{SIAM Journal on Optimization}, 30\penalty0 (3):\penalty0 2221--2250, 2020.

\bibitem[Qiu and Wang(2014)]{qiu2014chance}
Feng Qiu and Jianhui Wang.
\newblock Chance-constrained transmission switching with guaranteed wind power utilization.
\newblock \emph{IEEE Transactions on Power Systems}, 30\penalty0 (3):\penalty0 1270--1278, 2014.

\bibitem[Resnick(2007)]{resnick2007heavy}
Sidney~I Resnick.
\newblock \emph{Heavy-tail phenomena: probabilistic and statistical modeling}, volume~10.
\newblock Springer Science \& Business Media, 2007.

\bibitem[Rockafellar and Wets(2009)]{rockafellar2009variational}
R~Tyrrell Rockafellar and Roger J-B Wets.
\newblock \emph{Variational analysis}, volume 317.
\newblock Springer Science \& Business Media, 2009.

\bibitem[Romao et~al.(2022)Romao, Papachristodoulou, and Margellos]{romao2022exact}
Licio Romao, Antonis Papachristodoulou, and Kostas Margellos.
\newblock On the exact feasibility of convex scenario programs with discarded constraints.
\newblock \emph{IEEE Transactions on Automatic Control}, 68\penalty0 (4):\penalty0 1986--2001, 2022.

\bibitem[Rubinstein(1997)]{rubinstein1997optimization}
Reuven~Y Rubinstein.
\newblock Optimization of computer simulation models with rare events.
\newblock \emph{European Journal of Operational Research}, 99\penalty0 (1):\penalty0 89--112, 1997.

\bibitem[Rubinstein(2002)]{rubinstein2002cross}
Reuven~Y Rubinstein.
\newblock Cross-entropy and rare events for maximal cut and partition problems.
\newblock \emph{ACM Transactions on Modeling and Computer Simulation (TOMACS)}, 12\penalty0 (1):\penalty0 27--53, 2002.

\bibitem[Schildbach et~al.(2014)Schildbach, Fagiano, Frei, and Morari]{schildbach2014scenario}
Georg Schildbach, Lorenzo Fagiano, Christoph Frei, and Manfred Morari.
\newblock The scenario approach for stochastic model predictive control with bounds on closed-loop constraint violations.
\newblock \emph{Automatica}, 50\penalty0 (12):\penalty0 3009--3018, 2014.

\bibitem[Shapiro et~al.(2021)Shapiro, Dentcheva, and Ruszczynski]{shapiro2021lectures}
Alexander Shapiro, Darinka Dentcheva, and Andrzej Ruszczynski.
\newblock \emph{Lectures on stochastic programming: modeling and theory}.
\newblock SIAM, 2021.

\bibitem[Subramanyam(2023)]{subramanyam2023chance}
Anirudh Subramanyam.
\newblock {Chance-Constrained Programming: Rare Events}.
\newblock In Panos~M Pardalos and Oleg~A Prokopyev, editors, \emph{Encyclopedia of Optimization}. Springer International Publishing, 2023.

\bibitem[Tempo et~al.(2005)Tempo, Dabbene, and Calafiore]{tempo2005randomized}
Roberto Tempo, Fabrizio Dabbene, and Giuseppe Calafiore.
\newblock \emph{Randomized algorithms for analysis and control of uncertain systems}.
\newblock Springer, 2005.

\bibitem[Tong et~al.(2022)Tong, Subramanyam, and Rao]{tong2022optimization}
Shanyin Tong, Anirudh Subramanyam, and Vishwas Rao.
\newblock Optimization under rare chance constraints.
\newblock \emph{SIAM Journal on Optimization}, 32\penalty0 (2):\penalty0 930--958, 2022.

\bibitem[Valk(2016)]{de2016approximation}
Cees~de Valk.
\newblock Approximation of high quantiles from intermediate quantiles.
\newblock \emph{Extremes}, 19:\penalty0 661--686, 2016.

\bibitem[Vidyasagar(2003)]{Vidyasagar2003}
M.~Vidyasagar.
\newblock \emph{Learning and Generalisation: With Applications to Neural Networks}.
\newblock Communications and Control Engineering. Springer, 2 edition, 2003.
\newblock ISBN 9781849968676; 1849968675; 9781447137481; 1447137485.

\bibitem[Wang et~al.(2018)Wang, Fan, and Pardalos]{wang2018robust}
Ximing Wang, Neng Fan, and Panos~M Pardalos.
\newblock Robust chance-constrained support vector machines with second-order moment information.
\newblock \emph{Annals of Operations Research}, 263:\penalty0 45--68, 2018.

\bibitem[Wang et~al.(2016)Wang, Gao, and Liu]{wang2016multi}
Zhiyue Wang, Deli Gao, and Jianjun Liu.
\newblock Multi-objective sidetracking horizontal well trajectory optimization in cluster wells based on ds algorithm.
\newblock \emph{Journal of Petroleum Science and Engineering}, 147:\penalty0 771--778, 2016.

\bibitem[Xie and Ahmed(2020)]{xie2020bicriteria}
Weijun Xie and Shabbir Ahmed.
\newblock Bicriteria approximation of chance-constrained covering problems.
\newblock \emph{Operations Research}, 68\penalty0 (2):\penalty0 516--533, 2020.

\bibitem[Zhang et~al.(2023)Zhang, Gao, and Luedtke]{zhang2023new}
Zeyang Zhang, Chuanhou Gao, and James Luedtke.
\newblock New valid inequalities and formulations for the static joint chance-constrained lot-sizing problem.
\newblock \emph{Mathematical Programming}, 199\penalty0 (1):\penalty0 639--669, 2023.

\bibitem[Zimper(2014)]{zimper2014minimal}
Alexander Zimper.
\newblock The minimal confidence levels of basel capital regulation.
\newblock \emph{Journal of Banking Regulation}, 15\penalty0 (2):\penalty0 129--143, 2014.

\end{thebibliography}

\end{document}